\documentclass[reqno,dvipsnames]{amsart}
\usepackage[
	left=3.5cm,
	right=3.5cm
]{geometry}

\usepackage{amsmath}
\usepackage{amsfonts}
\usepackage{amssymb}
\usepackage{graphicx}
\usepackage{color}
\usepackage{tikz}
\usepackage{tikz-cd}
\usepackage{tikz-layers}
\usepackage{caption}
\usepackage{float}
\usepackage{appendix}
\usetikzlibrary{backgrounds}
\usetikzlibrary{calc}
\usetikzlibrary{decorations.pathmorphing}
\usepackage{mathtools}
\usetikzlibrary{graphs,graphs.standard,quotes}
\usepackage{url}
\usepackage[hidelinks]{hyperref}

\usepackage{csquotes}
\usepackage[
    backend = biber,        
    style = numeric,        % numeric citation style
    sorting = nyt,          % sort by name-year-title
    doi = true,             
    url = true,              
    giveninits = true,      % show only initials of first names
    maxnames = 4            % maximal number of authors shown per paper
]{biblatex}
\usepackage{amsthm, thmtools}
\usepackage{verbatim} 
\usepackage[hidelinks]{hyperref}

\newtheorem{theorem}{Theorem}[section]
\theoremstyle{plain}

\newtheorem{corollary}[theorem]{Corollary}

\newtheorem{lemma}[theorem]{Lemma}

\newtheorem{proposition}[theorem]{Proposition}

\numberwithin{equation}{section}
\theoremstyle{definition}
\newtheorem{definition}[theorem]{Definition}
\newtheorem{example}[theorem]{Example}
\newtheorem{remark}[theorem]{Remark}

\usepackage{enumitem}
\setlist[itemize,1]{label=\textbullet}   % erste Ebene: Bullet
\setlist[itemize,2]{label=\textbullet}	 % zweite Ebene: Bullet

\setlist[itemize]{
	itemsep=3pt,
	topsep=6pt,
	leftmargin=25pt
}
\makeatletter
\renewcommand{\l@subsection}{\@tocline{2}{2pt}{3em}{3.5em}{}}
\makeatother

\DeclareMathOperator{\Act}{Act}
\DeclareMathOperator{\degree}{deg}
\DeclareMathOperator{\Image}{Im}
\DeclareMathOperator{\Int}{Int}
\DeclareMathOperator{\Sp}{Sp}
\DeclareMathOperator{\Aut}{Aut}

\begin{document}
	
	\title{Network Realignment Complexes over General Graphs}

	% Authors
	\author{Fabienne Klatt}
	\author{Iason Papadopoulos}
	\author{Marc Raffelsiefen}

	% Shared affiliation
	\address[F. Klatt, I. Papadopoulos, and M. Raffelsiefen]{\newline
	Department of Mathematics\\
	University of Bremen\\
	Bibliothekstraße 5\\
	28359 Bremen\\
	Germany
	}

	% Email addresses
	\email[F. Klatt]{faklatt@uni-bremen.de}
	\email[I. Papadopoulos]{iason@uni-bremen.de}
	\email[M. Raffelsiefen]{mraffels@uni-bremen.de}

	\begin{abstract}
	Network realignment complexes were introduced by Kozlov \cite{Kozlov25}. We generalise their definition to arbitrary connected base graphs.

	For a connected graph $G$, we characterise the connected components of the associated network realignment complex $X_G$ and show that $X_G$ admits an $\Aut(G)$-equivariant strong deformation retraction onto the disjoint union of a complete graph and a discrete $\Aut(G)$-space.	For the complete base graph $K_n$, we study the metric structure of the network realignment graph $\mathcal{G}_n$ and obtain explicit upper and lower bounds for its diameter. Finally, we prove that $X_n$ is a cubical flag complex and that every automorphism of $X_n$ is induced by a relabelling of the underlying vertex set. In particular, $\Aut(X_n)\cong S_n$ for all $n\geq 5$.
	\end{abstract}

	\maketitle

\vspace*{-5mm}
\begin{center}
	\normalfont\scshape Introduction
\end{center}\vspace*{-2mm} 

Reconfiguration problems study how one configuration can be transformed into another through a sequence of small changes while maintaining validity at every step. This is modelled by a reconfiguration graph, whose vertices represent configurations and whose edges correspond to elementary transitions, called reconfiguration steps, between them \cite{Ito11}.

A fundamental question in reconfiguration concerns connectivity \cite{Nishimura18}. Can any configuration be transformed into any other by a sequence of elementary transitions? If not, which configurations can be reached from a given starting configuration? Equivalently, how can the connected components of the reconfiguration graph be characterised? 

This naturally leads to a second fundamental question: given two configurations in the same connected component, what is the length of a shortest sequence of elementary transitions between them? Further, can the distance between any two connected realignments be bounded? In the case where the reconfiguration graph is connected, this corresponds to determining its diameter \cite{Reconfiguration}.

In this paper, the configurations are spanning trees of a fixed base graph. The reconfiguration step is given by the \textit{leaf slide}, that is a leaf is detached from its parent and reconnected to a vertex adjacent to that parent. Two spanning trees are adjacent in the reconfiguration graph if they differ by a leaf slide. The resulting reconfiguration graph is called \textit{network realignment graph}.

A natural extension of the reconfiguration framework is to allow multiple reconfiguration steps to be performed simultaneously, provided that they are independent of each other (or, more suggestively, the corresponding operations commute) \cite{Ghrist07}. This idea results in a cubical complex whose 1-skeleton is the reconfiguration graph. In the setting of network realignments, this introduces higher-dimensional cubes to the network realignment graph, leading to the \textit{network realignment complex}.

This viewpoint naturally shifts the focus from connectivity of the reconfiguration graph to higher connectivity, i.e. the topology of the associated network realignment complex. In particular, one may ask whether any two sequences of leaf slides connecting the same pair of spanning trees can themselves be transformed into one another.

\bigskip

\noindent\textbf{Related Work.} Ito et al. \cite{Ito11} gave a systematic complexity-theoretic treatment of reconfiguration problems. Nishimura \cite{Nishimura18} surveys both the structural and algorithmic aspects of reconfiguration, including reachability, connectivity, shortest-transformation problems, and the diameter of reconfiguration graphs. These works provide the general framework in which the network realignment graph may be viewed, although their results do not directly address the particular transformation of spanning trees considered here.

More closely related are reconfiguration graphs whose states are the spanning trees of a fixed connected graph $G$. The tree graph $\mathcal{T}(G)$ has one vertex for each spanning tree of $G$, with two spanning trees adjacent whenever one is obtained from the other by deleting one edge and inserting another. Cummins \cite{Cummins66} studied Hamiltonian cycles in tree graphs. Their vertex-connectivity was subsequently investigated by Liu \cite{Liu88}, while Estivill-Castro, Noy, and Urrutia \cite{Estivill00} studied the chromatic numbers of tree graphs and of a restricted variant.

Several restrictions of the ordinary edge-exchange move have also been considered. In the adjacency tree graph, the deleted and inserted edges are additionally required to be adjacent in $G$. Liu \cite{Liu87Adjacent} studied the vertex-connectivity of adjacency tree graphs. An even more restrictive transformation gives the leaf-exchange spanning tree graph: the deleted and inserted edges must be incident with the same vertex, which is required to be a leaf in both spanning trees. Broersma and Li \cite{Broersma96} characterised the connected graphs whose leaf-exchange spanning tree graph is connected and studied its connectivity. Kaneko and Yoshimoto \cite{Kaneko99} continued this investigation for $2$-connected base graphs.

Every leaf slide considered in the present paper is a leaf-exchange, but the converse need not hold. In addition to preserving the same leaf, a leaf slide requires the old and new parents of that leaf to be joined by an edge of the current spanning tree. Thus, for a fixed connected base graph $G$, these graphs form the following chain of inclusions.
\[
	\text{tree graph}\supseteq\text{adjacency tree graph} \supseteq \text{leaf-exchange graph} \supseteq \text{network realignment graph}
\] 
The tree graph, adjacency tree graph, leaf-exchange graph, and related constructions are reviewed by Ozeki and Yamashita \cite{OzekiYamashita11}.

A related line of research considers graphs with a fixed number of vertices and edges, where a local move replaces one edge by another through an edge move, edge rotation, or edge slide. Goddard and Swart \cite{Goddard96} and Jarrett \cite{Jarrett97} study the metrics on the resulting reconfiguration graphs.

Another approach considers topological spaces constructed from graphs. Lovász \cite{Lovasz78} introduced the neighbourhood complex of a graph and used its connectivity to obtain lower bounds for the chromatic number. Babson and Kozlov \cite{Babson07} studied the more general complexes $\operatorname{Hom}(H,G)$, whose cells are graph multihomomorphisms from $H$ to $G$. Jonsson \cite{Jonsson} surveys simplicial complexes arising from graph families closed under edge deletion.

The construction most closely related to network realignment complexes is the state complex of a reconfigurable system. Ghrist and Peterson \cite{Ghrist07} associate a cubical complex to a system governed by local moves, where vertices represent states, edges represent individual moves, and higher-dimensional cubes encode collections of independent local moves. They study structural properties of these state complexes and spaces of reconfiguration paths using tools from CAT$(0)$ geometry. For a later treatment of state complexes as a class of cube complexes, see Peterson~\cite{Peterson15}.

The direct predecessor of the present work is Kozlov's network realignment complex \cite{Kozlov25}, which corresponds to the case of the complete base graph $K_n$. Using discrete Morse theory, Kozlov proves that the resulting complex $X_n$ is homotopy equivalent to the complete graph $K_n$.

\bigskip

\noindent\textbf{Our Contribution.} The main objective of this work is to investigate the topology and geometry of the network realignment graph and the associated cubical complex, with particular emphasis on the symmetries induced by automorphisms of the underlying base graph. Our results address these aspects from several complementary perspectives.

We begin with the topology of the network realignment complex $X_G$ associated with an arbitrary connected base graph $G$. In contrast to the complete graph case, the resulting reconfiguration graph need not be connected. We completely characterise the connected components of $X_G$, and show that it admits an $\Aut(G)$-equivariant strong deformation retraction onto the disjoint union of a complete graph and a discrete $\Aut(G)$-space, where the vertices of the complete graph are indexed by those spanning trees of $G$ that are star trees.

Next, we investigate the geometry of the network realignment graph. We restrict to the complete graph $K_n$ and study the metric structure of $\mathcal{G}_n$. In particular, we establish general upper and lower bounds for its diameter, partially answering an open problem posed in \cite{Kozlov25}. Our analysis also reveals several structural properties of the metric, including the role of star trees as central configurations and a characterisation of extremal trees. We also show that $\mathcal{G}_n$ is bipartite if and only if $n$ is even.

Finally, we determine the automorphism group of the network realignment complex $X_n$. As an intermediate result, we prove that $X_n$ is a cubical flag complex, i.e. is completely determined by its 1-skeleton, the network realignment graph $\mathcal{G}_n$. This allows us to show that every automorphism of $X_n$ is induced by a permutation of the underlying vertex set, and hence
\[
\Aut(X_n)\cong S_n
\]
for all $n\geq 5$.

The paper is organised as follows. In \autoref{sec:NRDef}, we introduce the network realignment graph and the associated cubical complex. \autoref{TopoChapter} is devoted to the topology of the network realignment complex over an arbitrary base graph. After introducing the necessary tools from equivariant generalised Morse theory, we characterise the connected components of $X_G$ and determine their $\Aut(G)$-equivariant homotopy types. We conclude the section with several illustrative examples. \autoref{sec:Geometry} investigates the geometry of the network realignment graph associated with the complete graph, establishing bounds on its diameter together with further structural properties of the induced metric. Finally, in \autoref{sec:automorphisms}, we determine the automorphism group of the network realignment complex $X_n$.

\bigskip

\noindent\textbf{Methodology.} The proofs combine techniques from graph theory, discrete Morse theory, and combinatorial optimisation. The topological results rely on the equivariant generalised discrete Morse theory of Freij \cite{Freij}, building on Forman's discrete Morse theory \cite{Forman}. A key technical ingredient is an extension of Freij's equivariant generalised discrete Morse theory to the present setting, developed in \autoref{sec:equivTheory}, allowing us to construct $\Aut(G)$-equivariant strong deformation retractions of the network realignment complex.

\bigskip

\noindent\textbf{Acknowledgement.} We are grateful to Professor Feichtner-Kozlov for proposing this research topic and for many insightful comments and valuable suggestions provided throughout this project.

\newpage

\tableofcontents

\section{Network Realignment Complexes over General Base Graphs}\label{sec:NRDef}

In this section, we introduce the network realignment complex over general base graphs. Before presenting the definition, we fix some basic graph-theoretic notation.

Let $G$ be a finite simple undirected connected graph on $n\geq 3$ vertices. Then $V(G)$ denotes the set of vertices and $E(G)\subseteq \{(v,w)\mid v,w\in V(G)\}$ the set of edges. In the following, we will refer to $G$ as the \textit{base graph}. Let $G\setminus \{v\}$ denote the graph obtained from $G$ by deleting the vertex $v$ and all incident edges, and let $G\setminus \{e\}$ denote the graph obtained from $G$ by deleting the edge $e$. Given two subgraphs $H_1,H_2$, write $H_1\cup H_2 = (V(H_1) \cup V(H_2), E(H_1)\cup E(H_2))$. Let $K_n$ denote the \textit{complete graph} on $n$ vertices, and let $s_i$ denote the \textit{star tree} with unique interior vertex $i$ and all other vertices adjacent to $i$. For a vertex $v$, let $\mathcal{N}_G(v)$ denote the \textit{neighbourhood} of $v$, given by the set of adjacent vertices. By $P=(p_1,\dots,p_k)$ we denote a \textit{path} of length $k-1$ in $G$, where $p_1,\dots,p_k$ are pairwise distinct but adjacent vertices of $G$. A tree consisting of a path is called a \textit{path tree}. For further graph theoretic basics, we refer to \cite{Diestel}.

We now introduce the cells of the network realignment complex. A 1-dimensional network realignment corresponds to a single leaf slide, whereas a $d$-dimensional network realignment is defined by a collection of $d$ parallel leaf slides. Note that the concepts formalised in this section generalise the definitions provided in \cite{Kozlov25} where the base graph is always given by the complete graph $G=K_n$. 

\begin{definition}
	A \textit{network realignment} $N = (T, r)$ over $G$ consists of a partition of the vertices $V(G) = A\cup B$ with $|A| \geq 2$ such that
	\begin{enumerate}[label=\roman*), itemsep=3pt, topsep=6pt]
		\item $T$ is a subtree of $G$ with vertex set $A$,
		\item $r\colon B\to E(T)$ is a map satisfying $(b,v),(b,w) \in E(G)$ for all $b\in B$ with $r(b) = (v,w)$.
	\end{enumerate}
	We call the vertices of $T$ the \textit{specified} vertices of $N$, and the elements of $B$ the \textit{unspecified} vertices of $N$.
\end{definition}

As an example, consider the complete graph on eight vertices with three edges erased, $G=K_8\setminus \{(1,2),(2,3),(4,5)\}$, as base graph and the network realignment in \autoref{fig:NetworkRealignment} (a), denoted by $N_1=(T_1,r_1)$. The figure depicts the vertex 3 in $B_1$ as connected to the edge $r_1(3)=(7,5)$ by a grey triangle. We notice that $T_1$ is a subtree of $G$ and that the edges $(3,7),(3,5)$ are in $E(G)$, hence $N_1$ is a network realignment over the base graph $G$. 

Next, we define the realignment step which is the operation leading from one network realignment to another:

\begin{figure}[tbp]
	\centering
	\begin{tikzpicture}[dot/.style={circle, draw, fill, inner sep=0pt, minimum size=4pt}]

		\newcommand{\diagramScale}{0.7}

		\begin{scope}[scale=\diagramScale]
			% vertices
			\node[dot, label={left:{\scriptsize $1$}}] (u1) at (2,1) {};
			\node[dot, label={below:{\scriptsize $2$}}] (u2) at (1,0) {};
			\node[dot, label={right:{\scriptsize $3$}}] (u3) at (2.5,1) {};
			\node[dot, label={left:{\scriptsize $4$}}] (u4) at (4,1) {};
			\node[dot, label={below:{\scriptsize $5$}}] (u5) at (3,0) {};
			\node[dot, label={below:{\scriptsize $6$}}] (u6) at (0,0) {};
			\node[dot, label={below:{\scriptsize $7$}}] (u7) at (2,0) {};
			\node[dot, label={below:{\scriptsize $8$}}] (u8) at (4,0) {};
			\node () at (2,-1.2) {(a)};
			% arrows
			\draw (u6) -- (u2) -- (u7) -- (u5) -- (u8);
			\draw (u1) -- (u7);
			\draw (u4) -- (u8);
			\draw[dashed] (u7) -- (u3) -- (u5);
		\end{scope}

		\begin{scope}[scale=\diagramScale, xshift = 6cm]
			% vertices
			\node[dot, label={left:{\scriptsize $1$}}] (v1) at (2,1) {};
			\node[dot, label={below:{\scriptsize $2$}}] (v2) at (1,0) {};
			\node[dot, label={left:{\scriptsize $3$}}] (v3) at (3,1) {};
			\node[dot, label={left:{\scriptsize $4$}}] (v4) at (4,1) {};
			\node[dot, label={below:{\scriptsize $5$}}] (v5) at (3,0) {};
			\node[dot, label={below:{\scriptsize $6$}}] (v6) at (0,0) {};
			\node[dot, label={below:{\scriptsize $7$}}] (v7) at (2,0) {};
			\node[dot, label={below:{\scriptsize $8$}}] (v8) at (4,0) {};
			\node () at (2,-1.2) {(b)};
			% arrows
			\draw (v6) -- (v2) -- (v7) -- (v5) -- (v8);
			\draw (v1) -- (v7);
			\draw (v4) -- (v8);
			\draw (v3) -- (v5);
		\end{scope}

		\begin{scope}[scale=\diagramScale, xshift=12cm]
			% vertices
			\node[dot, label={left:{\scriptsize $1$}}] (w1) at (1.5,1) {};
			\node[dot, label={below:{\scriptsize $2$}}] (w2) at (1,0) {};
			\node[dot, label={left:{\scriptsize $3$}}] (w3) at (2.5,1) {};
			\node[dot, label={right:{\scriptsize $4$}}] (w4) at (3,1) {};
			\node[dot, label={below:{\scriptsize $5$}}] (w5) at (3,0) {};
			\node[dot, label={below:{\scriptsize $6$}}] (w6) at (0,0) {};
			\node[dot, label={below:{\scriptsize $7$}}] (w7) at (2,0) {};
			\node[dot, label={below:{\scriptsize $8$}}] (w8) at (4,0) {};
			\node () at (2,-1.2) {(c)};
			% arrows
			\draw (w6) -- (w2) -- (w7) -- (w5) -- (w8);
			\draw (w4) -- (w5);
			\draw[dashed] (w2) -- (w1) -- (w7) -- (w3) -- (w5);
		\end{scope}

		\begin{scope}[on behind layer]
            \fill[lightgray] (u7.center) -- (u3.center) -- (u5.center) -- cycle;
            \fill[lightgray] (w2.center) -- (w1.center) -- (w7.center) -- cycle;
			\fill[lightgray] (w7.center) -- (w3.center) -- (w5.center) -- cycle;
        \end{scope}
	\end{tikzpicture}
	\vspace{-3mm}
	\caption{(a) A network realignment $N_1$ over $G$, (b) a realignment specification $N_1^{3\to 5}$, (c) A non-example of a network realignment over $G$}
	\label{fig:NetworkRealignment}
\end{figure}
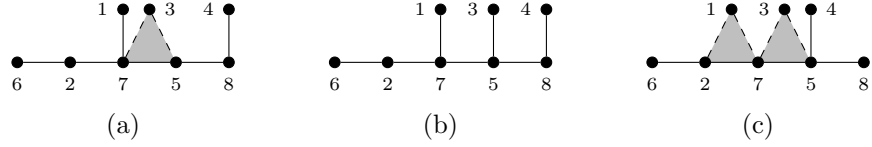

\begin{definition}
	Let $N=(T,r)$ be a network realignment over $G$, $b\in B$ and $v\in r(b)$. We define a new network realignment, which we denote by $N^{b\to v} = (T^{b\to v},r^{b\to v})$, called \textit{realignment specification}, by the following data:
	\begin{itemize}
		\item $A^{b\to v}= A\cup \{b\}, \> B^{b\to v}= B\setminus \{b\}$,
		\item $T^{b\to v}= (A^{b\to v}, E(T)\cup (v,b))$,
		\item $r^{b\to v}=r\vert_{B^{b\to v}}\colon B^{b\to v} \to E(T^{b\to v})$.
	\end{itemize}
\end{definition}

\begin{definition}
	Let $N=(T,r)$ be a network realignment over $G$. Let $b$ be a leaf of $T$ and $e = (p_T(b), v)\in E(T)$ with $v\neq b$. Assume $(b,p_T(b))\notin \Image(r)$ and $(v,b) \in E(G)$. We define a new network realignment, which we denote by $N^{b\to e} = (T^{b\to e},r^{b\to e})$, called \textit{realignment cospecification}, by the following data:
	\begin{itemize}
		\item $A^{b\to e}= A\setminus \{b\}, \> B^{b\to e}= B\cup \{b\}$,
		\item $T^{b\to e}= (A^{b\to e}, E(T)\setminus \{(b, p_T(b))\})$,
		\item $r^{b\to e}\colon B^{b\to e} \to E(T^{b\to e}), \> r^{b\to e}\vert _B = r, \> r(b) = e$.
	\end{itemize}
\end{definition}

\begin{remark}\label{rem:notation}
	We fix the following notation for $N=(T,r)$, which will be used throughout whenever the context is clear. Unless stated otherwise, $u,v,w$ denote vertices of $G$, and $e$ denotes an edge of $G$. Furthermore, $(u,w)$ denotes an edge of the spine $\Sp(N)$. If $v\in B$, then $r(v)=(u,w)$. Otherwise, $v\in T\setminus \Sp(N)$. Further, $b$ will always be an unspecified vertex.
\end{remark}

In \autoref{fig:NetworkRealignment} (b) we see a realignment specification of $N_1$ in (a), where the node $3$ of $B_1$ is specified to $5$ in $r_1(3)$. Let us denote $N_2=N_1^{3\to 5}$, then in the same manner $N_1$ is a realignment cospecification of $N_2$, namely $N_1=N_2^{3\to (7,5)}$.

In (c) we see a network realignment $N_3 = (T_3, r_3)$ over $K_8$ that is not a network realignment over $G$, since $(4,5)\in E(T_3)$ but not in $E(G)$ and $2\in r_3(1)$ but $(1,2)\notin E(G)$.

\begin{definition}
	We define $P_G$, called the \textit{poset of network realignments} over $G$, by the following:
	\begin{itemize}
		\item the elements of $P_G$ are all network realignments over $G$,
		\item the partial order relation is generated by the covering relation $N \succ N^{b\to v}$ for all network realignments $N$ and legal choices of vertices $b,v$.
	\end{itemize}
\end{definition}

Considering the order complex $\Delta (P)$ and taking the geometric realisation yields a simplicial complex $|\Delta (P_G)|$. As shown in \cite{Kozlov25}, the lower ideals $\Delta(P_G^{\leq N})$ for any network realignment $N$ over $G$ are barycentric subdivisions of cubes. This yields a cubical structure on $|\Delta (P_G)|$, where the cubes are indexed by the network realignments over $G$. 

\begin{definition}
	For a connected graph $G$, we let $X_G$ denote the \textit{network realignment complex}, given by the cubical structure on the geometric realisation $|\Delta (P_G)|$.
\end{definition}

	The action of the automorphism group $\Aut(G)$ of the base graph $G$ on $G$ translates to an action of $\Aut(G)$ on the set of spanning trees of $G$. This in turn induces an action of $\Aut(G)$ on the poset $P_G$. To be precise, for any $g\in \Aut(G)$ and any network realignment $N = (T, r)$, the network realignment $gN = (gT,gr)$ is given by the following data. The tree $gT$ has vertex set $gA$ and $(gv,gw)$ is an edge if and only if $(v,w)$ is an edge in $T$. The map $gr\colon gB\to E(gT)$ is given by $gr(gb) = g(r(b))$. By functoriality the action of $\Aut(G)$ on $P_G$ induces an action on $X_G$.
	
	In the following, we will not distinguish between a cubical complex and its face poset and use the same notation for both.

\begin{remark}\label{FiltrOfComplexes}
	If $G$ is a connected, spanning subgraph of $G'$ and $N$ is a network realignment over $G$, then it is also a network realignment over $G'$. Thus, $P_G$ is a full subposet of $P_{G'}$, and $X_G$ is a subcomplex of $X_{G'}$.
\end{remark}

\begin{figure}[t]%projektion NRC n=5
	\begin{tikzpicture}
		[scale = 0.75, every node/.style={circle,fill=black,scale=0.5}]
		
		%gray cube horizontal between (A) and (D)
		\node[draw = black, fill=white] (A) at (8,0) {};
		\node (B) at (5.33,0) {};
		\node (C) at (2.67,0.853) {};
		\node[draw = black, fill=white] (D) at (0,0) {};
		\node (E) at (2.67,-0.853) {};
		\node (F) at ($(B)+(0,-0.853)$) {};
		\node (G) at ($(B)+(0,0.853)$) {};
		
		\begin{pgfonlayer}{background}
			\filldraw[fill=gray!70, draw=black] (B.center) -- (C.center) -- (D.center) -- (E.center) -- cycle;
			\filldraw[fill=gray!40, draw=black] (A.center) -- (B.center) -- (E.center) -- (F.center) -- cycle;
			\filldraw[fill=gray!40, draw=black] (A.center) -- (B.center) -- (C.center) -- (G.center) -- cycle;
		\end{pgfonlayer}	
		
		%gray left cube between (D) and (D2)		
		\node[draw = black, fill=white] (D2) at (4,6.928) {};
		\node (B2) at (1.32,2.287) {};
		\coordinate (B2n) at (-2.287,1.32) {};
		\node (F2) at ($(B2)-0.319*(B2n)$) {};
		\node (E2) at ($(F2)+(B2)$) {};
		\node (G2) at ($(B2)+0.319*(B2n)$) {};
		\node (C2) at ($(G2)+(B2)$) {};
		
		\begin{pgfonlayer}{background}
			\filldraw[fill=gray!70, draw=black] (B2.center) -- (C2.center) -- (D2.center) -- (E2.center) -- cycle;
		\end{pgfonlayer}	
		\begin{pgfonlayer}{background}
			\filldraw[fill=gray!40, draw=black] (D.center) -- (B2.center) -- (C2.center) -- (G2.center) -- cycle;
		\end{pgfonlayer}
		\begin{pgfonlayer}{background}
			\filldraw[fill=gray!40, draw=black] (D.center) -- (B2.center) -- (E2.center) -- (F2.center) -- cycle;
		\end{pgfonlayer}
		
		%gray right cube between (D2) and (A)
		\node (B3) at (5.332,4.62) {};
		\coordinate (B3n) at ($0.188*(-4.62,-2.668)$);
		\node (G3) at ($(B3)-0.853*(B3n)$) {};
		\node (C3) at ($(G3)-0.5*(B3)+0.5*(A)$) {};
		\node (F3) at ($(B3)+0.853*(B3n)$) {};
		\node (E3) at ($(F3)-0.5*(B3)+0.5*(A)$) {};
		
		\begin{pgfonlayer}{background}
			\filldraw[fill=gray!70, draw=black] (B3.center) -- (C3.center) -- (A.center) -- (E3.center) -- cycle;
		\end{pgfonlayer}	
		\begin{pgfonlayer}{background}%oben
			\filldraw[fill=gray!40, draw=black] (D2.center) -- (B3.center) -- (C3.center) -- (G3.center) -- cycle;
		\end{pgfonlayer}
		\begin{pgfonlayer}{background}%unten
			\filldraw[fill=gray!40, draw=black] (D2.center) -- (B3.center) -- (E3.center) -- (F3.center) -- cycle;
		\end{pgfonlayer}
		
		%the violet split-cubes at (D)
		\node (U) at (0.216,-1.23) {};
		\node (V) at (-0.407,-2.307) {};
		\node (W) at (-0.623,-1.08) {};
		\node (X) at (-1.2465,0) {};
		\node (Y) at (-2.202,0.801) {};
		\node (Z) at (-0.955,0.801) {};
		
		\begin{pgfonlayer}{background}
			\filldraw[fill=violet!50, draw=black] (Z.center) -- (Y.center) -- (X.center) -- (D.center) -- cycle;
			\filldraw[fill=violet!50, draw=black] (W.center) -- (V.center) -- (U.center) -- (D.center) -- cycle;
		\end{pgfonlayer}
		
		%split-cubes star tree (A):
		\node (U2) at (7.784,-1.23) {};
		\node (V2) at (8.407,-2.307) {};
		\node (W2) at (8.623,-1.08) {};
		\node (X2) at (9.2465,0) {};
		\node (Y2) at (10.202,0.801) {};
		\node (Z2) at (8.955,0.801) {};
		
		\begin{pgfonlayer}{background}
			\filldraw[violet!50, draw=black] (U2.center) -- (V2.center) -- (W2.center) -- (A.center) -- cycle;
			\filldraw[violet!50, draw=black] (X2.center) -- (Y2.center) -- (Z2.center) -- (A.center) -- cycle;
		\end{pgfonlayer}
		
		%split-cubes star tree (D2):
		\node (U3) at (5.171,7.354) {};
		\node (V3) at (5.79,8.433) {};
		\node (W3) at (4.623,8.001) {};
		\node (X3) at (3.377,8.001) {};
		\node (Y3) at (2.206,8.433) {};
		\node (Z3) at (2.829,7.354) {};
		
		\begin{pgfonlayer}{background}
			\filldraw[fill=violet!50, draw=black] (U3.center) -- (V3.center) -- (W3.center) -- (D2.center) -- cycle;
			\filldraw[fill=violet!50, draw=black] (X3.center) -- (Y3.center) -- (Z3.center) -- (D2.center) -- cycle;
		\end{pgfonlayer}
		
		%Antennae
		\node (NV) at ($(V)-(0.25,0.433)$) {};
		\node (NV') at ($(V)+(0.087,-0.49)$) {};
		\node (NY) at ($(Y)+(-0.5,0)$) {};
		\node (NY') at ($(Y)+(-0.383,0.321)$) {};
		
		\node (NV2) at ($(V2)+(0.25,-0.433)$) {};
		\node (NV2') at ($(V2)-(0.087,0.49)$) {};
		\node (NY2) at ($(Y2)-(-0.5,0)$) {};
		\node (NY2') at ($(Y2)+(0.383,0.321)$) {};
		
		\node (NV3) at ($(V3)+(0.47,0.17)$) {};
		\node (NV3') at ($(V3)+(0.25,0.433)$) {};
		\node (NY3) at ($(Y3)+(-0.25,0.433)$) {};
		\node (NY3') at ($(Y3)+(-0.47,0.17)$) {};

		\draw[thick, Red] (V)--(NV);
		\draw[thick, Red] (V)--(NV');
		\draw[thick, Red] (Y)--(NY);
		\draw[thick, Red] (Y)--(NY');
		
		\draw[thick, Red] (V2)--(NV2);
		\draw[thick, Red] (V2)--(NV2');
		\draw[thick, Red] (Y2)--(NY2);
		\draw[thick, Red] (Y2)--(NY2');
		
		\draw[thick, Red] (V3)--(NV3);
		\draw[thick, Red] (V3)--(NV3');
		\draw[thick, Red] (Y3)--(NY3);
		\draw[thick, Red] (Y3)--(NY3');
		
		%Type 1 squares
		\node (XW) at ($(W)+(X)$) {};
		\node (XW2) at ($(W2)+(X2)-(8,0)$) {};
		\node (XW3) at ($(W3)+(X3)-(D2)$) {};
		
		\coordinate (MI1) at ($(C)!0.5!(F2)$) {};
		\node (I1) at ($(D)!1.5!(MI1)$) {};
		\coordinate (MI2) at ($(E3)!0.5!(G)$);
		\node (I2) at ($(A)!1.5!(MI2)$) {};
		\coordinate (MI3) at ($(E2)!0.5!(F3)$);
		\node (I3) at ($(D2)!1.5!(MI3)$) {};
		
		\node (UE) at ($(U)+0.5*(E)$) {};
		\node (ZG2) at ($(Z)+0.5*(G2)$) {};
		\node (U2F) at ($(U2)+0.5*(F)-0.5*(8,0)$) {};
		
		\node (Z2C3) at ($(Z2)+0.5*(C3)-0.5*(8,0)$) {};
		\node (U3G3) at ($(U3)+0.5*(G3)-0.5*(D2)$) {};
		\node (Z3C2) at ($(Z3)+0.5*(C2)-0.5*(D2)$) {};
		
		\begin{pgfonlayer}{background}
			\filldraw[fill=blue!20, draw=black] (XW.center) -- (W.center) -- (D.center) -- (X.center) -- cycle;
			\filldraw[fill=blue!20, draw=black] (XW2.center) -- (W2.center) -- (A.center) -- (X2.center) -- cycle;
			\filldraw[fill=blue!20, draw=black] (XW3.center) -- (W3.center) -- (D2.center) -- (X3.center) -- cycle;
			
			\filldraw[fill=blue!20, draw=black] (I1.center) -- (F2.center) -- (D.center) -- (C.center) -- cycle;
			\filldraw[fill=blue!20, draw=black] (I2.center) -- (G.center) -- (A.center) -- (E3.center) -- cycle;
			\filldraw[fill=blue!20, draw=black] (I3.center) -- (F3.center) -- (D2.center) -- (E2.center) -- cycle;
			
			\filldraw[fill=blue!20, draw=black] (UE.center) -- (U.center) -- (D.center) -- (E.center) -- cycle;
			\filldraw[fill=blue!20, draw=black] (ZG2.center) -- (Z.center) -- (D.center) -- (G2.center) -- cycle;
			\filldraw[fill=blue!20, draw=black] (U2F.center) -- (U2.center) -- (A.center) -- (F.center) -- cycle;
			
			\filldraw[fill=blue!20, draw=black] (Z2C3.center) -- (Z2.center) -- (A.center) -- (C3.center) -- cycle;
			\filldraw[fill=blue!20, draw=black] (U3G3.center) -- (U3.center) -- (D2.center) -- (G3.center) -- cycle;
			\filldraw[fill=blue!20, draw=black] (Z3C2.center) -- (Z3.center) -- (D2.center) -- (C2.center) -- cycle;
		\end{pgfonlayer}
		
	\end{tikzpicture} 
	\caption{Projection of the network realignment complex  $X_G$ with $G=K_5\setminus~\{e\}$.}
	\label{fig:NRCn=5}
\end{figure}
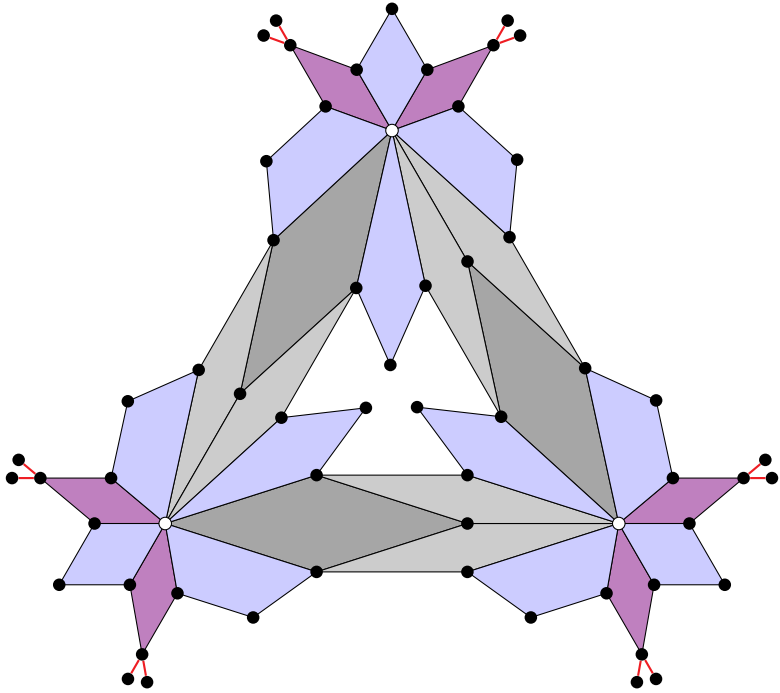

As an example of a network realignment complex $X_G$ see \autoref{fig:NRCn=5}, where the base graph is $G=K_5\setminus \{e\}$ with $e=(1,2)$. In \autoref{sec:examples}, \autoref{1EdgeM} we treat the case where the base graph is $K_n\setminus e$. The vertices of $X_G$ are indexed by the 75 spanning trees of $G$. There is only one connected component in which three 3-cubes, depicted in grey, intersect in those vertices indexed by the white star trees $s_3,s_4$ and $s_5$. The smaller violet squares represent the facets of the 3-cubes in $X_5$ between the star trees $s_i$ and $s_j$, where $i\in\{1,2\}$ and $j\in\{3,4,5\}$, viewing $X_G$ as a subcomplex of $X_5$. One can easily see from the picture that $X_G$ is homotopy equivalent to the complete graph $K_3$.
		
	\section{Equivariant Topology of the Network Realignment Complex}\label{TopoChapter}

In this section, we investigate the topology of the network realignment complex over a base graph $G$ in a way that is compatible with the symmetries of $G$. The automorphism group $\Aut(G)$ acts on $X_G$ by relabelling each network realignment, and we seek deformation retracts which do not depend on choices of particular vertex labels. This is reflected in the generalised Morse matchings below. Active vertices of a realignment are defined by structural features such as its spine and occupied leaves. Then matched partners are obtained by (co-)specifying all active vertices simultaneously, rather than choosing a preferred labelled vertex. This construction is completely independent of the labelling and yields $\Aut(G)$-equivariant strong deformation retracts.

The following theorem gives the resulting equivariant homotopy type of $X_G$. It shows that, up to an $\Aut(G)$-equivariant deformation retraction, $X_G$ has at most one non-contractible component, which is a complete graph on the vertices of $G$ that support star trees in $X_G$.

\begin{theorem} 
	Let $G$ be a connected graph on $n\geq 3$ vertices, and let $\mathcal S_G=\{v\in V(G) \mid \deg_G(v)=n-1\}$. The action of $\Aut(G)$ on $G$ restricts to an action on $\mathcal S_G$, and hence induces an action of $\Aut(G)$ on the complete graph $K_{\mathcal S_G}$ with vertex set $\mathcal S_G$. 
	
	There exists an $\Aut(G)$-equivariant embedding $K_{\mathcal S_G}\hookrightarrow X_G$ and an $\Aut(G)$-equivariant strong deformation retraction 
	\[
	X_G \simeq K_{\mathcal S_G} \sqcup D_G, 
	\]
	where $D_G$ is a discrete $\Aut(G)$-space. In particular, after forgetting the $\Aut(G)$-action, the only component of $X_G$ with possibly non-trivial homotopy is homotopy equivalent to the complete graph $K_{\mathcal S_G}$. If $k=|\mathcal S_G|\geq 2$, then this is a wedge of $(k-1)(k-2)/2$ circles.
\end{theorem}

The proof of this theorem occupies the rest of this section. After developing the necessary framework of equivariant discrete Morse theory for cubical complexes in \autoref{sec:equivTheory}, we introduce the minimal spine complex $Y_G$, an equivariant deformation retract of $X_G$, in \autoref{sec:MinimalSpineComplex}. In \autoref{sec:MainCpt} we identify a distinguished connected component of $Y_G$, the main component $M_G$.  We determine the topology of $M_G$ by showing that it $\Aut(G)$-equivariantly deformation retracts onto $K_{\mathcal S_G}$. In \autoref{sec:residual} we classify the residual connected components of $Y_G$ and show that they are contractible. Finally, in \autoref{sec:examples}, we illustrate the theory by examining the complexes for several classes of base graphs.

\subsection{Equivariant Generalised Morse Theory}\label{sec:equivTheory}

Constructing large equivariant Morse matchings is often difficult. Indeed, if there exists a group element that fixes a maximal cell $\tau$ while permuting all of its facets, then $\tau$ must be critical. In the setting of network realignments, the Morse matching given in \cite{Kozlov25} requires a choice of an active vertex by its label to determine the matched partner. Such a choice is generally incompatible with equivariance. 

Generalised Morse theory circumvents this issue by allowing a cell $\tau$ to be matched with any face $\sigma$ that is preserved whenever $\tau$ is preserved. Since we do not distinguish between the cubical complex $X$ and its poset, the interval is given by $[\sigma,\tau] = \{\alpha\in X\mid \sigma\leq \alpha\leq \tau\}$. A matched interval $[\sigma, \tau]$ then deformation retracts onto the part of $\partial \tau$ not containing $\sigma$. In the setting of network realignments, this amounts to simultaneously specifying (respectively cospecifying) all active vertices to determine the matched partner. Consequently, no active vertex needs to be chosen.

Freij introduced this equivariant generalisation of discrete Morse theory \cite{Freij}. We adapt his results for our purposes. First, simplicial complexes are replaced by cubical complexes. Further, we restrict to generalised Morse matchings whose critical cells form a subcomplex. This yields the stronger result that the subcomplex of critical cells is a $\Lambda$-equivariant deformation retract. The proof of this result is given in \autoref{AppA}.

\begin{definition}
	Let $X$ be a cubical $\Lambda$-complex, and let $\sim$ be an equivalence relation on $X$ whose equivalence classes are intervals in $X$. Then $\sim$ is called a \textit{generalised Morse matching} if there exists no cycle
	\[
	\sigma_1 < \tau_1 > \sigma_2 < \tau_2 > \ldots > \sigma_k < \tau_k > \sigma_{k+1} = \sigma_1
	\]
	with $k\geq 2$ and $\sigma_i \sim \tau_i \nsim \sigma_{i+1}$ for $1\leq i\leq k$. A cell is called \textit{critical} if it is alone in its equivalence class under $\sim$.
	
	Additionally, the matching $\sim$ is called \textit{$\Lambda$-equivariant} if, whenever $\sigma \sim \tau$, it follows that $g\sigma \sim g\tau$ for all $g\in \Lambda$.
\end{definition}

\begin{restatable}{theorem}{GeneralizedMorse} \label{Freij-generalization}
	Let $\sim$ be a $\Lambda$-equivariant generalised Morse matching on a finite cubical $\Lambda$-complex $X$. Assume that the critical cells of $\sim$ form a subcomplex $U$. Then the subcomplex $U$ is a $\Lambda$-equivariant strong deformation retract of $X$.
\end{restatable}

\subsection{Collapse onto the Minimal Spine Complex \texorpdfstring{$Y_G$}{YG}} \label{sec:MinimalSpineComplex}

The goal of this section is to construct a perfect, $\Aut(G)$-equivariant generalised Morse matching on $X_G\setminus Y_G$, where $Y_G$ denotes the subcomplex that consists of the network realignments with minimal spines. Consequently, $Y_G$ is an $\Aut(G)$-equivariant strong deformation retract of $X_G$.

The central tool for defining this matching is the spine of a network realignment. It was first defined in \cite{Kozlov25}.

\begin{definition}
	Let $N = (T,r)$ be a network realignment over $G$. The \textit{spine} of $N$, denoted by $\Sp(N)$, is the subtree of $T$ consisting of $\Int(T)$ and $\Image(r)$. The size of the spine $|\Sp(N)|$ is the number of the edges of $\Sp(N)$.
\end{definition}

Similar to the matching given in \cite{Kozlov25}, the matching we construct in this subsection is designed to reduce the size of the spine. For example, given a vertex in $X_G\setminus Y_G$, we match it to the coface obtained by cospecifying all leaves whose parent is a leaf of the spine to the corresponding leaf edge of the spine. This coface lies in $X_G$ if and only if each of the cospecifications is well-defined over the base graph $G$. 

The obstruction arises from edges missing in $G$, which prevent leaves from being slid freely along the underlying tree. For example, the vertex 2 restricts the possible slides of the leaf 1 in any network realignment over $G = K_8 \setminus \{(1,2),(2,3),(4,5)\}$. \autoref{fig:BarrierOccupied} (a) shows a network realignment over $G$. As $p_T(1) = 7$, the vertex 1 cannot be realigned to the vertex 6, since the network realignment in (b) is not a network realignment over $G$. The red dotted lines depict all of the connections that are not permitted for network realignments over $G$.

\begin{definition}
	Let $v,w \in V(G)$. Then $v$ is called a \textit{barrier} of $w$ if $(v,w) \notin E(G)$.
\end{definition}

This obstruction is particularly interesting if a leaf $v$ is attached to a leaf $u$ of the spine and cannot be slid into the interior of the spine, since $p_{\Sp(N)}(u)$ is a barrier of $v$. In this case, any sequence of realignments not involving $p_{\Sp(N)}(u)$ preserves the property that the unique path from $v$ to $p_{\Sp(N)}(u)$ passes through $u$. Therefore, $u$ remains an interior vertex of the resulting network and, in particular, stays in the spine. 

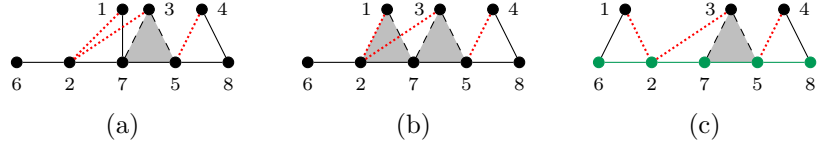
\begin{figure}[tbp]
	\centering
	\begin{tikzpicture}[dot/.style={circle, draw, fill, inner sep=0pt, minimum size=4pt}]
		
		\newcommand{\diagramScale}{0.7}
		
		\begin{scope}[scale=\diagramScale]
			% vertices
			\node[dot, label={left:{\scriptsize $1$}}] (u1) at (2,1) {};
			\node[dot, label={below:{\scriptsize $2$}}] (u2) at (1,0) {};
			\node[dot, label={right:{\scriptsize $3$}}] (u3) at (2.5,1) {};
			\node[dot, label={right:{\scriptsize $4$}}] (u4) at (3.5,1) {};
			\node[dot, label={below:{\scriptsize $5$}}] (u5) at (3,0) {};
			\node[dot, label={below:{\scriptsize $6$}}] (u6) at (0,0) {};
			\node[dot, label={below:{\scriptsize $7$}}] (u7) at (2,0) {};
			\node[dot, label={below:{\scriptsize $8$}}] (u8) at (4,0) {};
			\node () at (2,-1.2) {(a)};
			% arrows
			\draw (u6) -- (u2) -- (u7) -- (u5) -- (u8);
			\draw (u1) -- (u7);
			\draw (u4) -- (u8);
			\draw[dashed] (u7) -- (u3) -- (u5);
			\draw[red, densely dotted, thick] (u1) -- (u2) -- (u3);
			\draw[red, densely dotted, thick] (u4) -- (u5);
		\end{scope}
		
		\begin{scope}[scale=\diagramScale, xshift = 5.5cm]
			% vertices
			\node[dot, label={left:{\scriptsize $1$}}] (v1) at (1.5,1) {};
			\node[dot, label={below:{\scriptsize $2$}}] (v2) at (1,0) {};
			\node[dot, label={left:{\scriptsize $3$}}] (v3) at (2.5,1) {};
			\node[dot, label={right:{\scriptsize $4$}}] (v4) at (3.5,1) {};
			\node[dot, label={below:{\scriptsize $5$}}] (v5) at (3,0) {};
			\node[dot, label={below:{\scriptsize $6$}}] (v6) at (0,0) {};
			\node[dot, label={below:{\scriptsize $7$}}] (v7) at (2,0) {};
			\node[dot, label={below:{\scriptsize $8$}}] (v8) at (4,0) {};
			\node () at (2,-1.2) {(b)};
			% arrows
			\draw (v6) -- (v2) -- (v7) -- (v5) -- (v8);
			\draw (v4) -- (v8);
			\draw[dashed] (v1) -- (v7) -- (v3) -- (v5);
			\draw[red, densely dotted, thick] (v1) -- (v2) -- (v3);
			\draw[red, densely dotted, thick] (v4) -- (v5);
		\end{scope}
		
		\begin{scope}[scale=\diagramScale, xshift=11cm]
			% vertices
			\node[dot, label={left:{\scriptsize $1$}}] (w1) at (0.5,1) {};
			\node[dot, ForestGreen, label={below:{\scriptsize $2$}}] (w2) at (1,0) {};
			\node[dot, label={left:{\scriptsize $3$}}] (w3) at (2.5,1) {};
			\node[dot, label={right:{\scriptsize $4$}}] (w4) at (3.5,1) {};
			\node[dot, ForestGreen, label={below:{\scriptsize $5$}}] (w5) at (3,0) {};
			\node[dot, ForestGreen, label={below:{\scriptsize $6$}}] (w6) at (0,0) {};
			\node[dot, ForestGreen, label={below:{\scriptsize $7$}}] (w7) at (2,0) {};
			\node[dot, ForestGreen, label={below:{\scriptsize $8$}}] (w8) at (4,0) {};
			\node () at (2,-1.2) {(c)};
			% arrows
			\draw[ForestGreen] (w6) -- (w2) -- (w7) -- (w5) -- (w8);
			\draw (w4) -- (w8);
			\draw (w1) -- (w6);
			\draw[dashed] (w7) -- (w3) -- (w5);
			\draw[red, densely dotted, thick] (w1) -- (w2) -- (w3);
			\draw[red, densely dotted, thick] (w4) -- (w5);
		\end{scope}
		
		\begin{scope}[on behind layer]
			\fill[lightgray] (u7.center) -- (u3.center) -- (u5.center) -- cycle;
			\fill[lightgray] (v2.center) -- (v1.center) -- (v7.center) -- cycle;
			\fill[lightgray] (v7.center) -- (v3.center) -- (v5.center) -- cycle;
			\fill[lightgray] (w7.center) -- (w3.center) -- (w5.center) -- cycle;
		\end{scope}
	\end{tikzpicture}
	\vspace*{-3mm}
	\caption{Consider the base graph $G = K_8 \setminus \{(1,2),(2,3),(4,5)\}$. The missing edges of $G$ are drawn in red. (a) A network realignment over $G$, (b) a non-example of a network realignment over $G$, since 2 is a barrier of 1, and in (c), the vertices 6 and 8 are occupied by 1 and 4, respectively. Hence, the spine, drawn in green, cannot be reduced.}
	\label{fig:BarrierOccupied}
\end{figure}

\begin{definition}
	Let $N$ be a network realignment over $G$. A leaf $u$ of $\Sp(N)$ is called \textit{occupied} if there exists a vertex $v\in V(T)\setminus V(\Sp(N))$ such that $p_T(v) = u$ and $p_{\Sp(N)} (u)$ is a barrier of $v$. In this situation, we say that $v$ \textit{occupies} $u$. If no such vertex exists, we say that $u$ is \textit{unoccupied}.
\end{definition}

Given any two occupied leaves, then both must remain interior vertices, and thus the unique path between them remains in the interior. This provides a lower bound on the spine under any sequence of realignments. In \autoref{fig:BarrierOccupied} (c), there are two occupied leaves, namely 6 and 8. Every realignment sequence must keep the path between them in the spine. Moreover, if all leaves of the spine are occupied, no sequence of realignments can reduce the size of the spine.

\begin{definition}
	The spine $\Sp(N)$ of a network realignment $N=(T,r)$ is called \textit{minimal} if $|\Sp(N)|\leq 1$ or all leaves are occupied. We denote by $Y_G$ the subset of $X_G$ consisting of all network realignments with minimal spines.
\end{definition}

As we will see in \autoref{lem:MSCisSubcomplex}, the subset $Y_G$ is a subcomplex of $X_G$. To prove this we need the following lemmas. 

\begin{lemma} \label{lem:OccupiedLeaves}
	If $\Sp(N) = \Sp(N^{b\to u})$, then the sets of occupied leaves of the spines coincide.
\end{lemma}

\begin{proof}
	A vertex $v\in V(T)\setminus V(\Sp(N))$ occupies a leaf of $\Sp(N)$ if and only if it occupies the same leaf of $\Sp(N^{b\to u})$, since $T= T^{b\to u} \setminus \{b\}$. Moreover, if $u = p_{T^{b\to u}} (b)$ is a leaf of the spine, then $b$ does not occupy $u$ in $\Sp(N^{b\to u})$, because $p_{\Sp(N)}(u)\in r(b)$, and thus $(b, p_{\Sp(N)} (u))\in E(G)$, meaning that $p_{\Sp(N)}(u)$ is not a barrier of $b$. Hence, the occupied leaves coincide.
\end{proof}

\begin{lemma}\label{lem:SpineInclusion}
	Let $N = (T,r)$ be a network realignment over $G$. Then $\Sp(N^{b\to u}) \subseteq \Sp(N)$ for all legal choices $b,u$.
\end{lemma}

\begin{proof}
	Since $G\subseteq K_n$, $N$ is a network realignment over $K_n$, and the claim was proved for network realignments over $K_n$ in \cite[Proposition 4.7]{Kozlov25}.  
\end{proof}

\begin{lemma}\label{lem:MSCisSubcomplex}
	$Y_G$ is a subcomplex of $X_G$. 
\end{lemma}

\begin{proof}
	If all leaves of the spine are occupied, then $\Sp(N) = \Int(T)$. Indeed, in this situation every leaf $u$ of $\Sp(N)$ lies in $\Int(T)$, because there exists a vertex $v\in V(T)\setminus V(\Sp(N))$ with $p_T(v) = u$, implying $\degree_T(u) \geq 2$. Since $\Sp(N)$ is a subtree of $T$, all interior vertices of $\Sp(N)$ are also interior vertices of $T$. This shows the claim if $|\Sp(N)|\geq 1$. If $\Sp(N)$ is a single vertex, then $B = \emptyset$, and again $\Sp(N) = \Int(T)$.
	
	It suffices to prove that $N\in Y_G$ implies $N^{b\to u}\in Y_G$ for all legal choices of vertices $b,u$. If $|\Sp(N)|\leq 1$, then $|\Sp(N^{b\to u})|\leq 1$, since $\Sp(N^{b\to u}) \subseteq \Sp(N)$ by \autoref{lem:SpineInclusion}. If all leaves of $\Sp(N)$ are occupied, then 
	\[
	\Sp(N^{b\to u}) \subseteq \Sp(N) = \Int(T) \subseteq \Int(T^{b\to u}) \subseteq \Sp(N^{b\to u}),
	\] 
	and thus $\Sp(N^{b\to u}) = \Sp(N)$. By \autoref{lem:OccupiedLeaves}, the sets of occupied leaves coincide. Thus, all leaves of $\Sp(N^{b\to u})$ are occupied and $N^{b\to u} \in Y_G$.
\end{proof}

\begin{definition}
	The set $Y_G$ is called the \textit{minimal spine complex} of $X_G$.
\end{definition}

\begin{figure}[tbp]
	\centering
	\begin{tikzpicture}[dot/.style={circle, draw, fill, inner sep=0pt, minimum size=4pt}]
		
		\newcommand{\diagramScale}{0.7}
		
		\begin{scope}[scale=\diagramScale]
			% vertices
			\node[dot, label={left:{\scriptsize $1$}}] (u1) at (0.5,1) {};
			\node[dot, ForestGreen, label={below:{\scriptsize $2$}}] (u2) at (1,0) {};
			\node[dot, label={right:{\scriptsize $3$}}] (u3) at (2.5,1) {};
			\node[dot, label={right:{\scriptsize $4$}}] (u4) at (3.5,1) {};
			\node[dot, ForestGreen, label={below:{\scriptsize $5$}}] (u5) at (3,0) {};
			\node[dot, ForestGreen, label={below:{\scriptsize $6$}}] (u6) at (0,0) {};
			\node[dot, ForestGreen, label={below:{\scriptsize $7$}}] (u7) at (2,0) {};
			\node[dot, ForestGreen, label={below:{\scriptsize $8$}}] (u8) at (4,0) {};
			\node () at (2,-1.2) {(a)};
			% arrows
			\draw[ForestGreen] (u6) -- (u2) -- (u7) -- (u5) -- (u8);
			\draw (u1) -- (u6);
			\draw (u4) -- (u8);
			\draw[dashed] (u7) -- (u3) -- (u5);
			\draw[red, densely dotted, thick] (u1) -- (u2) -- (u3);
			\draw[red, densely dotted, thick] (u4) -- (u5);
		\end{scope}
		
		\begin{scope}[scale=\diagramScale, xshift = 6cm]
			% vertices
			\node[dot, label={left:{\scriptsize $1$}}] (v1) at (0.5,1) {};
			\node[dot, ForestGreen, label={below:{\scriptsize $2$}}] (v2) at (1,0) {};
			\node[dot, label={above:{\scriptsize $3$}}] (v3) at (2.5,1) {};
			\node[dot, ForestGreen, label={right:{\scriptsize $4$}}] (v4) at (3.5,1) {};
			\node[dot, ForestGreen, label={below:{\scriptsize $5$}}] (v5) at (3,0) {};
			\node[dot, ForestGreen, label={below:{\scriptsize $6$}}] (v6) at (0,0) {};
			\node[dot, ForestGreen, label={below:{\scriptsize $7$}}] (v7) at (2,0) {};
			\node[dot, ForestGreen, label={below:{\scriptsize $8$}}] (v8) at (4,0) {};
			\node () at (2,-1.2) {(b)};
			% arrows
			\draw[ForestGreen] (v6) -- (v2) -- (v7) -- (v5) -- (v8) -- (v4);
			\draw (v1) -- (v6);
			\draw (v3) -- (v4);
			\draw[red, densely dotted, thick] (v1) -- (v2) -- (v3);
			\draw[red, densely dotted, thick] (v4) -- (v5);
		\end{scope}
		
		\begin{scope}[on behind layer]
			\fill[lightgray] (u7.center) -- (u3.center) -- (u5.center) -- cycle;
		\end{scope}
	\end{tikzpicture}
	\vspace{-3mm}
	\caption{(a) A network realignment in the minimal spine complex $Y_G$, where $G = K_8 \setminus \{(1,2),(2,3),(4,5)\}$, (b) a network realignment not contained in $Y_G$.}
	\label{fig:MSC}
\end{figure}
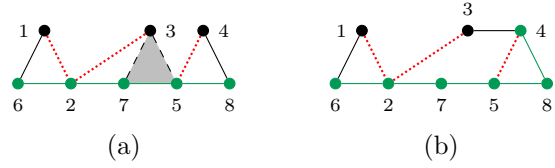

As an example, consider \autoref{fig:MSC}, where the base graph is again $G=K_8 \setminus \{(1,2),(2,3),(4,5)\}$. The spines are depicted in green. In (a), all leaves of the spine of the network realignment $N_1$, namely 6 and 8, are occupied, hence the spine is minimal and $N_1\in Y_G$. On the other hand, in (b) the leaf 4 of the spine of the network realignment $N_2$ is unoccupied. Hence, the spine is not minimal and $N_2$ is not in $Y_G$.

Moreover, $N_1$ and $N_2$ are connected in $X_G$ by a sequence of (co-)specifications of the vertex 3. In what follows, we will define an equivalence that matches realignments such as $N_2$, leaving only cells with minimal spines, such as $N_1$. We will then prove the equivalence relation to be an $\Aut(G)$-equivariant generalised Morse matching. 

We define the set of active vertices of a network realignment $N\in X_G\setminus Y_G$. The following is an adaptation of \cite[Definition 4.13]{Kozlov25} to the present restricted network realignment complex, taking into account the occupied leaves. 

\begin{definition}
	Let $N\in X_G\setminus Y_G$ be a network realignment over $G$. We define
	\begin{itemize}
		\item $\Act_Y^+(N) = \{ v\in B\mid r(v)=(u,w), \> u \text{ an unoccupied leaf of }\Sp(N)\}$,
		\item $\Act_Y^-(N) = \{ v\in A\mid v \text{ a leaf of } T, \> p_T(v) \text{ an unoccupied leaf of } \Sp(N)\}$,
	\end{itemize}
	and $\Act_Y(N)=\Act_Y^+(N)\cup \Act_Y^-(N)$, the set of \textit{active vertices}.
\end{definition}

\begin{lemma}\label{lem:ActNonempty}
	If $\Sp(N)$ is not minimal, then $\Act_Y(N) \neq \emptyset$. 
\end{lemma}

\begin{proof}
	Since $\Sp(N)$  is not minimal, $|\Sp(N)|\geq 2$ and there exists a leaf $u$ of $\Sp(N)$ that is not occupied. Set $w = p_{\Sp(N)}(u)$. If $(u,w) \in \Image(r)$, then $\Act_Y^+(N) \neq~\emptyset$. On the other hand, if $(u,w) \notin \Image(r)$, then there exists a vertex $v\in V(T)\setminus V(\Sp(N))$ with $p_T(v)=u$. Since $u$ is not occupied, $v \in \Act_Y^-(N)$. In both cases, $\Act_Y(N) \neq~\emptyset$.
\end{proof}

\begin{lemma} \label{lem:SpineEqualityIter}
	Let $N = (T,r)\in X_G\setminus Y_G$. Suppose that $N^\prime$ is obtained by
	\begin{itemize}
		\item a specification $N^\prime = N^{b\to u}$, where $b\in \Act_Y^+(N)$ and $u$ is a leaf of $\Sp(N)$, or
		\item a cospecification $N^\prime = N^{v\to e}$, where $v\in \Act_Y^-(N)$ and $e = (p_T(v), p_{\Sp(N)}(p_T(v)))$.
	\end{itemize}
	Then $\Sp(N) = \Sp(N^\prime)$ and $\Act_Y(N) = \Act_Y(N^\prime)$.
\end{lemma}

\begin{proof}
	By \autoref{FiltrOfComplexes}, $X_G$ is a subcomplex of $X_n=X_{K_n}$. Furthermore, the definition of the set of active vertices given in the present work differs from the one given in \cite{Kozlov25} only by those vertices whose parent, respectively parent of their specification, is an occupied leaf of the spine. Hence, the proof that the spines of $N$ and $N^\prime$ are identical follows from the argument given in \cite[Lemma 4.18]{Kozlov25}. Now, \autoref{lem:OccupiedLeaves} implies that the sets of occupied leaves of $\Sp(N)$ and $\Sp(N^\prime)$ coincide, and consequently, the arguments used in \cite[Lemma 4.18]{Kozlov25} prove the claim.
\end{proof} 

The previous lemma implies that (co-)specifying all of the active vertices at once will result in network realignments with the same spines and active vertices.

In order to define the partition of cells into intervals for the generalised Morse matching, we first define maps $d_Y$ and $u_Y$ on the non-critical cells such that, for each $N$, the interval in the partition containing $N$ is $[d_Y(N),u_Y(N)]$.

\begin{definition}\label{def:interval_limits}
	We define the two maps $d_Y,u_Y\colon X_G\setminus Y_G\rightarrow X_G\setminus Y_G$ as follows. For $N=(T,r)\in X_G\setminus Y_G$, let $r(v)=(u_v,w_v)$ for each $v\in\Act_Y^+(N)$, where $u_v$ is a leaf of $\Sp(N)$. Then the network realignment $d_Y(N)=(d_Y(T),d_Y(r))$ is given by the data
	\begin{itemize}
		\item $d_Y(A)=A\cup \Act_Y^+(N)$, \> $d_Y(B)=B\setminus \Act_Y^+(N)$,
		\item $d_Y(T)=T\cup \bigcup\limits_{v\in \Act_Y^+(N)} (v,u_v)$,
		\item $d_Y(r) = r\vert_{d_Y(B)}$,
	\end{itemize}
	and the network realignment $u_Y(N)=(u_Y(T),u_Y(r))$ is defined by the data	
	\begin{itemize}
		\item $u_Y(A)=A\setminus \Act_Y^-(N)$, \> $u_Y(B) = B\cup \Act_Y^-(N)$,
		\item $u_Y(T) = T \setminus \Act_Y^-(N)$,
		\item  $u_Y(r)\vert_{B}=r$ and $u_Y(r)(v)=(p_T(v),p_{\Sp(N)}(p_T(v)))$ for all $v\in \Act_Y^-(N)$.
	\end{itemize}
\end{definition}

The network realignment $d_Y(N)$ is obtained by specifying every vertex in the set $\Act_Y^+(N)$ simultaneously. Accordingly, the network realignment $u_Y(N)$ is obtained by cospecifying every vertex in the set $\Act_Y^-(N)$ simultaneously. This implies that $d_Y(N)$ is a face of $N$ and $u_Y(N)$ is a coface. In conclusion, we have $d_Y(N)\leq N\leq u_Y(N)$. We now define the equivalence relation in the following way.

\begin{definition}
	Let $\sim$ be the collection of intervals in $X_G\setminus Y_G$ of the form
	\[
	[d_Y(N),u_Y(N)]
	\]
	for $N\in X_G\setminus Y_G$.
\end{definition}

\begin{lemma}\label{lem:firstMatchingWellDef}
	The collection $\sim$ is a partition of $X_G\setminus Y_G$. In particular, for a network realignment $\tilde{N}\in X_G\setminus Y_G$, we have
	\begin{equation}\label{Eq:NInInterv}
		N\in [d_Y(\tilde{N}),u_Y(\tilde{N})] \Rightarrow d_Y(N)=d_Y(\tilde{N}) \text{ and } u_Y(N)=u_Y(\tilde{N}).
	\end{equation}
\end{lemma}

\begin{proof}
	First, we note that the intervals in $\sim$ really cover the whole set $X_G\setminus Y_G$. This follows from the existence of an interval $[d_Y(N),u_Y(N)]$ for each $N\in X_G\setminus Y_G$. Next, we turn to the implication (\ref{Eq:NInInterv}).
	
	Since $N\in  [d_Y(\tilde{N}),u_Y(\tilde{N})]$, we have 
	\begin{equation}\label{Eq:NInInterv2}
		d_Y(\tilde{N}) \leq N \leq u_Y(\tilde{N}).
	\end{equation}
	
	By definition, $u_Y(\tilde{N})$ is obtained from $d_Y(\tilde{N})$ by cospecifying every active vertex. Hence, by (\ref{Eq:NInInterv2}), $N$ must be obtained from $d_Y(\tilde{N})$ by cospecifying a subset $\mathcal{A}\subseteq \Act_Y(\tilde{N})$ to the corresponding leaf edge of the spine. By iterating \autoref{lem:SpineEqualityIter}, it follows that $\Act_Y(N)= \Act_Y(d_Y(\tilde{N}))$. This implies conversely that $d(\tilde{N})$ is obtained from $N$ by specifying every vertex in $\mathcal{A} = \Act_Y^+(N)$ to the corresponding leaf of the spine, and hence $d_Y(N)=d_Y(\tilde{N})$.
	
	On the other hand, $N$ is obtained from $u_Y(\tilde{N})$ by specifying a subset of $\Act_Y(\tilde{N})$. Similar arguments lead to $u_Y(N)=u_Y(\tilde{N})$, and hence the implication of (\ref{Eq:NInInterv}) is proved.
	
	This implies that the intervals of $\sim$ are disjoint, which means that $\sim$ is a partition.
\end{proof}

\begin{lemma} \label{lem:SpineEquality}
	Let $I$ be an interval in $\sim$ and $N,\tilde{N}\in I$. We have $\Sp(N) = \Sp(\tilde{N})$ and $\Act_Y(N) = \Act_Y(\tilde{N})$.
\end{lemma}

\begin{proof}
	As in the proof \autoref{lem:firstMatchingWellDef}, applying \autoref{lem:SpineEqualityIter} iteratively yields the claim.
\end{proof} 

\begin{lemma}\label{lem:ActiveDecreasing}
	Let $N$ be a network realignment over $G$ that is not minimal. If $\Sp(N) = \Sp(N^{b\to u})$, then $\Act_Y(N^{b\to u}) \subseteq \Act_Y(N)$.
\end{lemma}

\begin{proof}
	The following arguments are similar to the proof of \cite[Lemma 4.20]{Kozlov25}, but we have to take into account the relation to the occupied leaves in the set of potential active vertices.
	
	Let $v\in \Act_Y(N^{b\to u})$. We show that $v\in \Act_Y(N)$. If $v\in \Act_Y^+(N^{b\to u})$, then $r(v)$ is a leaf edge in $\Sp(N^{b\to u}) = \Sp(N)$. Since $v\in \Act_Y^+(N^{b\to u})$, both incident vertices to $r(v)$ are unoccupied in $\Sp(N^{b\to u})$, and thus also unoccupied in $\Sp(N)$ by \autoref{lem:OccupiedLeaves}. Hence, $v\in \Act_Y^+(N)$.
	
	Now, assume $v\in \Act_Y^-(N^{b\to u})$. If $v=b$, then $u$ is a leaf of $\Sp(N^{b\to u})$ and $r(b)$ is a leaf edge of $\Sp(N)=\Sp(N^{b\to u})$. Then $p_{T^{b\to u}}(b)=u$ is unoccupied in $N^{b\to u}$, and thus also unoccupied in $N$ by \autoref{lem:OccupiedLeaves}. Hence, $b\in \Act_Y^+(N)$.
	
	If $v\neq b$, then $v$ is a leaf of $T$, which is a subtree of $T^{b\to u}$, and $p_{T^{b\to u}}(v) = p_T(v)$ is a leaf of $\Sp(N^{b\to u}) = \Sp(N)$. Since $p_T(v)$ is not occupied in $N^{b\to u}$, and thus also unoccupied in $\Sp(N)$ by \autoref{lem:OccupiedLeaves}, we have $v\in \Act_Y^-(N)$.
\end{proof}

This leads us to the main result.

\begin{proposition}\label{equivMatchingProp}
	The equivalence relation $\sim$ is an $\Aut(G)$-equivariant generalised Morse matching on $X_G\setminus Y_G$.
\end{proposition}

\begin{proof}
	Assume for a contradiction the existence of a cycle in $X_G\setminus Y_G$ of the form
	\begin{equation} \label{equivCycle}
		\tilde{N}_1 < N_1 > \tilde{N}_2 < N_2 > \cdots < N_k > \tilde{N}_1,
	\end{equation}
	where $k\geq 2$, $\tilde{N}_i \sim N_i$, and $N_i$ and $\tilde{N}_{i+1}$ belong to distinct equivalence classes for all $1\leq i\leq k$. For convenience, we set $\tilde{N}_{k+1}=\tilde{N}_1$.
	
	By \autoref{lem:SpineInclusion} and \autoref{lem:SpineEquality}, all cells appearing in \eqref{equivCycle} have the same spine, i.e. $\Sp(\tilde{N}) = \Sp(N)$ for all $\tilde{N},N\in\{\tilde{N}_1,\dots,\tilde{N}_k,N_1,\dots,N_k\}$. Furthermore, \autoref{lem:SpineEquality} and \autoref{lem:ActiveDecreasing} imply $\Act_Y(\tilde{N}) = \Act_Y(N)$.
	
	The relation $N_i>\tilde{N}_{i+1}$ means that $\tilde{N}_{i+1}$ is obtained from $N_i$ by specifying a non-empty subset of vertices in $B$ of $N_i$, not all of which are active and specified to the leaf of the spine. Thus one of the following occurs.
	\begin{itemize}
		\item[(a)] A vertex $v\in\Act_Y^+(N_i)$ is specified such that $\tilde{N}_{i+1} \not\sim N_i$. Hence, $v$ is attached to an interior vertex of the spine. However, since $\Sp(\tilde{N}_{i+1})=\Sp(N_i)$, it follows that $v\notin \Act_Y(\tilde{N}_{i+1})$, contradicting $\Act_Y(N_i) = \Act_Y(\tilde{N}_{i+1})$.
		\item[(b)] A vertex $v\notin \Act_Y(N_i)$, $v\in B$ is specified. Since $v\notin \Act_Y(\tilde{N}_i)$, it can never be cospecified later. Consequently, the sequence $\vert B\setminus \Act_Y\vert$ decreases along the cycle, preventing a return to $\tilde{N}_1$.
	\end{itemize}
	To conclude, this contradicts the existence of the cycle \eqref{equivCycle}. Hence, $\sim$ is a generalised Morse matching.
	
	By \autoref{lem:ActNonempty}, every $N\in X_G\setminus Y_G$ satisfies $\Act_Y(N)\neq \emptyset$, and therefore $d_Y(N)\neq u_Y(N)$. Thus, no interval of $\sim$ consists of a single cell, and the critical cells are precisely those of $Y_G$. Moreover, \autoref{lem:firstMatchingWellDef} implies that the matching is well-defined.
	
	Finally, it is easy to see that any automorphism $g$ of $G$ preserves the spine: $g\Sp(N)=\Sp(gN)$. Furthermore, since the adjacency of $K_n\setminus G$ is preserved, $g$ maps occupied leaves to occupied leaves and this implies $g\Act^+(N)=\Act^+(gN)$ as well as $g\Act^-(N)=\Act^-(gN)$. Hence, $gd_Y(N) = d_Y(gN)$ and $gu_Y(N) = u_Y(gN)$. Consequently,
	\[
	[gd_Y(N), gu_Y(N)] = [d_Y(gN), u_Y(gN)]
	\] 
	is matched in $\sim$ for any $N\in X_G\setminus Y_G$. It follows that $\sim$ is $\Aut(G)$-equivariant.
\end{proof}

\begin{proposition}\label{Y_G-matching-thm}
	The minimal spine complex $Y_G$ is an $\Aut(G)$-equivariant strong deformation retract of $X_G$.
\end{proposition}

\begin{proof}
	By \autoref{equivMatchingProp}, we know that $\sim$ is an $\Aut(G)$-equivariant generalised Morse matching whose set of critical cells $Y_G$ is a subcomplex of $X_G$ by \autoref{lem:MSCisSubcomplex}. Consequently, \autoref{Freij-generalization} gives the strong deformation retraction.
\end{proof}

\subsection{The Main Component of \texorpdfstring{$Y_G$}{Y_G}}\label{sec:MainCpt}

As discussed earlier, the path between any two occupied leaves remains contained in the spine under any sequence of network realignments over $G$. Conversely, if a network realignment with minimal spine has precisely one occupied leaf, then every vertex can be slid to that occupied leaf. The resulting network realignment is a star tree of $G$. 

We split the analysis of the topology of $Y_G$ into two parts. We begin with the main component $M_G$, which contains all the star trees of $G$ and is the only connected component that is not necessarily contractible. To show that $M_G$ $\Aut(G)$-equivariantly retracts onto $K_{\mathcal{S}_G}$, the complete graph whose vertices are the star trees of $G$, we first construct a further generalised Morse matching that collapses all network realignments with a single occupied leaf to the star tree with the occupied leaf as interior vertex. The resulting subcomplex contains precisely the star trees and the $(n-2)$-cubes connecting them. $K_{\mathcal{S}_G}$ is a strong deformation retract of this subcomplex.

\begin{definition}
	Let $M_G\subseteq Y_G$ consist of all cells $N$ of $Y_G$ whose spine has at most one occupied vertex in $\Sp(N)$. We call $M_G$ the \textit{main component} of $Y_G$.
\end{definition}

Note that $N\in M_G$ implies $|\Sp(N)|\leq 1$. Let $N$ be a network realignment with $|\Sp(N)|\geq 2$. Then its spine contains at least two leaves. Thus, if $N\in Y_G$, the spine would have at least two occupied vertices, and hence $N\notin M_G$.

Associated to each star tree $s_i$ of $G$ is a vertex of $M_G$, which we will also denote by $s_i$. Clearly, the star tree with center vertex $i$ is a tree of $G$ if and only if $\deg_G(i) = n-1$. The star trees are precisely the vertices of $X_G$ whose spine has size $0$.

\begin{proposition}\label{prop:MainComponent}
	$M_G$ is a connected component of $Y_G$. Moreover, $M_G$ is preserved by the action of $\Aut(G)$.
\end{proposition}

\begin{proof}
	We start by showing that $M_G$ is a subcomplex of $Y_G$. Let $N\in M_G$. If $|\Sp(N)| = 0$, then $B = \emptyset$ and $T$ is a star tree. If $|\Sp(N)| = 1$, then either $|\Sp(N^{b\to u})| = 0$, and thus $N^{b\to u}\in M_G$, or $\Sp(N) = \Sp(N^{b\to u})$. Then, by \autoref{lem:OccupiedLeaves}, the occupied leaves of $\Sp(N)$ and $\Sp(N^{b\to u})$ coincide, and again $N^{b\to u}\in M_G$. Therefore, $M_G$ is a subcomplex of $Y_G$.
	
	To show that $M_G$ is connected, we show that any two vertices of $M_G$ are connected. Let $N_0 =(T_0, r_0)\in M_G$ be 0-dimensional with $|\Sp(N_0)| = 1$. Choose $u$ to be the unoccupied vertex of $\Sp(N_0)$, and let $w$ denote the other vertex. Let $w,v_1,\ldots,v_k$ denote the neighbours of $u$ in $T_0$. Define $N_i = (N_{i-1}^{v_i\to (u,w)})^{v_i\to w}$ for $1\leq i\leq k$, and consider the path
	\[
	N_0, N_1, \ldots, N_{k-1}, N_k
	\]
	in $Y_G$. We have 
	\[
	\Sp(N_0) = \Sp(N_0^{v_1\to (u,w)}) = \ldots = \Sp(N_{k-1}) = \Sp(N_{k-1}^{v_k\to (u,w)}),
	\]
	and $V(\Sp(N_k)) = \{w\}$. By \autoref{lem:OccupiedLeaves}, all network realignments of this sequence are contained in $M_G$. Thus, any 0-dimensional cell in $M_G$ is connected to a star tree. Clearly, any two star trees $s_i$ and $s_j$ that are contained in $Y_G$ are connected by the $(n-2)$-cubes. Therefore, $M_G$ is connected.
	
	To conclude the proof that $M_G$ is a connected component, it is enough to show that if $N\in M_G$, then $N^{v\to e}\in Y_G$ implies $N^{v\to e}\in M_G$. Given $N\in M_G$ and $N^{v\to e}\in Y_G$, we write $e = (u,w)$ with $p_T(v) = u$. If $|\Sp(N)| = 0$, then $|\Sp(N^{v\to e})| = 1$. Since $\degree_{T^{v\to e}} (w) = \degree_T(w) = 1$, it follows that $w$ is not occupied in $\Sp(N^{v\to e})$. Hence, $\Sp(N^{v\to e})$ has at most one occupied vertex and is contained in $M_G$. On the other hand, assume that $|\Sp(N)| = 1$. If $\Sp(N) = \Sp(N^{v\to e})$, then the occupied leaves coincide by \autoref{lem:OccupiedLeaves}. Therefore, $N\in M_G$ directly implies $N^{v\to e}\in M_G$. Otherwise, $|\Sp(N^{v\to e})| = 2$. We observe that $\mathcal{N}_{T^{v\to e}} (w) = \mathcal{N}_T(w) = \{u\}$, and thus $w$ is not occupied in $\Sp(N^{v\to e})$ and $N^{v\to e} \notin Y_G$.
	
	It remains to show that $gM_G = M_G$ for all $g\in \Aut(G)$. Since the $\Aut(G)$-action is induced by permuting the vertices of each network realignment, it preserves the isomorphism type of the underlying tree. In particular, star trees are mapped to star trees. As $M_G$ is a connected component containing all star trees of $G$, it follows that $gM_G = M_G$.
\end{proof}

Our goal is to retract $M_G$ onto the subcomplex $\Sigma_G$, which contains the following types of cells. Let $s_i$ and $s_j$ correspond to two star trees of $G$. Both are faces of the $(n-2)$-dimensional cell indexed by the network realignment $N=(T,r)$, where $A=\{i,j\}, \> B=V(G)\setminus\{i,j\}$, and $r$ is the constant map with value $(i,j)$. Let $C_{i,j}$ denote the subcomplex of $M_G$ containing $N$ and all of its specifications. $C_{i,j}$ is isomorphic to an $(n-2)$-cube. For the preceding, we also refer to \cite[Proposition 4.9]{Kozlov25}.

\begin{definition}
	Let $\mathcal{S}_G = \{v\in V(G) \mid \degree_G(v) = n-1\}$ be the set of vertices of $G$ that determine the star trees in $X_G$. If $\vert\mathcal{S}_G\vert\geq 2$, we define the subcomplex 
	\begin{equation*}
		\Sigma_G = \bigcup_{i\neq j\in \mathcal{S}_G}C_{i,j}.
	\end{equation*}
	If $\mathcal{S}_G=\{i\}$, then we define $\Sigma_G=\{s_i\}$ and, if $\mathcal{S}_G=\emptyset$, we define $\Sigma_G=\emptyset$.
\end{definition}

$\Sigma_G$ is a subcomplex, since it is the union of the cubes $C_{i,j}$. Moreover, $\Sigma_G$ is $\Aut(G)$-invariant, as it consists precisely of the $(n-2)$-cubes of $X_G$. Finally, if $N\in C_{i,j}$, then $V(\Sp(N)) \subseteq \{i,j\} \subseteq \mathcal{S}_G$, so every vertex of $\Sp(N)$ is unoccupied. Therefore, $\Sigma_G$ is contained in $M_G$. 

To determine the $\Aut(G)$-equivariant topology of the main component, we first show that $\Sigma_G$ is an $\Aut(G)$-equivariant strong deformation retract of $M_G$. For this, we introduce another generalised Morse matching.

Notice that a network realignment belongs to $\Sigma_G$ if and only if its spine consists either of a single point or of a single edge with none of its leaves occupied. In particular, the network realignments in $M_G\setminus \Sigma_G$ are those whose spine is a single edge with exactly one occupied vertex. An example of this is illustrated in \autoref{fig:MatchedPairsMain}~(a).

\begin{definition}
	Let $N \in M_G\setminus \Sigma_G$ be a network realignment. We denote its spine by $(u_N,w_N)$, where $u_N$ is the unique occupied vertex. We define
	\begin{itemize}
		\item $\Act_M^+(N) = B$,
		\item $\Act_M^-(N) = \mathcal{N}_T(w_N)\setminus\{u_N\}$ \quad  and
		\item $\Act_M(N) = \Act_M^+(N) \cup \Act_M^- (N)$.
	\end{itemize}
	The vertices in $\Act_M(N)$ are called \textit{active vertices}.
\end{definition}

We observe that the vertices which are not active are exactly $u_N$ together with its neighbourhood, precisely:
\begin{equation}\label{eq:2ndMNeighb}
	\mathcal{N}_T(u_N) = V(G)\setminus \{\Act_M(N)\cup u_N\}.
\end{equation}
The cardinality of the set of active vertices is a measure of the distance to the star tree $s_{u_N}$.

\begin{lemma} \label{lem:ActDetermined}
	Any network realignment $N\in M_G\setminus \Sigma_G$ has at least one active vertex. Moreover, $N$ is uniquely determined by its spine, $\Act_M^+(N)$ and $\Act_M^-(N)$. 
\end{lemma}

\begin{proof}
	Suppose that $N$ has no active vertices. $\Act_M^+(N) = \emptyset$ means that $N$ is a 0-dimensional realignment. It follows from $\Act_M^-(N) = \emptyset$ and $N\in M_G$ that $N$ is a star tree. But then $N$ lies in $\Sigma_G$.
	
	Let $N_1 = (T_1,r_1), N_2 = (T_2,r_2) \in M_G\setminus \Sigma_G$ with $\Sp(N_1) = \Sp(N_2), \> \Act_M^+(N_1) = \Act_M^+(N_2)$, and $\Act_M^-(N_1) = \Act_M^-(N_2)$. Since $w_{N_1}$ is not occupied, we have $(u_{N_1},v)\in E(G)$ for all $v\in \mathcal{N}_{T_1}(w_{N_1})\setminus \{u_{N_1}\}$ and any specified vertex of $G$ not contained in $\mathcal{N}_{T_1}(w_{N_1})\setminus \{u_{N_1}\}$ is adjacent to $u_{N_1}$ in $T_1$. Further, for an unspecified vertex $b\in B_1$, we have $r_1(b) = (u_{N_1}, w_{N_1})$, and thus $\degree_G(u_{N_1}) = n-1$. Since $u_{N_1}$ is occupied, there exists a vertex $v\in V(G)\setminus \{w_{N_1}\}$ with $(w_{N_1},v) \notin E(G)$, and thus $\degree_G(w_{N_1}) < n-1$. Similarly for $N_2$. Therefore, $u_{N_1} = u_{N_2}$ and $w_{N_1} = w_{N_2}$. By definition of the active vertices, it follows that $N_1 = N_2$.
\end{proof}

Let $N=(T,r)\in M_G\setminus \Sigma_G$, and $v\in \Act_M(N)$. Let $x$ be a leaf of $T$ that occupies $u_N$. If $v\in \Act_M^+(N)$, then $u_N$ is adjacent to $w_N$ and $x$, and $w_N$ is adjacent to $u_N$ and $v$ in $N^{v\to w_N}$, and thus $\Sp(N^{v\to w_N}) = \Sp(N)$. By \autoref{lem:OccupiedLeaves}, the occupied leaves remain unchanged. Consequently, $N^{v\to w_N} \in M_G\setminus \Sigma_G$. If $v\in \Act_M^-(N)$, then again $\Sp(N^{v\to (u_N,w_N)}) = \Sp(N)$ and $N^{v\to (u_N,w_N)} \in M_G\setminus \Sigma_G$.

This motivates the definition of the realignments $d_M(N)$ and $u_M(N)$, in which all active vertices are specified and cospecified, respectively. This corresponds to \autoref{def:interval_limits} of the interval limits in $X_G\setminus Y_G$.

\begin{definition}
	We define two maps $d_M, u_M\colon M_G\setminus \Sigma_G\rightarrow M_G\setminus \Sigma_G$. Let $N\in M_G\setminus \Sigma_G$. The network realignment $d_M(N)=(d_M(T),d_M(r))$ is given by
	\begin{itemize}
		\item $d_M(A) = A\cup \Act_M^+(N) = V(G), \quad d_M(B) = B\setminus \Act_M^+(N) = \emptyset$,
		\item $d_M(T) = T\cup \bigcup\limits_{v\in \Act_M^+(N)} (v,w_N)$,
	\end{itemize}
	and the network realignment $u_M(N)=(u_M(T),u_M(r))$ is given by
	\begin{itemize}
		\item $u_M(A) = A\setminus \Act_M^-(N), \quad u_M(B) = B\cup \Act_M^-(N)$,
		\item $u_M(T) = T \setminus \Act_M^-(N)$,
		\item  $u_M(r)(v) = (u_N,w_N)$ for all $v\in u_M(B)$.
	\end{itemize}
\end{definition}

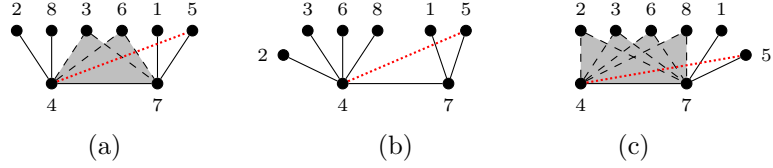
\begin{figure}[tbp]
	\centering
	\begin{tikzpicture}[dot/.style={circle, draw, fill, inner sep=0pt, minimum size=4pt}]
		
		\newcommand{\diagramScale}{0.7}
		
		\begin{scope}[scale=\diagramScale]
			% vertices
			\node[dot, label={above:{\scriptsize $1$}}] (u1) at (2,1) {};
			\node[dot, label={above:{\scriptsize $2$}}] (u2) at (-.66,1) {};
			\node[dot, label={above:{\scriptsize $3$}}] (u3) at (.66,1) {};
			\node[dot, label={below:{\scriptsize $4$}}] (u4) at (0,0) {};
			\node[dot, label={above:{\scriptsize $5$}}] (u5) at (2.66,1) {};
			\node[dot, label={above:{\scriptsize $6$}}] (u6) at (1.33,1) {};
			\node[dot, label={below:{\scriptsize $7$}}] (u7) at (2,0) {};
			\node[dot, label={above:{\scriptsize $8$}}] (u8) at (0,1) {};
			\node () at (1,-1.2) {(a)};
			% arrows
			\draw (u4) -- (u7);
			\draw (u2) -- (u4) -- (u8);
			\draw (u1) -- (u7) -- (u5);
			\draw[dashed] (u4) -- (u3) -- (u7);
			\draw[dashed] (u4) -- (u6) -- (u7);
			%\draw[red, densely dotted, thick] (u1) -- (u2) -- (u3);
			\draw[red, densely dotted, thick] (u4) -- (u5);
		\end{scope}
		
		\begin{scope}[scale=\diagramScale, xshift = 5.5cm]
			% vertices
			\node[dot, label={above:{\scriptsize $1$}}] (v1) at (1.66,1) {};
			\node[dot, label={left:{\scriptsize $2$}}] (v2) at (-1.12,.54) {};
			\node[dot, label={above:{\scriptsize $3$}}] (v3) at (-.66,1) {};
			\node[dot, label={below:{\scriptsize $4$}}] (v4) at (0,0) {};
			\node[dot, label={above:{\scriptsize $5$}}] (v5) at (2.33,1) {};
			\node[dot, label={above:{\scriptsize $6$}}] (v6) at (0,1) {};
			\node[dot, label={below:{\scriptsize $7$}}] (v7) at (2,0) {};
			\node[dot, label={above:{\scriptsize $8$}}] (v8) at (0.66,1) {};
			\node () at (1,-1.2) {(b)};
			% arrows
			\draw (v4) -- (v7);
			\draw (v2) -- (v4) -- (v3);
			\draw (v6) -- (v4) -- (v8);
			\draw (v1) -- (v7) -- (v5);
			%\draw[red, densely dotted, thick] (v1) -- (v2) -- (v3);
			\draw[red, densely dotted, thick] (v4) -- (v5);
		\end{scope}
		
		\begin{scope}[scale=\diagramScale, xshift = 10cm]
			% vertices
			\node[dot, label={above:{\scriptsize $1$}}] (w1) at (2.66,1) {};
			\node[dot, label={above:{\scriptsize $2$}}] (w2) at (0,1) {};
			\node[dot, label={above:{\scriptsize $3$}}] (w3) at (0.66,1) {};
			\node[dot, label={below:{\scriptsize $4$}}] (w4) at (0,0) {};
			\node[dot, label={right:{\scriptsize $5$}}] (w5) at (3.12,.54) {};
			\node[dot, label={above:{\scriptsize $6$}}] (w6) at (1.33,1) {};
			\node[dot, label={below:{\scriptsize $7$}}] (w7) at (2,0) {};
			\node[dot, label={above:{\scriptsize $8$}}] (w8) at (2,1) {};
			\node () at (1,-1.2) {(c)};
			% arrows
			\draw (w4) -- (w7);
			\draw (w1) -- (w7) -- (w5);
			\draw[dashed] (w4) -- (w2) -- (w7);
			\draw[dashed] (w4) -- (w3) -- (w7);
			\draw[dashed] (w4) -- (w6) -- (w7);
			\draw[dashed] (w4) -- (w8) -- (w7);
			%\draw[red, densely dotted, thick] (w1) -- (w2) -- (w3);
			\draw[red, densely dotted, thick] (w4) -- (w5);
		\end{scope}
		
		\begin{scope}[on behind layer]
			\fill[lightgray] (u4.center) -- (u3.center) -- (u7.center) -- cycle;
			\fill[lightgray] (u4.center) -- (u6.center) -- (u7.center) -- cycle;
			\fill[lightgray] (w4.center) -- (w2.center) -- (w7.center) -- cycle;
			\fill[lightgray] (w4.center) -- (w3.center) -- (w7.center) -- cycle;
			\fill[lightgray] (w4.center) -- (w6.center) -- (w7.center) -- cycle;
			\fill[lightgray] (w4.center) -- (w8.center) -- (w7.center) -- cycle;
		\end{scope}
	\end{tikzpicture}
	\vspace{-3mm}
	\caption{Consider the base graph $G = K_8 \setminus \{(1,2),(2,3),(4,5)\}$. (a) A network realignment $N\in M_G\setminus \Sigma_G$ with occupied vertex 7, (b) the network realignment $d_M(N)$, and (c) the network realignment $u_M(N)$.}
	\label{fig:MatchedPairsMain}
\end{figure}

As an example, consider \autoref{fig:MatchedPairsMain}. (a) which shows a network realignment $N\in M_G\setminus \Sigma_G$. The unique occupied vertex of the spine is 7, since $(4,5)\notin E(G)$ and 7 has degree $n-1$ in $G$. The network realignment $d_M(N)$, shown in (b), is given by specifying all vertices in $\Act_M^+(N) = B =\{3,6\}$ to the unoccupied vertex 4, whereas $u_M(N)$ is given by cospecifying all vertices in $\Act_M^-(N) = \mathcal{N}_T(4)\setminus\{7\} = \{2,8\}$ to the edge $(4,7)$, as depicted in (c).

For each $N$, it is clear that $d_M(N)$ is a face of $N$ and $u_M(N)$ is a coface. Now, let $N,\tilde{N}\in M_G\setminus \Sigma_G$ such that $d_M(N)\leq\tilde{N}\leq u_M(N)$. Then, it holds that $\mathcal{N}_{\tilde{N}}(u_{\tilde{N}})=\mathcal{N}_N(u_N)$, and by (\ref{eq:2ndMNeighb}) equivalently $\Act_M(\tilde{N})=\Act_M(N)$.
Further, the spines of every cell in the interval $[d_M(N),u_M(N)]$ are identical, since specifying a vertex $v~\in~\Act_M^+(N)$ to $w_N$ and cospecifying a vertex $v\in \Act_M^-(N)$ to $(u_N,w_N)$, respectively, does not change the spine. Using additionally \autoref{lem:ActDetermined}, we then describe the interval in an explicit way:
\[
[d_M(N), u_M(N)] = \{N'\in M_G\setminus \Sigma_G \mid \Sp(N') = \Sp(N), \> \Act_M(N') = \Act_M(N) \}.
\]
This also implies that $d_M(\tilde{N}) = d_M(N)$ and $u_M(\tilde{N}) = u_M(N)$. By the preceding, these intervals give a well-defined partition.

\begin{definition}
	Let $\sim$ be the collection of intervals in $M_G\setminus \Sigma_G$ of the form
	\[
	[d_M(N),u_M(N)]
	\]
	for $N\in M_G\setminus \Sigma_G$.
\end{definition}

\begin{proposition} \label{prop:EquivariantMatchingII}
	The collection $\sim$ is an $\Aut(G)$-equivariant generalised Morse matching.
\end{proposition}

\begin{proof}
	We show that $\sim$ is a generalised Morse matching. Let
	\[
	\tilde{N}_1 < N_1 > \tilde{N}_2 < N_2 > \ldots > \tilde{N}_k < N_k > \tilde{N}_{k+1} = \tilde{N}_1
	\]
	be a cycle in $M_G\setminus \Sigma_G$ with $d_M(\tilde{N}_i) = d_M(N_i) \neq d_M(\tilde{N}_{i+1})$ for $1\leq i \leq k$. Because $\tilde{N}_{i+1} < N_i$ and $\vert\Sp(\tilde{N}_{i+1})\vert =1$, we have $\Sp(\tilde{N}_{i+1}) = \Sp(N_i)$. This implies $\Act_M(N_i) \neq \Act_M(\tilde{N}_{i+1})$ for $1\leq i \leq k$, and $u_{\tilde{N}_{i+1}} = u_{N_i}$. Hence, combining this with equation (\ref{eq:2ndMNeighb}), we find $\mathcal{N}_{T_i}(u_{N_i}) \subseteq \mathcal{N}_{\tilde{T}_{i+1}}(u_{\tilde{N}_{i+1}})$. It follows
	\begin{align*}
		\Act_M (\tilde{N}_{i+1}) &= V(G)\setminus (\{u_{\tilde{N}_{i+1}}\} \cup \mathcal{N}_{\tilde{T}_{i+1}}(u_{\tilde{N}_{i+1}}) ) \\
		&\subseteq V(G)\setminus (\{u_{N_i}\} \cup \mathcal{N}_{T_i}(u_{N_i})) = \Act_M (N_i).
	\end{align*}
	Therefore, $\Act_M(N_i) \supsetneq \Act_M(\tilde{N}_{i+1})$ and 
	\[
	\Act_M(\tilde{N}_1) = \Act_M(N_1) \supsetneq \ldots \supsetneq \Act_M(\tilde{N}_k) = \Act_M(N_k) \supsetneq \Act_M(\tilde{N}_1).
	\]
	Consequently, $k=1$ and $\sim$ is acyclic.
	
	Finally, the $\Aut(G)$-equivariance of the generalised Morse matching~$\sim$ is given since the construction is independent of the labelling of the vertices.
\end{proof}

\begin{proposition}\label{prop:MontoSigma}
	$\Sigma_G$ is an $\Aut(G)$-equivariant strong deformation retract of $M_G$. 
\end{proposition}

\begin{proof}
	This is an immediate consequence of \autoref{prop:EquivariantMatchingII} and \autoref{Freij-generalization}. 
\end{proof}

It remains to determine the $\Aut(G)$-equivariant homotopy type of $\Sigma_G$. We extend Corollary 4.10 in \cite{Kozlov25} to general graphs in an $\Aut(G)$-equivariant setting. For this purpose, we introduce the subcomplex $K_{\mathcal{S}_G}$.

\begin{definition} 
	Let $\mathcal{S}_G = \{v\in V(G)\mid \deg(v) = n-1\}$ and let $K_{\mathcal S_G}$ denote the complete graph with vertex set $\mathcal S_G$. Since graph automorphisms preserve degrees, the action of $\Aut(G)$ on $G$ restricts to $\mathcal{S}_G$ and thus induces an action of $\Aut(G)$ on~$K_{\mathcal{S}_G}$. 
	
	We define an $\Aut(G)$-equivariant topological embedding 
	\[
	\iota\colon K_{\mathcal S_G}\hookrightarrow \Sigma_G 
	\]
	by sending the vertex $i\in\mathcal S_G$ to the star tree $s_i$ and the edge $(i,j)$ to the diagonal $d_{i,j}$ between $s_i$ and $s_j$ in $C_{i,j}\subseteq \Sigma_G$.
	
	Henceforth, we identify $K_{\mathcal S_G}$ with its image under $\iota$.
\end{definition}

\begin{proposition}\label{lem:graphinclusion}
	$K_{\mathcal{S}_G}$ is an $\Aut(G)$-equivariant strong deformation retract of~$\Sigma_G$.
\end{proposition}

\begin{proof}
	For any $i\neq j\in \mathcal{S}_G$ the intersection of $C_{i,j}$ and the subspace $K_{\mathcal S_G}$ contains precisely the diagonal joining the two opposite vertices $s_i$ and $s_j$. We retract $C_{i,j}$ onto this diagonal by the straight-line homotopy from each point to its orthogonal projection onto the diagonal. This fixes the diagonal pointwise.
	
	These deformation retractions agree on intersections, because distinct $(n-2)$-cubes meet only in star tree vertices, where the homotopy is constant. Hence the contractions glue to a deformation retraction of $\Sigma_G$ onto $K_{\mathcal S_G}$.
	
	Any automorphism $g\in \Aut(G)$ sends $C_{i,j}$ to $C_{g i,g j}$, sends the diagonal $d_{i,j}$ to $d_{g i,g j}$, and preserves orthogonal projection to these diagonals. Therefore, the resulting strong deformation retraction is $\Aut(G)$-equivariant.
\end{proof}

\begin{proposition}\label{MainCptKn}
	$K_{\mathcal{S}_G}$ is an $\Aut(G)$-equivariant strong deformation retract of $M_G$.
\end{proposition}

\begin{proof}
	This follows from \autoref{lem:graphinclusion} and \autoref{prop:MontoSigma}. 
\end{proof}

\subsection{The Residual Components of \texorpdfstring{$Y_G$}{Y_G}}\label{sec:residual}

The preceding subsection gives a complete description of the topology of the main component $M_G$. Our next goal is to determine the topology of the remaining connected components of $Y_G$. 

\begin{definition}
	By $R_G$ we denote the complement of $M_G$ in $Y_G$. The connected components of $R_G$ are called \textit{residual components}.
	
	Since, by \autoref{prop:MainComponent}, the $\Aut(G)$-action on $Y_G$ restricts to the main component, it also restricts to its complement $R_G$.
\end{definition}

We proceed by classifying the residual components by their spines and spine assignments. We then give a cubical description of the residual components as products of subtrees of their spines. This description is compatible with the action of $\Aut(G)$ and allows us to contract $R_G$ equivariantly to a discrete $\Aut(G)$-space.

\begin{definition}
	Let $N\in X_G$ be a network realignment. Let $Z_N = V(G)\setminus V(\Sp(N))$. The \textit{spine assignment} $c_N$ of $N$ assigns to every vertex $v\in Z_N$ a connected component $c_N(v)$ of $\Sp(N)\setminus \{w\in \Sp(N) \mid w \text{ is a barrier of }v\}$ as follows:
	\begin{itemize}
		\item if $v \in B$, then $c_N(v)$ is the component containing $r(v)$,
		\item if $v\in A$, then $c_N(v)$ is the component containing $p_T(v)$.
	\end{itemize}
\end{definition}

\begin{example}%better example: in A(G)
	\begin{figure}[tbp]
		\centering
		\begin{tikzpicture}[dot/.style={circle, draw, fill, inner sep=0pt, minimum size=4pt}]
			
			\newcommand{\diagramScale}{0.7}
			
			\begin{scope}[scale=\diagramScale]
				% vertices
				\node[dot, label={left:{\scriptsize $1$}}] (u1) at (0.5,1) {};
				\node[dot, label={below:{\scriptsize $2$}}] (u2) at (1,0) {};
				\node[dot, label={right:{\scriptsize $3$}}] (u3) at (2.5,1) {};
				\node[dot, label={right:{\scriptsize $4$}}] (u4) at (3.5,1) {};
				\node[dot, cyan, label={below:{\scriptsize $5$}}] (u5) at (3,0) {};
				\node[dot, label={below:{\scriptsize $6$}}] (u6) at (0,0) {};
				\node[dot, cyan, label={below:{\scriptsize $7$}}] (u7) at (2,0) {};
				\node[dot, cyan, label={below:{\scriptsize $8$}}] (u8) at (4,0) {};
				\node () at (2,-1.2) {(a)};
				% arrows
				\draw (u6) -- (u2) -- (u7);
				\draw[cyan] (u7) -- (u5) -- (u8);
				\draw (u1) -- (u6);
				\draw (u4) -- (u8);
				\draw[dashed] (u7) -- (u3) -- (u5);
				\draw[red, densely dotted, thick] (u1) -- (u2) -- (u3);
				\draw[red, densely dotted, thick] (u4) -- (u5);
			\end{scope}
			
			\begin{scope}[scale=\diagramScale, xshift = 6.5cm]
				% vertices
				\node[dot, label={left:{\scriptsize $1$}}] (v1) at (-.3,1) {};
				\node[dot, label={below:{\scriptsize $2$}}] (v2) at (1,0) {};
				\node[dot, label={right:{\scriptsize $3$}}] (v3) at (.3,1) {};
				\node[dot, label={right:{\scriptsize $4$}}] (v4) at (3.5,1) {};
				\node[dot, label={below:{\scriptsize $5$}}] (v5) at (3,0) {};
				\node[dot, cyan, label={below:{\scriptsize $6$}}] (v6) at (0,0) {};
				\node[dot, label={below:{\scriptsize $7$}}] (v7) at (2,0) {};
				\node[dot, label={below:{\scriptsize $8$}}] (v8) at (4,0) {};
				\node () at (2,-1.2) {(b)};
				% arrows
				\draw (v6) -- (v2) -- (v7) -- (v5) -- (v8);
				\draw (v4) -- (v8);
				\draw (v1) -- (v6);
				\draw (v3) -- (v6);
				\draw[red, densely dotted, thick] (v1) -- (v2) -- (v3);
				\draw[red, densely dotted, thick] (v4) -- (v5);
			\end{scope}
			
			\begin{scope}[on behind layer]
				\fill[lightgray] (u7.center) -- (u3.center) -- (u5.center) -- cycle;
			\end{scope}
		\end{tikzpicture}
		\vspace{-3mm}
		\caption{Two network realignments in the minimal spine complex $Y_G$, where $G = K_8 \setminus \{(1,2),(2,3),(4,5)\}$ with the spine assignment of 3 highlighted in blue.}
		\label{fig:SpineAss}
	\end{figure}
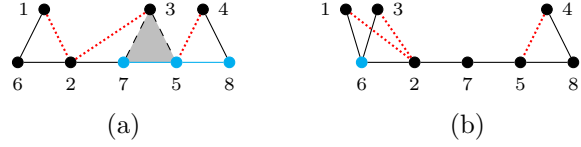
	As an example of two different spine assignments, consider \autoref{fig:SpineAss}. Take $G=K_8 \setminus \{(1,2),(2,3),(4,5)\}$. The network realignments depicted in \autoref{fig:SpineAss} certainly have the same spine. However, since $2$ is a barrier of $3$, the movement of $3$ is restricted to either the left hand side or the right hand side of $2$. This condition is expressed by the spine assignment. In (a), the spine assignment assigns to $3$ the subtree $(7,5,8)$ and in (b) it assigns to $3$ the single vertex subtree $(6)$. As stated in the following proposition, this means that the two network realignments are contained in different connected components of $Y_G$.
\end{example}

\begin{proposition}\label{prop:ResidualComponents}
	Any two network realignments in $R_G$ lie in the same connected component of $R_G$ if and only if they have the same spine and the same spine assignment.
\end{proposition}

\begin{proof}
	For $N=(T,r)\in R_G$, spine and spine assignments are stable under specification. Suppose $V(\Sp(N)) \setminus V(\Sp(N^{b\to w})) = \{u\}$. Then $r(b) = (u,w)$ and $\degree_T(u)=1$. Hence, $u$ is not occupied in $\Sp(N)$, contradicting $N\in R_G$. It follows that $\Sp(N) = \Sp(N^{b\to w})$. Since $p_T(v) = p_{T^{b\to w}} (v)$ if $v\in V(T)\setminus V(\Sp(N))$ and $r(v) = r^{b\to w}(v)$ if $v\in B\setminus\{b\}$, we have $c_N(v) = c_{N^{b\to w}} (v)$ for all $v\neq b$. Because $p_{T^{b\to w}} (b) = w\in r(b)$, also $c_N(b) = c_{N^{b\to w}}(b)$. Thus, $c_N = c_{N^{b\to w}}$, proving the claim.
	
	It remains to show that $N = (T,r)$ and $N' = (T',r')$ lie within the same connected component if their spine and spine assignments coincide. It suffices to consider 0-dimensional realignments $N,N'\in R_G$ with $\Sp(N) = \Sp(N')$ and $c_N = c_{N'}$. Consider $v\in Z_N = Z_{N^\prime}$. Then $p_T(v)$ and $p_{T'}(v)$ are contained within the same connected component $c_N (v) = c_{N'} (v)$. Let $p_T(v) = u_0, u_1,\ldots, u_k = p_{T'}(v)$ be the unique path between $p_T(v)$ and $p_{T'}(v)$ in $\Sp(N)$. Realigning $v$ along this path and doing this iteratively for all vertices in $Z_N$ yields a sequence of network realignments connecting $N$ and $N'$.
\end{proof}

\begin{definition}
	Let $\Gamma$ be a residual component. By \autoref{prop:ResidualComponents}, $\Sp(N)$, $Z_N$ and $c_N$ are independent of the choice of $N\in \Gamma$. We define $\Sp(\Gamma) = \Sp(N)$, $Z_\Gamma = Z_N$ and $c_\Gamma = c_N$ for any choice of $N\in \Gamma$.
\end{definition}

\begin{lemma}\label{lem:spineeqivariant}
	The spine and spine assignment respect the $\Aut(G)$-action in the following sense. Let $g\in \Aut(G)$ and let $\Gamma$ be a connected component of $R_G$, then $\Sp(g\Gamma)= g\Sp(\Gamma)$. Further, $Z_{g\Gamma} = gZ_\Gamma$, and for any $w\in Z_{g\Gamma}$ we have $c_{g\Gamma}(w) = g c_{\Gamma}(g^{-1}w)$.
\end{lemma}

\begin{proof}
	Let $N = (T,r)\in \Gamma$ and $g\in \Aut(G)$. Since $g$ maps interior vertices of $T$ to interior vertices of $gT$, and maps $\Image(r)$ to $\Image(gr)$, we have $\Sp(gN)=g\Sp(N)$, and hence $\Sp(g\Gamma)=g\Sp(\Gamma)$. In particular, $Z_{g\Gamma} = gZ_\Gamma$.
	
	Let $w\in Z_{g\Gamma}$ and set $v=g^{-1}w\in Z_\Gamma$. Let $B_v=\{x\in V(\Sp(\Gamma)) \mid x \text{ is a barrier of } v\}$. Then, because $g$ preserves barriers, $gB_v = \{y\in V(\Sp(g\Gamma)) \mid y \text{ is a barrier of } w\}$. Thus, $g$ induces an isomorphism
	\[
	\Sp(\Gamma)\setminus B_v
	\to
	\Sp(g\Gamma)\setminus gB_v.
	\]
	This isomorphism maps connected components to connected components. Therefore, the connected component assigned to $v$ is sent to the connected component assigned to $w$. This implies that 
	\[
	c_{g\Gamma}(w)=g\,c_\Gamma(v)=g\,c_\Gamma(g^{-1}w),
	\]
	which proves the claim.
\end{proof}

Consider the topological space 
\[
\coprod_{\Gamma\in \pi_0(R_G)}\prod_{v\in Z_\Gamma} c_\Gamma(v).
\]
Since the connected components of this space are products of trees, it has a cubical structure, where a face $f$ is a collection $(f_v)_{v\in Z_\Gamma}$, such that for each $v\in Z_\Gamma$, $f_v$ is either a vertex or an edge of $c_\Gamma(v)$. 

For each $g\in \Aut (G)$ and $w\in gZ_\Gamma$, we define
\[
(gf)_{w} = g(f_{g^{-1}w})\in c_{g\Gamma}(w).
\]
By \autoref{lem:spineeqivariant}, the resulting family $gf$ is a well-defined element of $\prod_{w\in Z_{g\Gamma}}c_{g\Gamma}(w)$. It is easy to see that this defines a group action of $\Aut(G)$.

\begin{proposition}\label{cptscontractible}
	There is an $\Aut(G)$-equivariant cubical isomorphism
	\begin{equation}\label{Eq:CubProduct}
		\varphi\colon R_G \to \coprod_{\Gamma\in \pi_0(R_G)}\prod_{v\in Z_\Gamma} c_\Gamma(v).
	\end{equation}
\end{proposition}

\begin{proof}
	Since both sides decompose into coproducts indexed by $\pi_0(R_G)$, we may construct the isomorphism $\varphi$ on the individual connected components of $R_G$. Let $\Gamma$ be a connected component.
	
	Let 
	\[
	\varphi^\Gamma\colon \Gamma \to \prod_{v\in Z_\Gamma} c_\Gamma(v)
	\] 
	be the following map. For any network realignment $N\in \Gamma$, we define
	\[
	\varphi^\Gamma(N)_v = \begin{cases}
		r(v) \quad 		&\text{if }v\in B,\\
		p_T(v) \quad 	&\text{if }v\in A.
	\end{cases}
	\] 
	This is a well-defined face of the product, because in both cases $\varphi^\Gamma(N)_v$ is a face of $c_\Gamma(v)$.
	
	Conversely, let 
	\[
	\psi^\Gamma\colon \prod_{v\in Z_\Gamma} c_\Gamma(v)\to \Gamma 
	\] 
	be the following map. For a face $f$ of the product, let $D_f$ be the vertices $v$ in $Z_\Gamma$ for which $f_v$ is a vertex and let $B_f$ be the ones for which $f_v$ is an edge. We define $\psi^\Gamma(f)$ to be the network realignment $(T_f,r_f)$, where 
	\[
	T_f = \Sp(\Gamma) \cup \bigcup_{v\in D_f}\{(v,f_v)\}
	\]
	and 
	\[
	r_f\colon B_f \to E(T_f),\qquad v\mapsto f_v.
	\]
	To see that the resulting realignment $\psi^\Gamma(f) = (T_f, r_f)$ lies in the component $\Gamma$, we first verify that it has the correct spine. Using the fact that every leaf of $\Sp(\Gamma)$ is occupied, we see that the interior of $T_f$ is exactly $\Sp(\Gamma)$. Further, since the image of $r_f$ is contained in $\Sp(\Gamma)$, we know that $\Sp(\psi^\Gamma(f)) = \Int(T_f) = \Sp(\Gamma)$.
	
	We also have to verify that $\psi^\Gamma(f)$ has the correct spine assignment, i.e. that $c_{\psi^\Gamma(f)} = c_\Gamma$. Let $v\in Z_\Gamma$. Then $c_{\psi^\Gamma(f)}(v)$ is by construction the component of $\Sp(\Gamma)\setminus \{\text{barriers of }v\}$ containing $f_v$. This means that $c_{\psi^\Gamma(f)}(v) = c_\Gamma(v)$, because $f_v$ is a cell of $c_\Gamma(v)$. Consequently, $\psi^\Gamma(f) \in \Gamma$.
	
	We claim that the two maps $\varphi^\Gamma$ and $\psi^\Gamma$ are mutual inverses. Let $N = (T,r)$ be a network realignment in $\Gamma$, then
	\[
	T_{\varphi^\Gamma(N)} = \Sp(\Gamma) \cup \bigcup_{v\in D_{\varphi^\Gamma(N)}}\{(v,\varphi^\Gamma(N)_v)\} = \Sp(N) \cup \bigcup_{v\in V(T)\setminus V(\Sp(N))}\{(v,p_T(v))\} = T.
	\]
	For $b\in B = B_{\varphi^\Gamma(N)}$, we have
	\[
	r_{\varphi^\Gamma(N)}(b) = \varphi^\Gamma(N)_b = r(b).
	\]
	This shows that $\psi^\Gamma(\varphi^\Gamma(N)) = N$.
	
	Now, let $f$ be a face of the product and $v\in Z_\Gamma$. If $v\in B_f$, then 
	\[
	\varphi^\Gamma(\psi^\Gamma(f))_v = r_f(v) = f_v.
	\] 
	If $v\in D_f$, then
	\[
	\varphi^\Gamma(\psi^\Gamma(f))_v = p_{T_f}(v).
	\]
	But in $T_f$, $v$ is a leaf with parent $f_v$, so $\varphi^\Gamma(\psi^\Gamma(f))_v = f_v$. This shows that $\varphi^\Gamma(\psi^\Gamma(f)) = f$. We conclude that $\varphi^\Gamma$ and $\psi^\Gamma$ are mutual inverses.
	
	We now verify that $\varphi^\Gamma$ and $\psi^\Gamma$ are maps of posets. Consider a network realignment $N$ and a specification $N^{b\to u}$ of $N$. We must show that $\varphi^\Gamma(N^{b\to u})$ is a face of $\varphi^\Gamma(N)$. Let $v$ be a vertex in $Z_\Gamma$. If $v\neq b$, then $\varphi^\Gamma(N)_v = \varphi^\Gamma(N^{b\to u})_v$ and if $v = b$, then $\varphi^\Gamma(N^{b\to u})_v = u$, which is a face of $\varphi^\Gamma(N)_v = r(v)$. Hence, $\varphi^\Gamma(N^{b\to u})$ is a face of $\varphi^\Gamma(N)$, which shows that $\varphi^\Gamma$ is order preserving. 
	
	The proof that the inverse map $\psi^\Gamma$ is order preserving is similar. Replacing an edge $f_v$ with a vertex corresponds exactly to specifying $v$.
	
	We denote by 
	\[
	\varphi\colon R_G \to \coprod_{\Gamma\in \pi_0(R_G)}\prod_{v\in Z_\Gamma} c_\Gamma(v).
	\]
	the resulting isomorphism assembled from the isomorphisms $\varphi_\Gamma$ on the individual connected components. It remains to show that $\varphi$ is equivariant.
	
	Let $g\in \Aut(G)$ and let $N = (T,r)\in \Gamma$ for some connected component $\Gamma$ of $R_G$. Then, for any $w\in Z_{g\Gamma}$, if $w$ is unspecified in $gN$, we have
	\[
	\varphi(gN)_w = (gr)(w) = g(r(g^{-1}w)) = g (\varphi(N)_{g^{-1}w}) = (g\varphi(N))_w,
	\]
	and if $w$ is specified in $gN$, we have
	\[
	\varphi(gN)_w = p_{gT}(w) = g(p_{T}(g^{-1}w)) = g(\varphi(N)_{g^{-1}w}) = (g\varphi(N))_w.
	\]
	This shows that $\varphi$ is equivariant.
\end{proof}

\begin{definition}\label{def:centroid}
	A vertex $c$ of a tree $T$ with $n$ vertices is called a \textit{centroid} if all connected components of $T\setminus \{c\}$ have at most $n/2$ vertices.
\end{definition}

A theorem of Jordan \cite{Centroid} states that every finite tree has either a unique centroid or exactly two adjacent centroids. 

\begin{lemma}[\cite{Centroid}]\label{existenceofcentral}
	Let $T$ be a tree on $n$ vertices. If $n$ is odd, $T$ has a unique centroid. If $n$ is even, then either $T$ has exactly one centroid, or two adjacent ones.
\end{lemma}

The main feature of centroids that interests us is that the set of centroids is preserved by every tree isomorphism. This allows us to construct a point in the geometric realisation of a tree that is fixed by any group action on the tree.

\begin{definition}
	The \textit{geometric center} $z_\mathrm{geom}(T)$ of a tree $T$ is defined to be the centroid of $T$, viewed as a point in the geometric realisation of $T$, if $T$ has a unique centroid, or the central point on the edge connecting the two centroids, if $T$ has two adjacent centroids.
\end{definition}

\begin{proposition}
	The residual components $R_G$ admit an $\Aut(G)$-equivariant strong deformation retraction onto a discrete $\Aut(G)$-subspace $D_G$.
\end{proposition}	

\begin{proof}
	It suffices to give an equivariant homotopy on 
	\[
	\coprod_{\Gamma\in \pi_0(R_G)}\prod_{v\in Z_\Gamma} c_\Gamma(v)
	\]
	between the identity and a map that collapses every connected component to a point. For a connected component $\Gamma$, we choose this point to be $(z_\mathrm{geom}(c_\Gamma(v)))_{v\in Z_\Gamma}$. Thus, under the identification from \autoref{cptscontractible}, the discrete $\Aut(G)$-space is 
	\[
	D_G = \{(z_\mathrm{geom}(c_\Gamma(v)))_{v\in Z_\Gamma}\mid \Gamma \in \pi_0(R_G)\}.
	\]
	Let $T$ be a tree and let $z_\mathrm{geom}(T)$ be its geometric center. By $H^T\colon T\times I\to T$ we denote the homotopy that linearly collapses $T$ to $z_\mathrm{geom}(T)$. Since this homotopy is defined purely in terms of the metric on the geometric realisation of $T$, and tree isomorphisms preserve the metric, this homotopy is also preserved by tree isomorphisms. This means that if $f\colon T\to T^\prime$ is a tree isomorphism, then for any $t\in I$ and $x\in T$, we have 
	\begin{equation}
		\label{isominv}
		f(H^T(x,t)) = H^{T^\prime}(f(x), t).
	\end{equation}
	Now, for any connected component $\Gamma$ of $R_G$, we define 
	\[
	H^\Gamma \colon \left(\prod_{v\in Z_\Gamma} c_\Gamma(v) \right)\times I\to \prod_{v\in Z_\Gamma} c_\Gamma(v)
	\]
	via $H^\Gamma(x,t)_v = H^{c_\Gamma(v)}(x_v,t)$, whenever $x\in \prod_{v\in Z_\Gamma} c_\Gamma(v)$ and $t\in I$. The homotopies $H^\Gamma$ on $\prod_{v\in Z_\Gamma} c_\Gamma(v)$ combine to give a homotopy
	\[
	H\colon \left(\coprod_{\Gamma\in \pi_0(R_G)}\prod_{v\in Z_\Gamma} c_\Gamma(v)\right)\times I \to \coprod_{\Gamma\in \pi_0(R_G)}\prod_{v\in Z_\Gamma} c_\Gamma(v).
	\]
	The homotopy $H$ is $\Aut(G)$-equivariant, because for any $g\in \Aut(G)$ and $x\in\prod_{v\in Z_\Gamma} c_\Gamma(v)$ we have
	\[
	\begin{aligned}
		H^{g\Gamma} (gx,t)_w
		&= H^{c_{g\Gamma}(w)}\bigl((gx)_w,t\bigr) \\
		&= H^{c_{g\Gamma}(w)}\bigl(g(x_{g^{-1}w}),t\bigr) \\
		&= H^{g c_\Gamma(g^{-1}w)}\bigl(g(x_{g^{-1}w}),t\bigr) \\
		&\overset{(\ref{isominv})}{=} g\left(H^{c_\Gamma(g^{-1}w)}(x_{g^{-1}w},t)\right) \\
		&= g\left(H^\Gamma(x,t)_{g^{-1}w}\right) \\
		&= \bigl(gH^\Gamma(x,t)\bigr)_w.
	\end{aligned}
	\]
	Hence, we have produced an equivariant strong deformation retraction onto a discrete $\Aut(G)$-space.
\end{proof}

	\subsection{Examples}\label{sec:examples}
In this section we examine several examples of network realignment complexes for natural and relatively symmetric classes of base graphs. These examples already exhibit the phenomenon of having multiple connected components. Throughout, we apply the techniques and terminology developed in the previous sections, such as the spine and spine assignments of the connected components. We also investigate the action of the automorphism group of the base graph on the corresponding complex and determine the resulting orbit structure. For general base graphs, however, the enumeration of the connected components or even their orbits under the group action remains open, and the following examples will illustrate some of the challenges involved. 

The first example is given by Kozlov \cite{Kozlov25}, who showed that $X_{K_n} \simeq K_n$. The present results strengthen this statement by taking the symmetry of the complex into account. The automorphism group of the base graph is given by $\Aut(K_n)=S_n$. Hence, we obtain the equivariant homotopy equivalence $X_{K_n} \simeq_{S_n}~K_n$. 

In the following examples, the base graph we consider is $G = K_n\setminus E(P)$, where $P$ is a path of length $k$ and $0<k<n$. An automorphism of $G$ can either fix or reverse $P$ and can freely relabel the vertices not contained in $P$. Hence, $\Aut(G) \cong S_{n-k-1}\times C_2$, where $C_2$ is the cyclic group with two elements.

\begin{corollary}\label{1EdgeM}
	Let $n\geq 3$, and let $G=K_n\setminus \{e\}$ for $e\in E(K_n)$. Then $X_G$ admits an $(S_{n-2}\times C_2)$-equivariant strong deformation retraction onto $K_{n-2}$.
\end{corollary}

\begin{proof} 
	Let $e=(u,w)$. In the cubical complex $X_G$, the vertices $u$ and $w$ form a pair of mutual barriers. Since every network realignment contains at most one occupied leaf, no connected component other than the main component $M_G$ can arise. The set of star trees is given by $\mathcal{S}_G = \lbrack n \rbrack \setminus\{u,w\}$, and it follows from \autoref{MainCptKn} that
	\begin{equation*}
		X_G \simeq_{(S_{n-2}\times C_2)} K_{\mathcal S_G} \cong K_{n-2},
	\end{equation*}
	which proves the claim.
\end{proof}

A projection of the above complex $X_G$ for $n=5$ is shown in \autoref{fig:NRCn=5}. Let $e=(1,2)$. The three grey 3-cubes connect the white star trees $s_1,s_2$ and $s_3$. The violet squares attached to the star trees represent the cells with spine size one and a single occupied leaf. The following maximal cells do not have minimal spine: The blue 2-cubes correspond to cells with spine size two and no occupied leaf, and the red edges correspond to network realignments with spine size two and one occupied leaf. $X_G$ admits a strong deformation retraction onto the main component $M_G$, which does not contain the blue and the red cells. The diagonals of the three grey 3-cubes connect the star trees and form the complete graph $K_3$, which is a strong deformation retract of $X_G$. The symmetric group $S_3$ acts by permuting the star trees.

\begin{corollary}\label{2PathM}
	Assume $n\geq 5$. Let $P$ be a path of length two, and let $G=K_n\setminus E(P)$. Then we have the following $(S_{n-3}\times C_2)$-equivariant strong deformation retraction of $X_G$:
	\begin{equation*}
		X_G \simeq_{(S_{n-3}\times C_2)} K_{n-3} \cup \coprod\limits_{i=1}^{(n-3)(n-4)} \{\ast\},
	\end{equation*}
	where there are two orbits of connected components under the action of $\Aut(G)$.
\end{corollary}

\begin{proof}
	Since the argument is independent of the labelling, we denote $P=(1,2,3)$. After collapsing the complex $\Aut(G)$-equivariantly onto its minimal spine complex $Y_G$ by \autoref{Y_G-matching-thm}, it follows from \autoref{prop:ResidualComponents} that every connected component distinct from $M_G$ is uniquely determined by its spine and spine assignment.
	
	We show that the $\Aut(G)$-action on $Y_G$ has precisely two orbits of connected components, one of them containing only $M_G$. Let $\Gamma$ be any connected component of the other orbit. Its spine is of the form $\Sp(\Gamma)=(x,2,y)$ with $x,y\in \mathcal{S}_G$ being the occupied leaves. Indeed, a realignment can only contain two occupied leaves if the vertex 2 serves as a barrier for both 1 and 3. Without loss of generality, the spine assignment is given by $c_{\Gamma}(1)=\{x\},c_{\Gamma}(3)=\{y\}$ and $c_{\Gamma}(v)=\Sp(\Gamma)$ for every $v\in [n]\setminus\{1,2,3,x,y\}$, since none of those vertices admits a barrier. As $x$ and $y$ may be chosen independently from $\mathcal{S}_G=\{4,\dots,n\}$, which has cardinality $(n-3)$, the orbit contains $(n-3)(n-4)$ components. By \autoref{cptscontractible}, the residual components $R_G$ admit an $\Aut(G)$-equivariant strong deformation retract onto a discrete space. Furthermore, \autoref{MainCptKn} yields $M_G\simeq_{\Aut(G)} K_{n-3}$, which finishes the proof.
\end{proof}

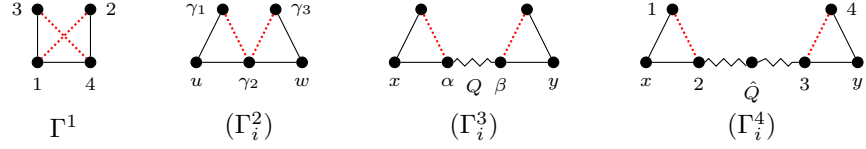
\begin{figure}[tbp]
	\centering
	\begin{tikzpicture}[dot/.style={circle, draw, fill, inner sep=0pt, minimum size=4pt}]
		
		\newcommand{\diagramScale}{0.7}
		
		\begin{scope}[scale=\diagramScale]
			% vertices
			\node[dot, label={left:{\scriptsize $3$}}] (w1) at (0,1) {};
			\node[dot, label={below:{\scriptsize $4$}}] (w2) at (1,0) {};
			\node[dot, label={right:{\scriptsize $2$}}] (w4) at (1,1) {};
			\node[dot, label={below:{\scriptsize $1$}}] (w6) at (0,0) {};
			\node () at (0.5,-1.2) {$\Gamma^1$};
			% arrows
			\draw (w6) -- (w2);
			\draw (w4) -- (w2);
			\draw (w1) -- (w6);
			\draw[red, densely dotted, thick] (w4) -- (w6);
			\draw[red, densely dotted, thick] (w1) -- (w2);
		\end{scope}
		
		\begin{scope}[scale=\diagramScale, xshift=3cm]
			% vertices
			\node[dot, label={left:{\scriptsize $\gamma_1$}}] (w1) at (0.5,1) {};
			\node[dot, label={below:{\scriptsize $\gamma_2$}}] (w2) at (1,0) {};
			\node[dot, label={right:{\scriptsize $\gamma_3$}}] (w4) at (1.5,1) {};
			\node[dot, label={below:{\scriptsize $u$}}] (w6) at (0,0) {};
			\node[dot, label={below:{\scriptsize $w$}}] (w7) at (2,0) {};
			\node () at (1,-1.2) {$(\Gamma^2_i)$};
			% arrows
			\draw (w2) -- (w7);
			\draw (w6) -- (w2);
			\draw (w4) -- (w7);
			\draw (w1) -- (w6);
			\draw[red, densely dotted, thick] (w4) -- (w2);
			\draw[red, densely dotted, thick] (w1) -- (w2);
		\end{scope}
		
		\begin{scope}[scale=\diagramScale, xshift=6.75cm]
			% vertices
			\node[dot] (w1) at (0.5,1) {};
			\node[dot, label={below:{\scriptsize $\alpha$}}] (w2) at (1,0) {};
			\node[dot] (w4) at (2.5,1) {};
			\node[dot, label={below:{\scriptsize $y$}}] (w5) at (3,0) {};
			\node[dot, label={below:{\scriptsize $x$}}] (w6) at (0,0) {};
			\node[dot, label={below:{\scriptsize $\beta$}}] (w7) at (2,0) {};
			\node () at (1.5,-1.2) {$(\Gamma^3_i)$};
			% arrows
			\draw[
			decorate,
			decoration={
				zigzag,
				amplitude=0.5mm,
				segment length=2mm
			}
			] (w2) -- node[below, yshift=-1mm]{\scriptsize $Q$} (w7);
			\draw (w6) -- (w2);
			\draw (w5) -- (w7);
			\draw (w4) -- (w5);
			\draw (w1) -- (w6);
			\draw[red, densely dotted, thick] (w4) -- (w7);
			\draw[red, densely dotted, thick] (w1) -- (w2);
		\end{scope}
		
		\begin{scope}[scale=\diagramScale, xshift=11.5cm]
			% vertices
			\node[dot, label={left:{\scriptsize $1$}}] (w1) at (0.5,1) {};
			\node[dot, label={below:{\scriptsize $2$}}] (w2) at (1,0) {};
			\node[dot, label={right:{\scriptsize $4$}}] (w4) at (3.5,1) {};
			\node[dot, label={below:{\scriptsize $3$}}] (w5) at (3,0) {};
			\node[dot, label={below:{\scriptsize $x$}}] (w6) at (0,0) {};
			\node[dot, label={[label distance=0.5mm]below:{\scriptsize $\hat{Q}$}}] (w7) at (2,0) {};
			\node[dot, label={below:{\scriptsize $y$}}] (w8) at (4,0) {};
			\node () at (2,-1.2) {$(\Gamma^4_i)$};
			% arrows
			\draw[
			decorate,
			decoration={
				zigzag,
				amplitude=0.5mm,
				segment length=2mm
			}
			] (w2) -- (w7) -- (w5);
			\draw (w6) -- (w2);
			\draw (w5) -- (w8);
			\draw (w4) -- (w8);
			\draw (w1) -- (w6);
			\draw[red, densely dotted, thick] (w4) -- (w5);
			\draw[red, densely dotted, thick] (w1) -- (w2);
		\end{scope}
	\end{tikzpicture}
	\vspace{-3mm}
	\caption{The spines of the components in each orbit type of \autoref{cor:3-pathdeleted} together with the occupying vertices.}
	\label{fig:3-pathOrbits}
\end{figure}

In the next example, the base graph is obtained from the complete graph by deleting a path of length three, where the different orbit types are depicted in \autoref{fig:3-pathOrbits}.

\begin{corollary}\label{cor:3-pathdeleted}
	Assume $n\geq 5$. Let $P$ be a path of length three, and let $G=K_n\setminus E(P)$. Then we have the following $\Aut(G)$-equivariant strong deformation retraction of $X_G$:
	\begin{equation*}
		X_G \simeq_{(S_{n-4}\times C_2)} K_{n-4} \cup \coprod\limits_{i=1}^{f(n)} \{\ast\},
	\end{equation*}
	where
	\begin{equation*}
		f(n)=4(n-4)(n-5)\left( \sum\limits_{k=0}^{n-6} \binom{n-6}{k}k! \right) + (n-4)(n-3) +1,
	\end{equation*}
	and the union of points forms a discrete $(S_{n-4}\times C_2)$-space containing $(3n-13)$ orbits.
\end{corollary}

\begin{proof}
	First assume $n\geq 6.$ For convenience, let $P=(1,2,3,4)$ denote the deleted path. As in the previous examples, by \autoref{Y_G-matching-thm} we replace $X_G$ with its $\Aut(G)$-equivariantly retracted minimal spine complex $Y_G$ and analyse the connected components of $Y_G$, whose network realignments contain at least two occupied leaves. By \autoref{prop:ResidualComponents}, each such component is uniquely determined by its spine and spine assignment. We analyse the possible spines and spine assignments in increasing order of spine size. The corresponding spines together with the vertices occupying the occupied leaves are illustrated in \autoref{fig:3-pathOrbits}.
	
	There exists a single residual component with spine size one, namely $\Gamma_1$. The spine is given by $\Sp(\Gamma_1)=(1,4)$ and the vertices 1 and 4 are occupied by 3 and 2, respectively. For every vertex $v\in \mathcal{S}_G$, one has $c_{\Gamma_1}(v)=\Sp(\Gamma_1)$.
	
	Next, we analyse the connected components of network realignments in $Y_G$ with spine size two, denoted by $\Gamma^2$. Let $(u,\gamma_2, w)$ denote the spine of such a $\Gamma\in\Gamma^2$. Since the spine is minimal, $\gamma_2$ is a barrier of the vertices $\gamma_1$ and $\gamma_3$, which occupy $u$ and $w$, respectively, for any realignment in $\Gamma$. Hence, $(\gamma_1, \gamma_2, \gamma_3)$ is a subpath of $P$. Without loss of generality, we may assume $\gamma_1 \in \{1,4\}$. If $\gamma_1 = 1$, then $w\in [n]\setminus \{\gamma_1,\gamma_2,\gamma_3, 4\}$ and $u\in [n]\setminus \{\gamma_1,\gamma_2,\gamma_3, w\}$. The same holds with 1 and 4 interchanged. Now, the spine assignments are uniquely determined. Therefore, we enumerate
	\begin{equation*}
		\vert\Gamma^2\vert = 2(n-4)^2.
	\end{equation*}
	All connected components with $u \notin \{1,4\}$ belong to the same orbit, denoted $\Gamma^2_1$. Further, all connected components with $u \in \{1,4\}$ belong to a second orbit, denoted by $\Gamma^2_2$, since the $C_2$-action interchanges the subpaths $(1,2,3)$ and $(2,3,4)$ as well as $1$ and $4$.
	
	Finally, we analyse the connected components of network realignments in $Y_G$ with spine size greater than two. Note that no vertex is a barrier of more than two vertices. Hence, the only possibility for the spine to have all leaves occupied is to be path-shaped. We distinguish two classes of network realignments: those in which both vertices 2 and 3 lie in the spine, and those in which they do not.
	
	We further divide the first class into families of connected components according to spine size $i+2$, denoted $(\Gamma_i^3)$. A component $\Gamma$ lies in $\Gamma^3_i$ if $\Sp(\Gamma)=(x,\alpha,Q,\beta,y)$, where $x,y\in \mathcal{S}_G$ are occupied leaves, $\alpha,\beta\in V(P)$ satisfy $(\alpha,\beta)\in E(G)$ and $(\alpha,Q,\beta)$ is a path of length $i\in\{1,\dots, n-5\}$ whose interior vertices lie in $\mathcal{S}_G \setminus \{x,y\}$. Assuming that $x$ and $y$ are occupied leaves, the spine assignments of all vertices are uniquely determined. 
	
	To enumerate the components of $(\Gamma_i^3)$, we first choose the pair $(\alpha,\beta)$ from $\{(1,3),(1,4),(2,4)\}$, then one selects two vertices in $\mathcal{S}_G$ for $x$ and $y$ as well as $i-1$ further vertices from $\mathcal{S}_G$ to serve as the vertices of $Q$ and orders the latter along a path. Consequently,
	\begin{equation*}
		\left\vert\bigcup\limits_{i=1}^{n-5}\Gamma_i^3\right\vert = 3(n-4)(n-5)\left(\sum\limits_{k=0}^{n-6} \binom{n-6}{k}k! \right).
	\end{equation*}
	Since the action of $S_{n-4}$ is given by permuting vertices of $\mathcal{S}_G$, the orbit of a network realignment in $\Gamma_i^3$ depends solely on the values of $\alpha$ and $\beta$. $C_2$ acts on $G$ by swapping the edges $(1,3)$ and $(2,4)$ and fixing $(1,4)$. Consequently, there are $2(n-5)$ orbits under the action of $\Aut(G)$ for a choice of the pair $(\alpha,\beta)$ and each path length $i$.
	
	Similarly, we divide the second class into families of connected components according to spine size $i+2$, denoted $(\Gamma_i^4)$. For any $\Gamma\in\Gamma^4_i$, the spine is of the form $\Sp(\Gamma)=(x,2,\hat{Q},3,y)$, where $x,y\in \mathcal{S}_G$ are occupied leaves and $(2,\hat{Q},3)$ is a path of length $i\in\{2,\dots,n-5\}$ with $V(\hat{Q})\subseteq \mathcal{S}_G\setminus \{x,y\}$. Again, the spine assignments of all vertices are uniquely determined. To enumerate the components of $(\Gamma^4_i)$, one chooses $x$ and $y$ from $\mathcal{S}_G$. Next, one selects the $i-1$ vertices of $\hat{Q}$ and orders them along the path. Thus,
	\begin{equation*}
		\left\vert\bigcup\limits_{i=2}^{n-5} \Gamma_i^4\right\vert = (n-4)(n-5) \left( \sum\limits_{k=1}^{n-6} \binom{n-6}{k}k! \right).
	\end{equation*}
	Then there are $(n-6)$ orbits, one for each path length $i$, because $C_2$ interchanges the edges $(1,2)$ and $(3,4)$ and $S_{n-4}$ permutes $\mathcal{S}_G$. 
	
	Overall, we sum up and find
	\begin{equation*}
		\vert\Gamma^1\vert + \left\vert\bigcup\limits_{i=1}^2\Gamma_i^2\right\vert + \left\vert\bigcup\limits_{i=1}^{n-5}\Gamma_i^3\right\vert + \left\vert\bigcup\limits_{i=2}^{n-5}\Gamma_i^4\right\vert = f(n).
	\end{equation*} 
	Again, by \autoref{cptscontractible} the counted components admit an $\Aut(G)$-equivariant strong deformation retraction onto a discrete space and we sum up the number of orbits:
	\begin{equation*}
		1+2+2(n-5)+(n-6) =  3n-13.
	\end{equation*}
	The topology of the main component is given by \autoref{MainCptKn}.
	
	For $n = 5$, besides the main component, there is one connected component of type $\Gamma^1$ and two connected components of type $\Gamma^2$, which belong to the same orbit.
\end{proof}

The preceding case already illustrates the rapid growth of (orbits of) connected components as the number of omitted edges increases, as well as the difficulty of enumerating them. While the components in the present example are still characterised by path-shaped spines, this feature is no longer preserved once more than three edges are removed or when three deleted edges do not themselves form a path. The authors leave this as an open problem.

\section{Geometry of the Network Realignment Graph \texorpdfstring{$\mathcal{G}_n$}{Gn}} \label{sec:Geometry}

In the previous section we studied the equivariant topology of the network realignment complex $X_G$ over an arbitrary connected base graph $G$. We now specialise to the complete base graph $K_n$ and turn from topology to geometry. More precisely, instead of studying the homotopy type of the cubical complex $X_n$, we study the graph metric on its $1$-skeleton. This graph records the possible single leaf slides between labelled spanning trees, and its metric measures the minimal number of such slides needed to pass from one tree to another.

\begin{definition}
	The \textit{network realignment graph} $\mathcal{G}_n$ is the $1$-skeleton of the network realignment complex $X_n=X_{K_n}$, where $K_n$ has vertex set $[n]$. Equivalently, the vertices of $\mathcal{G}_n$ are the spanning trees of $K_n$, or labelled trees on $[n]$, and two vertices are adjacent if the corresponding trees differ by a single leaf slide.
\end{definition}

The study of the geometry of the network realignment graph is motivated by Kozlov's work on network realignment complexes. Kozlov asks in \cite{Kozlov25} whether one can give a formula for the diameter of $\mathcal{G}_n$. We recall the relevant notions.

\begin{definition}
	We write $d_{\mathcal{G}_n}$ for the graph distance on $\mathcal{G}_n$. Thus $d_{\mathcal{G}_n}(T_1,T_2)$ is the smallest number of leaf slides needed to transform $T_1$ into $T_2$. The \textit{diameter} of $\mathcal{G}_n$ is given by
	\[
	\operatorname{diam}(\mathcal{G}_n) = \max_{T_1,T_2\in V(\mathcal{G}_n)} d_{\mathcal{G}_n}(T_1,T_2).
	\]
\end{definition}

The main result of this section is a pair of explicit upper and lower bounds. They are summarised in the following theorem.

\begin{theorem}
	Let $n\geq 3$. The diameter of $\mathcal{G}_n$ satisfies 
	\[
	\left\lfloor \frac{3n^2-4n}{8} \right\rfloor \leq \operatorname{diam}(\mathcal{G}_n)\leq \frac{n^2}{2}-n,
	\]
	if $n$ is even and 
	\[
	\left\lfloor \frac{3n^2-6n}{8} \right\rfloor \leq \operatorname{diam}(\mathcal{G}_n)\leq \frac{n^2-1}{2}-n,
	\]
	if $n$ is odd.
\end{theorem}

The proof of the upper bound occupies the next three subsections. We first introduce two quantities, $\Phi_c$ and $\Psi$, which respectively give an upper and a lower estimate for the distance from a tree to a star tree. We then characterise the vertices for which these estimates agree in terms of centroids. The upper bound for the diameter is then obtained by choosing, for any pair of trees, a distance-minimising star tree for each of the two given trees and constructing a path that passes through both star trees.

The lower bound is proved in the last subsection. We introduce the matching distance, an auxiliary metric on $\mathcal{G}_n$, which bounds the graph distance from below, and determine the exact diameter of the network realignment graph with respect to this distance. This gives us the desired lower bound.

Computations suggest that the upper bound is sharp for all $n\geq 6$. We have verified that it agrees with the exact diameter for all $6\leq n\leq 9$. However, exact computations quickly become infeasible, because by Cayley's formula, $\mathcal{G}_n$ has $n^{n-2}$ vertices, and hence its size grows superexponentially in $n$.

Beyond the diameter bounds, the methods developed in this section also reveal additional structure in the geometry of the realignment graph. In particular, the characterisation of nearest star trees is an interesting structural consequences of the theory developed in this section. Some aspects of this metric rigidity will also play a role in the next section, where we determine the automorphism group of $X_n$.

\subsection{Distances to Star Trees}

The invariant $\Phi_c(T)$ depends on a chosen vertex $c$ and measures how far $T$ is from the star tree $s_c$ by summing the excess distances from $c$. The invariant $\Psi(T)$ measures, edge by edge, how balanced the cuts of $T$ are. 

The main point of this subsection is that $\Psi(T)$ gives a lower bound for the distance from $T$ to any star tree, while $\Phi_c(T)$ gives an upper bound for the distance from $T$ to the specific star tree $s_c$. In the next subsection, we determine exactly when these two bounds agree (cf. \autoref{centraleq}).

\begin{definition}
	Let $T\in V(\mathcal{G}_n)$ be a tree. Let $c\in [n]$ be any vertex of $T$. We define
	\[
	\Phi_c(T) = \sum_{x\in V(T)\setminus \{c\}} (d_{T}(x,c)-1).
	\]
	For every very edge $e$ of $T$, removing $e$ from $T$ leaves us with a forest consisting of exactly two trees $A_e$ and $B_e$. We define 
	\[\sigma_T(e) = \min(|V(A_e)|,|V(B_e)|)
	\]
	and 
	\[
	\Psi(T) = \sum_{e\in E(T)}(\sigma_T(e)-1).
	\]
\end{definition}

\begin{lemma}
	Let $T\in V(\mathcal{G}_n)$ be a spanning tree. Let $c\in [n]$ be any vertex of $T$. Then $T$ is a star tree if and only if $\Psi(T) = 0$. Further, $T$ is a star tree centered at $c$ if and only if $\Phi_c(T) = 0$.
\end{lemma}

\begin{proof}
	This follows immediately from the definitions.
\end{proof}

\begin{lemma}\label{iter}
	Let $N = (T,r)$ be a 1-dimensional realignment (i.e. leaf slide). We write $B = \{v\}$ and  $r(v) = (u,w)$. Let $c\in V(T) = [n]\setminus\{v\}$.

	In this situation, the leaf slide changes the values of $\Psi$ by at most one. To be precise, we have
	\begin{equation}\label{Psiiter}
		\begin{split}
			|\Psi(N^{v\to u})-\Psi(N^{v\to w})|&  = 1\text{, if $n$ is even, and}\\
			|\Psi(N^{v\to u})-\Psi(N^{v\to w})|&  \leq 1\text{, if $n$ is odd.}
		\end{split}
	\end{equation}
	The behaviour of $\Phi$ depends on whether the slide is performed towards or away from $c$. Let $\gamma$ be the unique path connecting $w$ and $c$ in $T$, then
	\begin{equation}\label{Phiiter}
		\Phi_c(N^{v\to u}) = 
		\begin{cases}
			\Phi_c(N^{v\to w}) - 1, &\text{if $u$ lies on $\gamma$,}\\
			\Phi_c(N^{v\to w}) + 1, &\text{otherwise.}\\
		\end{cases}
	\end{equation}
\end{lemma} 

\begin{remark}
	It is important to note that, if the leaf $v$ that is being slid is equal to $c$, then $\Phi_c$ may change by more than one under a leaf slide.
\end{remark}

\begin{proof}[Proof of \autoref{iter}]
	We first show (\ref{Psiiter}). Given any edge $e$ of any tree $T$, we want to compare the connected components $A_e$ and $B_e$ that are obtained by removing $e$ from $T$, for different choices of $T$. To keep the notation consistent while considering different trees, we fix an arbitrary orientation on the edges of $K_n$. Now, we can redefine $A_e^T$ to be the connected component of $T\setminus \{e\}$, that contains the source of $e$ and $B_e^T$ to be the connected component that contains the sink of $e$.

	To simplify notation, we write $T_u = T^{v\to u}$ and $T_w = T^{v\to w}$. Consider the edges $(v,u)$ in $T_u$ and $(v,w)$ in $T_w$. Since both of these edges are leaf edges, they do not contribute to $\Psi(T_u)$ and $\Psi(T_w)$, respectively.

	The other edges of $\Psi(T_u)$ are precisely those which the two trees have in common. Let $e$ be any of these shared edges other than $(u,w)$. Then the leaf slide clearly preserves the vertices in $A_e$ and $B_e$. Hence, we see that $\sigma_{T_u}(e) = \sigma_{T_w}(e)$. This shows that 
	\[
	\begin{split}
		|\Psi(T_u)-\Psi(T_w)| 	&= \left| \sum_{e\in E(T_u)}\sigma_{T_u}(e) - \sum_{e\in E(T_w)}\sigma_{T_w}(e)\right|\\
		&= |\sigma_{T_u}(u,w)-\sigma_{T_w}(u,w)|.
	\end{split}
	\]
	Without loss of generality, we may assume that the edge $e = (u,w)$ is oriented from $u$ to $w$. Hence,  we have $V(A_e^{T_w}) = V(A_e^{T_u})\setminus \{v\}$ and $V(B_e^{T_w}) = V(B_e^{T_u})\cup \{v\}$. It follows that 
	\[
	\begin{multlined}
		|\sigma_{T_u}(u,w)-\sigma_{T_w}(u,w)|  = \left|\min(|V(A_e^{T_u})|,|V(B_e^{T_u})|) - \min(|V(A_e^{T_w})|,|V(B_e^{T_w})|)\right|\\
		 = \left|\min(|V(A_e^{T_u})|,|V(B_e^{T_u})|) - \min(|V(A_e^{T_u})|-1,|V(B_e^{T_u})|+1)\right| \leq1
	\end{multlined}
	\]
	This proves the inequality in (\ref{Psiiter}). The equality for even $n$ can be seen as follows. 

	Let $a = |V(A_e^{T_u})|$, then $|V(B_e^{T_u})| = n-a$. Notice that since $n$ is even, both of $a$ and $n-a$ have the same parity and both of $a-1$ and $n-a+1$ have the opposite parity to that of $a$. This means that 
	\[
	|\Psi(T_u)-\Psi(T_w)| = |\min(a,n-a)-\min(a-1,n-a+1)|
	\]
	must be odd, which is only possible if it is equal to one. This concludes the proof of (\ref{Psiiter}).

	To prove (\ref{Phiiter}), it suffices to consider the case where $u$ lies on $\gamma$. The other case follows by reversing the roles of $u$ and $w$. Hence, our goal is to show that 
	\[
	\Phi_c(T_w) - \Phi_c(T_u) = 1
	\]
	under the assumption that $u$ lies on $\gamma$.

	Let $x$ be any vertex other than $v$. Since $v$ is a leaf in both $T_u$ and $T_w$, the unique path connecting $x$ to $c$ in either of the two trees does not contain $v$. Since $T_w \setminus \{v\} = T_u \setminus \{v\}$, this means that the two paths are identical, so $d_{T_u}(x,c) = d_{T_w}(x,c)$. This shows that
	\[
	\begin{split}
		\Phi_c(T_w) - \Phi_c(T_u) &= \sum_{x\in V(T_w)}d_{T_w}(x,c) - \sum_{x\in V(T_u)}d_{T_u}(x,c)\\
		&= d_{T_w}(v,c) - d_{T_u}(v,c).
	\end{split}
	\]
	Since $u$ lies on $\gamma$ and $w$ is connected to $u$, we know that in both trees $d(w,c) = d(u,c) +1$. Moreover, $d_{T_w}(v,c) = d_{T_w}(w,c) + 1$ and $d_{T_u}(v,c) = d_{T_u}(u,c) +1$, because $v$ is a leaf of $T_w$ (respectively $T_u$) that is attached to $w$ (respectively $u$). We conclude that 
	\[
	\begin{split}
		d_{T_w}(v,c) - d_{T_u}(v,c) &= d_{T_w}(w,c) + 1 - ( d_{T_u}(u,c) + 1)\\
		&= d_{T_u}(u,c) + 1 - d_{T_u}(u,c)\\
		&= 1,
	\end{split}
	\]
	which concludes the proof of (\ref{Phiiter}).
\end{proof}

\begin{proposition}
	Let $n\geq 2$. $\mathcal{G}_n$ is bipartite if and only if $n$ is even.
\end{proposition}

\begin{proof}
	Let $n$ be even. Consider the map 
	\[
	V(\mathcal{G}_n)\to \{\pm 1\}, \quad T\mapsto (-1)^{\Psi(T)}.
	\] 
	It is an immediate consequence of \autoref{iter} that this is a 2-coloring on $\mathcal{G}_n$.

	Now, let $n\geq 3$ be odd. We will construct a cycle of odd length to show that $\mathcal{G}_n$ is not bipartite.

	For any two distinct $a,b\in [n]$, let $\gamma_{a,b}$ be a path between $s_a$ and $s_b$ in $\mathcal{G}_n$ consisting of the $n-2$ leaf slides for each of which a vertex $x\in [n]\setminus \{a,b\}$ is slid along the edge $(a,b)$, in an arbitrary order.

	For three pairwise distinct points $a,b,c$, the concatenation $\gamma_{a,b} * \gamma_{b,c} * \gamma_{c,a}$ is a cycle of length $3(n-2)$, which is odd.
\end{proof}

\begin{lemma}\label{boundineq}
	Let $T\in V(\mathcal{G}_n)$ be a spanning tree. Let $c\in V(T)$ be any vertex of $T$. Then 
	\[
	\Psi(T) \leq d_{\mathcal{G}_n}(T, s_c) \leq \Phi_c(T).
	\]
\end{lemma}

\begin{proof}
	We first show the inequality $\Psi(T) \leq d_{\mathcal{G}_n}(T, s_c)$. Let $k = d_{\mathcal{G}_n}(T,s_c)$. Then there exists a sequence of $T = T_0,T_1,\dots,T_k = s_c$, such that each two consecutive trees are related via a single leaf slide. 

	Applying $\Psi$ to this sequence gives a sequence of integers $\Psi(T_0),\dots, \Psi(T_k)$, which starts at $\Psi(T)$ and ends at $\Psi(s_c) = 0$. Since any two consecutive trees are related by a leaf slide, it follows from \autoref{iter} that this sequence changes by at most one each step. Hence, the sequence must have length $k\geq \Psi(T)$.

	Now, suppose that $U$ is any spanning tree on $n$ vertices. We define a new spanning tree $f(U)$ as follows. If $U = s_c$, then $f(U) = U$. For any tree $U$ other than $s_c$, pick an arbitrary leaf $v$, which has distance at least $2$ from $c$. Let $(v,w,u ,\dots, c)$ be the unique path between $v$ and $c$ in $U$. We now define $f(U)$ to be the tree obtained by sliding the leaf $v$ along the edge $(w,u)$.

	Notice that in the case where $U\neq s_c$, the vertex $w$ is not equal to $c$, because $d_U(v,c)\geq 2$. Hence, $u$ lies on the unique path connecting $w$ and $c$ in $U$, so it follows from \autoref{iter} that $\Phi_c(f(U)) = \Phi_c(U) - 1$. 

	Applying $f$ repeatedly to our tree $T$ yields a sequence $T = T_0,\dots ,T_l = s_c$ starting at $T$ and ending at $s_c$, such that $T_{i+1} = f(T_i)$ for each $i$. In particular, consecutive trees in this sequence are related via a leaf slide, so the sequence describes a path in $\mathcal{G}_n$. This path must have length $\Phi_c(T)$, because $\Phi_c$ decreases by exactly one at each step and $\Phi_c(s_c) = 0$. We conclude that $d_{\mathcal{G}_n}(T,s_c)\leq \Phi_c(T)$.

\end{proof}

\begin{remark}
	Notice that this argument does not show that $d_{\mathcal{G}_n}(T,s_c) =  \Phi_c(T)$, because a leaf slide of $c$ may decrease $\Phi_c$ by more than one.

	In fact, equality in the above equation does not hold in general. A counterexample is obtained by considering a path tree $T$, and choosing $c$ to be one of its two leaves. For large $n$, it is more efficient to first slide the leaf to the middle of the path tree and then slide the rest of the tree towards $c$, instead of sliding the entire tree towards $c$ immediately.
\end{remark}

\subsection{Centroids and Nearest Star Trees}

The above counterexample depends on $T\setminus \{c\}$ having a large component. Sliding $c$ toward this component decreases the distance to each of its vertices by one while increasing the distance to every other vertex by one. If the component contains more than half of the vertices, $\Phi_c(T)$ decreases by more than one. Conversely, if no such component exists, then every leaf slide of $c$ increases $\Phi_c(T)$. This observation yields a criterion for equality in the preceding inequality.

\begin{definition}
	Let $T\in V(\mathcal{G}_n)$ be a spanning tree and let $c\in [n]$. Removing $c$
	from $T$ leaves a forest. We define
	\[
	F_c(T)
	=
	\max\{|V(C)| \mid C \text{ is a connected component of } T\setminus \{c\}\}.
	\]
\end{definition}
	Recall from \autoref{def:centroid}, that $c$ is a centroid of $T$ if and only if 
	\[
	F_c(T)\leq \frac n2.
	\]

\begin{lemma}\label{PhiPsiDifference}
	Let $T\in V(\mathcal{G}_n)$ be a spanning tree and let $c\in V(T)$. For each edge $e\in E(T)$, let $A_e^c$ be the connected component of $T\setminus \{e\}$ which does not contain $c$, and write $a_e^c = |V(A_e^c)|$. Then
	\[
	\Phi_c(T)-\Psi(T) =	\sum_{e\in E(T)}\max(0,2a_e^c-n).
	\]
	In particular, $\Phi_c(T)=\Psi(T)$ if and only if $c$ is a centroid of $T$.
\end{lemma}

\begin{proof}
    We first show that 
    \[
    \Phi_c(T) = \sum_{e\in E(T)}(a_e^c-1).
    \]
    Since $T$ has $n-1$ edges, this equality is equivalent to 
    \[
    \sum_{e\in E(T)}a_e^c(T) = \sum_{x\in V(T)} d_T(x,c).
    \]
    Now, consider any vertex $x\neq c$. For any given $e\in E(T)$, $x$ lies in $A_e^c$ if and only if $e$ lies on the unique path connecting $x$ and $c$. Hence, the vertex $x$ contributes one to the left hand side for each edge on the unique path between $x$ and $c$. This path contains exactly $d_T(x,c)$ edges, so the equality follows.

    The two connected components of $T\setminus \{e\}$ have sizes $a_e^c$ and $n- a_e^c$, so $\sigma_T(e) = \min(a_e^c, n- a_e^c)$. Thus
    \[
    \Phi_c(T)-\Psi(T) = \sum_{e\in E(T)} \left(a_e^c(T) - \min(a_e^c, n- a_e^c) \right)= \sum_{e\in E(T)}\max(0,2a_e^c-n).
    \]
    It remains to characterise when the right hand side vanishes. The connected	components of $T\setminus \{c\}$ are exactly the subtrees $A_e^c$ for those	edges $e$ incident to $c$. Every other subtree $A_e^c$ is contained in one of these components. Hence
	\[
	F_c(T)=\max_{e\in E(T)}a_e^c.
	\]
	Therefore, $\Phi_c(T)=\Psi(T)$ if and only if $a_e^c\leq n/2$ for every edge $e$. This is equivalent to $F_c(T) \leq n/2$, that is, $c$ is a centroid. 
\end{proof}

\begin{lemma}\label{auxiliary}
	Let $c\in [n]$. If $T_1$ and $T_2$ are related by a leaf slide and
	$F_c(T_1)>\frac{n}{2}$ but $F_c(T_2)< F_c(T_1)$, then
	\[
	    \Psi(T_2) \geq \Psi(T_1).
	\]
\end{lemma}

\begin{proof}
	Let $T_2$ be obtained from $T_1$ by sliding the leaf edge
	$(v,u)\in E(T_1)$ along the edge $e=(u,w)$. Since
	$F_c(T_1)>n/2$, there is a unique component $C$ of
	$T_1\setminus\{c\}$ with $|V(C)|>n/2$.

	The only way for the leaf slide to decrease $F_c$ is for the leaf
	$v$ to be reattached from the component $C$ directly to $c$. Hence,
	$w=c$ and $u\in V(C)$.

	As in the proof of \autoref{iter}, the only edge whose contribution
	to $\Psi$ changes is $e=(u,c)$. Set $q=|V(C)|$. Then
	\[
	    \sigma_{T_1}(e)=n-q.
	\]
	After the slide, the two components of $T_2\setminus\{e\}$ have
	sizes $q-1$ and $n-q+1$. Therefore,
	\[
	    \sigma_{T_2}(e)=\min(q-1,n-q+1)\geq n-q=\sigma_{T_1}(e).
	\]
	Hence, $\Psi(T_2)\geq \Psi(T_1)$.
\end{proof}

\begin{proposition}\label{centraleq}
    Let $T\in V(\mathcal{G}_n)$ be a spanning tree and let $c\in V(T)$. Then the following are equivalent.
	\begin{enumerate}
		\item $c$ is a centroid of $T$. \label{central}
		\item $\Psi(T) = d_{\mathcal{G}_n}(T, s_c) = \Phi_c(T)$. \label{equalities}
		\item $s_c$ has minimal distance from $T$ in $\mathcal{G}_n$ among all star trees. \label{minstar}
	\end{enumerate}
\end{proposition}

\begin{proof}
    The equivalence of $(\ref{equalities})$ and $(\ref{central})$ follows immediately from \autoref{PhiPsiDifference} and \autoref{boundineq}.

	Suppose $(\ref{equalities})$ holds. By \autoref{boundineq}, $\Psi(T)\leq d_{\mathcal{G}_n}(T,s_a)$ for each star tree $s_a$. On the other hand, $(\ref{equalities})$ states that $d_{\mathcal{G}_n}(T,s_c)=\Psi(T)$. Thus, $s_c$ has minimal distance from $T$ among all star trees. This proves $(\ref{equalities})\Rightarrow(\ref{minstar})$.

	It remains to prove $(\ref{minstar})\Rightarrow(\ref{central})$. Choose a centroid $d$ of $T$, which exists by \autoref{existenceofcentral}. By the already proved implication $(\ref{central})\Rightarrow(\ref{equalities})$, we have $d_{\mathcal{G}_n}(T,s_d)=\Psi(T)$. \autoref{boundineq} gives $\Psi(T)$ as a lower bound for the distance from $T$ to any star tree. Hence, every distance-minimising star tree has distance exactly $\Psi(T)$ from $T$.

	Now, assume that $s_c$ is distance-minimising, and suppose that $c$ is not a centroid of $T$. Then
	\[
	F_c(T)> \frac{n}{2}.
	\]
	Let	$T=T_0,T_1,\dots,T_k=s_c$ be a shortest path from $T$ to $s_c$ in $\mathcal{G}_n$. Since $s_c$ is distance-minimising, we have $k=d_{\mathcal{G}_n}(T,s_c)=\Psi(T)$.

	By \autoref{iter}, the value of $\Psi$ changes by at most one under each leaf slide. Since $\Psi(T_0)=\Psi(T)$ and $\Psi(T_k)=\Psi(s_c)=0$, and since	the path has length $\Psi(T)$, the value of $\Psi$ must decrease by exactly one	at each step.

	Conversely, $F_c(T_0)> \frac{n}{2}$ and $F_c(T_k)=F_c(s_c)=1\leq \frac{n}{2}$. Hence, there exists an index $i$ such that $F_c(T_i)> \frac n2$ and $F_c(T_{i+1})<F_c(T_i)$.
	By \autoref{auxiliary}, this implies that
	\[
	\Psi(T_{i+1})\geq\Psi(T_i),
	\]
	contradicting the fact that $\Psi$ decreases by exactly one at each step. Therefore, $c$ is a centroid of $T$. 
\end{proof}

\subsection{An Upper Bound for the Diameter}

The previous analysis showed that it suffices to study the quantity $\Phi_c(T)$ for centroids $c$ of $T$ to determine the distance to the closest star tree. To bound $\Phi_c(T)$ from above, we introduce two optimisation problems: the first one analyses the structure of the connected components of $T\setminus \{c\}$, and the second one enumerates their number. This upper bound allows us to bound the distance between any two network realignments by passing through their respective closest star trees.

\begin{lemma}\label{linearphibound}
	For any vertex $c$ of $T$, we have
	\[
	\Phi_c(T) \leq \frac{(n-2)(n-1)}{2}.
	\]
	Equality holds if and only if $T$ is a path tree and $c$ is a leaf of $T$.
\end{lemma}

\begin{proof}
	For a fixed $n\in \mathbb{N}$, consider the following optimisation problem.
	\begin{alignat*}{4}
		\max_{\alpha_1,\alpha_2,\dots}	&\qquad	&&\sum_{i\geq 1} (i-1) \alpha_i  &\qquad&&\\
		\text{subject to:} 				&\qquad	&&\alpha_i\in \mathbb{Z}_{\geq 0} &\qquad&&\forall i\geq 1\\
		&\qquad	&&\sum_{i\geq 1}\alpha_i = (n-1) &\qquad&&\\
		&\qquad	&& \alpha_i = 0 \Rightarrow \alpha_{i+1} = 0 &\qquad&&\forall i\geq 1
	\end{alignat*}
	Solutions $(\alpha_i)_{i\geq 1}$ to this optimisation problem can be interpreted as the number of vertices  of distance $i$ from the root $c$ in a rooted tree $T$ with $n$ vertices. In this situation, we get
	\[
		\Phi_c(T) = \sum_{i\geq 1} (i-1) \alpha_i.
	\]
	Suppose that $(\alpha_i)_{i\geq 1}$ is a solution to this problem with $\alpha_k>1$ for some $k$. The constraints of the problem force there to be a $j\geq 1$ for which $\alpha_j \neq 0$ and $\alpha_i = 0$ for all $i>j$. From the given solution we obtain a new solution by replacing $\alpha_k$ with $1$ and $\alpha_{j+1}$ with $\alpha_k-1$. This is a strictly better solution.

	Hence, any optimal solution is one which contains only ones and zeroes. But the only such solution is $\alpha_1 = \dots = \alpha_{n-1} = 1$ and $\alpha_i = 0$ for $i\geq n$. This corresponds to path trees rooted at one of their leaves. 
	
	The value attained by this unique optimal solution is
	\[
	\sum_{i= 1}^{n-1} (i-1) = \frac{(n-2)(n-1)}{2}.
	\]
	This proves the assertion.
\end{proof}

\begin{lemma}\label{phibound}
	Let $c$ be a centroid of $T$. Then 
	\[
	\Phi_c(T) \leq \left\lfloor \frac{n^2}{4}\right\rfloor-n+1.
	\]
	Equality holds if and only if $T$ is a path tree.
\end{lemma}

\begin{proof}
	The assertion is immediate for $n=2$, so assume $n\geq 3$. We first rewrite $\Phi_c(T)$ in terms of the connected components of $T\setminus \{c\}$. Let $v_1,\ldots, v_k$ be the neighbours of $c$, and let $V_i$ denote the connected component of $T\setminus\{c\}$ that contains~$v_i$. For any vertex $x\in V_i$, the unique path from $x$ to $c$ in $T$ must contain $v_i$. Thus, $d_T(x,c) = d_{V_i}(x,v_i) + 1$. It follows that
	\begin{align*}
		\Phi_c(T) &= \sum_{x\in V(T)\setminus\{c\}} (d_T(x,c) - 1) \\
			&= \sum_{i=1}^k \sum_{x\in V_i} d_{V_i}(x,v_i) \\
			&= n-k-1 + \sum_{i=1}^k \sum_{x\in V_i\setminus \{v_i\}} (d_{V_i}(x,v_i) - 1) \\
			&= n-k-1 + \sum_{i=1}^k \Phi_{v_i} (V_i).
	\end{align*}
	By \autoref{linearphibound}, we get
	\[
		\Phi_{v_i} (V_i) \leq \frac{(d_i-1) (d_i-2)}{2},
	\]
	where $d_i$ is the number of vertices of $V_i$. 
	
	Now, we consider another optimisation problem to analyse the number and structure of the connected components that maximise $\Phi_c(T)$. For $d = (d_i)_{i\geq 1}$, we define
	\[
		K(d) = \sum_{i\geq 1} \frac{d_i (d_i-1)}{2} = \sum_{i\geq 1} \left(\frac{(d_i-1) (d_i-2)}{2} + d_i-1 \right).
	\]
	The shift ensures that indices with $d_i=0$ do not contribute to $K(d)$, allowing to sum over all $i\geq 1$. For a fixed $n\in \mathbb{N}$, consider the following optimisation problem.
	\begin{alignat*}{4}
		\max_{d = (d_1,d_2,\dots)}	&\qquad	&& K(d) &\qquad&&\\
		\text{subject to:} 				&\qquad	&&d_i\in \mathbb{Z}_{\geq 0} &\qquad&&\forall i\geq 1\\
		&\qquad	&&d_i\leq \frac{n}{2} &\qquad&&\forall i\geq 1\\
		&\qquad	&& d_i\leq d_{i-1} &\qquad&&\forall i\geq 2\\
		&\qquad	&&\sum_{i\geq 1}d_i = (n-1) &\qquad&&\\
	\end{alignat*}
	Note that the non-zero $d_i$'s of a solution can be interpreted as the sizes of the connected components $V_i$ of $T\setminus\{c\}$ of a rooted tree $(T,c)$, ordered decreasingly. In this situation, we get
	\[	
		K(d) \geq \Phi_c(T) 
	\]
	using \autoref{linearphibound}. Equality holds if and only if all $V_i$'s are path trees. Consequently, to maximise $\Phi_c(T)$, each connected component of $T\setminus\{c\}$ must be a path tree.

	We claim that the unique optimal solution to this problem is given by $d_1 = \lfloor n/2 \rfloor$, $d_2 = (n-1)- d_1$, and $d_i = 0$ for $i\geq 3$. This corresponds to trees for which $T\setminus \{c\}$ has only two connected components. Consequently, the value of $\Phi_c(T)$ is maximal if and only if $T$ is a path tree. The value of the unique optimal solution is 
	\[
	K(d) = \frac{1}{2}\left\lfloor \frac{n}{2} \right\rfloor\left( \left\lfloor \frac{n}{2} \right\rfloor -1 \right) +\frac{1}{2} \left(n- \left\lfloor \frac{n}{2} \right\rfloor-1\right) \left(n-\left\lfloor \frac{n}{2} \right\rfloor -2\right)  = \left\lfloor\frac{n^2}{4}\right\rfloor -n+1.
	\]

	Let $d$ be a solution that is not of the claimed form. Then there exist indices $i<j$ with $d_i<\lfloor n/2\rfloor$ and $d_j>0$. Choose such $i,j$ with $i<j$, and move one unit from $d_j$ to $d_i$. After reordering, the resulting sequence again satisfies the constraints, and the value of $K$ increases by $d_i-d_j+1\geq 1$. Hence, $d$ is not optimal.
\end{proof}

\begin{lemma}\label{stardist}
	Let $a,b\in [n]$ be two distinct numbers. Then $d_{\mathcal{G}_n}(s_a,s_b) = n-2$.
\end{lemma}

\begin{proof}
	A quick calculation shows that $\Phi_{a}(s_b) = n-2$. Thus by \autoref{boundineq} we have $d_{\mathcal{G}_n}(s_a,s_b) \leq n-2$.

	Conversely, since every $x\in [n]\setminus\{a,b\}$ defines a leaf edge $(x,a)$ in $s_a$ and none of these edges exist in $s_b$, any sequence of leaf slides connecting the two star trees must slide the edge $(x,a)$ at least once. Hence, $d_{\mathcal{G}_n}(s_a,s_b) \geq n-2$.
\end{proof}

\begin{theorem}
	\label{upper diam bound}
	For any $n\geq2$ the diameter of $\mathcal{G}_n$ is bounded from above by $\left\lfloor \frac{n^2}{2}\right\rfloor -n$.
\end{theorem}

\begin{proof}
	Let $T_1$ and $T_2$ be arbitrary spanning trees. By \autoref{existenceofcentral} there exists a centroid $c_1\in V(T_1)$ (respectively $c_2\in V(T_2)$). Using Lemmas \ref{boundineq}, \ref{phibound}, \ref{stardist} and the triangle inequality, we find that
	\[
	\begin{split}
		d_{\mathcal{G}_n}(T_1,T_2)	&\leq d_{\mathcal{G}_n}(T_1,s_{c_1}) + d_{\mathcal{G}_n}(s_{c_1},s_{c_2}) + d_{\mathcal{G}_n}(s_{c_2},T_2)\\
		&\leq \Phi_{c_1}(T_1) + (n-2) + \Phi_{c_2}(T_2)\\
		&\leq 2\left(\left\lfloor \frac{n^2}{4}\right\rfloor-n+1\right)+ n-2\\
		&= \left\lfloor \frac{n^2}{2}\right\rfloor -n.
	\end{split}
	\] 
	This concludes the proof.
\end{proof}

\subsection{Matching Distance and Lower Bounds}

For the lower bound, we introduce an auxiliary metric that bounds the graph distance from below. Unlike the graph distance, this metric does not depend on the structure of $\mathcal{G}_n$. Instead, it solely compares distances between certain vertices within the considered trees. In this metric, the distance between certain network realignments can be calculated explicitly, yielding a lower bound.

\begin{definition}
	A matching $\mathcal{M}$ on $[n]$ is a set of disjoint 2-element subsets of $[n]$. Given a \textit{matching} $\mathcal{M}$ on $[n]$, we associate to $\mathcal{M}$ a pseudometric $\delta_\mathcal{M}$ on the network realignment graph $\mathcal{G}_n$ via
	\[
	\delta_\mathcal{M}(T_1,T_2) = \sum_{\{i,j\}\in \mathcal{M}} |d_{T_1}(i,j)-d_{T_2}(i,j)|.
	\] 
	The \textit{matching distance} of two trees is defined to be
	\[
	\delta(T_1,T_2) = \max_{\mathcal{M} \text{ matching on }[n]} \delta_\mathcal{M}(T_1,T_2).
	\]
\end{definition}

\begin{lemma}\label{lem:matchingdist}
	The matching distance $\delta$ is a metric and 
	\[
	\delta(T_1,T_2)\leq d_{\mathcal{G}_n}(T_1,T_2).
	\]
	In particular, $\operatorname{diam}_\delta(\mathcal{G}_n)\leq \operatorname{diam}(\mathcal{G}_n)$, where $\operatorname{diam}_\delta(\mathcal{G}_n)$ denotes the diameter of $\mathcal{G}_n$ with respect to the matching distance.
\end{lemma}

\begin{proof}
	To see that $\delta$ is a metric, observe that for any two distinct trees $T_1$ and $T_2$, there must exist vertices $i,j$, such that $d_{T_1}(i,j)\neq d_{T_2}(i,j)$. Hence, if $\mathcal{N}$ is any matching containing the pair $\{i,j\}$, then $\delta$ satisfies
	\[
	1\leq \delta_\mathcal{N}(T_1,T_2)\leq \delta(T_1,T_2).
	\]
	Thus,
	$\delta(T_1,T_2) = 0$ if and only if $T_1 = T_2$. The other axioms of a metric are obviously satisfied by $\delta$.

	To see that for any two trees $\delta(T_1,T_2)\leq d_{\mathcal{G}_n}(T_1,T_2)$, it suffices to observe that $\delta(T_2,T_2) = 0$ and that the value of $\delta(T_1,T_2)$ changes by at most one, when a leaf slide is applied to $T_1$. The proof of this is identical to the proof that $\Phi_c$ changes by at most one under a leaf slide, despite sliding the chosen vertex $c$, in \autoref{iter}.
\end{proof}

\begin{lemma}\label{lem:bestmatching}
	Let $\mathcal{M}$ be a matching on $[n]$, then for any tree $T\in V(\mathcal{G}_n)$ we have
	\[
	\sum_{\{i,j\}\in \mathcal{M}} d_T(i,j) \leq na-a^2, 
	\]
	where $a = |\mathcal{M}|$.
\end{lemma}

\begin{proof}
	For an edge $e$ of $T$, we denote by $\tau_\mathcal{M}(e)$ the number of pairs $\{i,j\}$ in $\mathcal{M}$ for which $e$ lies on the unique path connecting $i$ and $j$ in $T$. Then
	\[
	\sum_{\{i,j\}\in \mathcal{M}} d_T(i,j) = \sum_{e\in E(T)}\tau_\mathcal{M}(e).
	\] 
	Further, we denote by $A_e$ and $B_e$ the connected components of $T\setminus \{e\}$. If $e$ is an edge of $T$ and $\{i,j\}\in \mathcal{M}$, then $e$ lies on the unique path connecting $i$ and $j$ if and only if $i$ and $j$ lie in different connected components of $T\setminus \{e\}$. In particular, $\tau_\mathcal{M}(e) \leq \min(|V(A_e)|,|V(B_e)|)$. Additionally, it follows immediately from the definition of $\tau_\mathcal{M}$ that $\tau_\mathcal{M}(e)\leq a$. Thus,
	\[
	\begin{split}
	\sum_{e\in E(T)}\tau_\mathcal{M}(e) &= \sum_{k = 1}^{a} |\{e\in E(T)\mid \tau_\mathcal{M}(e)\geq k\}|\\
	&\leq \sum_{k = 1}^{a}|\{e\in E(T)\mid \min(|V(A_e)|,|V(B_e)|)\geq k\}|.\\
	\end{split}
	\]
	Since 
	\[
	\sum_{k = 1}^{a} n-2k +1 = na-a^2,
	\]
	all that is left to show is that for any $k$ there are at most $n-2k + 1$ edges $e$ in $T$ for which $\min(|V(A_e)|,|V(B_e)|)\geq k$. 

	To see this, notice first that the set of edges satisfying this condition form a connected subgraph $H_k$ of $T$. If this subgraph is empty or consists of a single edge, then there is nothing to prove. Otherwise, $H_k$ is a tree with at least two vertices.

	Let $e_1$ and $e_2$ be two distinct leaf edges of $H_k$. Let $i\in \{1,2\}$. Without loss of generality, we may assume that $H_k \setminus \{e_i\}\subseteq A_{e_i}$. Further, since $e_i$ lies in $H_k$, we have $|B_{e_i}|\geq k$.

	We conclude that $H_k\setminus \{e_1,e_2\} \subseteq T\setminus (B_{e_1}\cup B_{e_2})$, and thus $|V(H_k)|-2 \leq n-2k $, because $B_{e_1}\cap B_{e_2}$ is empty. But this means that there are at most $n-2k + 1$ edges in $H_k$, which concludes the proof.
\end{proof}

\begin{proposition}\label{prop:deltadiam}
	The diameter of $\mathcal{G}_n$ with respect to the matching distance is
	\[
	\operatorname{diam}_{\delta}(\mathcal{G}_n)=
	\begin{cases}
		6k^2-2k, & n=4k,\\
		6k^2, & n=4k+1,\\
		6k^2+4k, & n=4k+2,\\
		6k^2+6k+1, & n=4k+3.
	\end{cases}
	\]
\end{proposition}

\begin{proof}
	We give the proof in the case $n=4k$, where the notation is simplest. The other congruence classes are obtained by the same argument, replacing the two blocks of $k$ matching pairs by blocks of sizes
	\[
	p=\left\lfloor \frac{\lfloor n/2\rfloor}{2}\right\rfloor,
	\qquad
	q=\left\lceil \frac{\lfloor n/2\rfloor}{2}\right\rceil,
	\]
	and, when \(n\) is odd, leaving one vertex unmatched. This gives the formula above. Let $n = 4k$. We give a matching $\mathcal{M}$, as well as trees $T_1$ and $T_2$, such that 
	\[
	\delta_\mathcal{M}(T_1,T_2) = 6k^2-2k.
	\]
	Let $L$ be the path tree with labels $(1,\dots,4k)$, and let $S$ be the path tree with labels
	\[
	(k+1,k+3,\dots,3k-1,1,4k,2,4k-1,\dots,k,3k+1,3k,3k-2,\dots,k+4,k+2).
	\]
	To maximise the distance between $L$ and $S$, consider the matching $\mathcal{M} = \mathcal{M}_1\cup\mathcal{M}_2$, where 
	\[
	\mathcal{M}_1 = \left\{\{i,4k-i+1\} \mid i = 1,\dots,k\right\}
	\]
	and 
	\[
	\mathcal{M}_2 = \{\{k+2i-1,k+2i\} \mid i = 1,\dots,k\}.
	\]
	We get
	\begin{align*}
		\delta_\mathcal{M}(L,S) &= \sum_{\{i,j\}\in\mathcal{M}_1}|d_L(i,j)-d_S(i,j)| + \sum_{\{i,j\}\in\mathcal{M}_2}|d_L(i,j)-d_S(i,j)|\\
		&= 2\sum_{i = 1}^{k} 4k-2i \\
		&= 6k^2 - 2k 
	\end{align*}
	We now show that this lower bound for $\operatorname{diam}_\delta(\mathcal{G}_n)$ is also an upper bound. For any matching $\mathcal{M}$ and arbitrary trees $T_1,T_2$, let 
	\begin{align*}
		\mathcal{M}_1 &= \{\{i,j\}\in \mathcal{M} \mid d_{T_1}(i,j)\leq d_{T_2}(i,j)\},\\
		\mathcal{M}_2 &= \{\{i,j\}\in \mathcal{M} \mid d_{T_1}(i,j)> d_{T_2}(i,j)\}.
	\end{align*}
	We write $t_i = |\mathcal{M}_i|$ for $i = 1,2$. Since $\mathcal{M} = \mathcal{M}_1\cup\mathcal{M}_2$, we have $t_1 + t_2 \leq 2k$. Using \autoref{lem:bestmatching}, we get
	\begin{align*}
		\delta_\mathcal{M}(T_1,T_2) &= \sum_{\{i,j\}\in\mathcal{M}_1}d_{T_2}(i,j)-d_{T_1}(i,j) + \sum_{\{i,j\}\in\mathcal{M}_2}d_{T_1}(i,j)-d_{T_2}(i,j)\\
		&\leq  4kt_2-t_2^2 - t_1 + 4kt_1-t_1^2 - t_2\\
		&= (4k-1)(t_1+t_2) - (t_1^2+t_2^2).  
	\end{align*}
	An elementary calculation shows that, under the assumption $k\geq 1$ and $t_1 + t_2 \leq 2k$, we have
	\[
	(4k-1)(t_1+t_2) - (t_1^2+t_2^2) \leq 6k^2-2k
	\]
	It follows that $\delta_\mathcal{M}(T_1,T_2)\leq 6k^2-2k$, which proves the assertion.
\end{proof}

\begin{corollary}\label{cor:lowerbound}
	For all $n\geq 3$ we have
	\[
	\operatorname{diam}(\mathcal{G}_n) \geq \operatorname{diam}_{\delta}(\mathcal{G}_n)=
	\begin{cases}
		\displaystyle
		\left\lfloor \frac{3n^2-4n}{8} \right\rfloor,
		& \text{if $n$ is even},\\[3mm]
		\displaystyle
		\left\lfloor \frac{3n^2-6n}{8} \right\rfloor,
		& \text{if $n$ is odd}.
	\end{cases}
	\]
\end{corollary}

\section{Automorphisms of \texorpdfstring{$X_n$}{Xn}}\label{sec:automorphisms}

		Each permutation in the symmetric group $S_n$ induces an automorphism of $X_n$ by relabelling the vertices of each network realignment. In this chapter, we prove that, for $n\geq 5$, every automorphism of $X_n$ arises in this way. To achieve this, we first show that $X_n$ is fully determined by its 1-skeleton. We then identify an $\Aut(X_n)$-invariant subcomplex $H$ whose automorphism group is the symmetric group, and prove that the natural map
		\[
			\Aut (X_n) \hookrightarrow \Aut(H)
		\]
		given by restriction is injective. The proof relies on some of the quantities introduced in \autoref{sec:Geometry}.

		The condition $n\geq 5$ is necessary. Indeed, $X_4 = \Sigma_{K_4}$, and thus
		\[
			\Aut(X_4) = \Aut(\Sigma_{K_4}) = S_2^6 \rtimes S_4.
		\]
		An explanation of the second equality is given in \autoref{sec:RigidSubcomplex}. Although $\Aut(X_3) = \Aut(K_3) = S_3$, the argument for $n\geq5$ does not apply in this case.

	\subsection{\texorpdfstring{$X_n$}{Xn} is a Cubical Flag Complex}

		First, we show that $X_n$ is determined by its 1-skeleton. Consequently, every automorphism is uniquely determined by its restriction to the 1-skeleton, implying
		\[
			\Aut(X_n) \cong \Aut(\mathcal{G}_n).
		\]
		Let $I$ denote an interval and $I^k$ a $k$-cube. We write $(I^k)_1$ for its 1-skeleton and $\partial I^k$ for its boundary.

		\begin{definition}
			A cubical complex $K$ is called \textit{cubical flag complex}, or determined by its 1-skeleton, if any injective cellular map $(I^k)_1 \to K$ admits a unique extension $I^k\to K$, meaning that the diagram
			\begin{center}
				\begin{tikzcd}
					I^k \arrow[r, dashed, "\exists!"] & K \\
					(I^k)_1 \arrow[ur, hook] \arrow[u, hook]
				\end{tikzcd}
			\end{center}
			admits a unique extension for all $k\geq2$.
		\end{definition}

		Note that the aforementioned property is equivalent to requiring that any injective cellular map $\partial I^k \to K$ admits a unique extension $I^k \to K$ for all $k\geq 2$. This easily follows from an inductive argument.

		\begin{lemma} \label{lem:FourCycles}
			Any injective cellular map $\partial I^2 \to X_n$ admits a unique extension $I^2 \to X_n$. In particular, 4-cycles in $\mathcal{G}_n$ are uniquely determined by any three of its vertices.
		\end{lemma}

		\begin{proof}
			An injective cellular map $\partial I^2\to X_n$ corresponds to a 4-cycle $N_1,N_2,N_3,N_4, N_1$ in $\mathcal{G}_n$. We take all indices modulo 4. Let $v_i$ be the vertex that is realigned along the edge from $N_i$ to $N_{i+1}$. Set $u_i = p_{T_i}(v_i)$ and $w_i = p_{T_{i+1}} (v_i)$. We have to determine the unique 2-cell with vertices $N_1,N_2,N_3,N_4$.

			We claim that $v_i \neq v_{i+1}$ for all $i$. Suppose, to the contrary, that $v_k = v_{k+1}$ for some $k$. Then 
			\[
				w_k = p_{T_{k+1}}(v_k) = p_{T_{k+1}}(v_{k+1}) = u_{k+1}.
			\] 
			Since $N_k \neq N_{k+2}$, we have $u_k \neq w_{k+1}$. Moreover, as $T_k\setminus\{v_k\} = T_{k+2}\setminus \{v_k\}$, it follows that $d_{T_k}(u_k, w_{k+1}) = 2$. Consequently, the unique geodesic between $N_k$ and $N_{k+2}$ is $N_k, N_{k+1}, N_{k+2}$, contradicting that $N_1,N_2,N_3,N_4,N_1$ is a 4-cycle. Therefore, $v_i\neq v_{i+1}$ for all $i$.

			Since $p_{T_1}(v_1) = u_1 \neq w_1 = p_{T_2}(v_1)$, the vertex $v_1$ has to be realigned along the path $N_1, N_4, N_3, N_2$. Hence, $v_1$ has to be realigned twice along the 4-cycle $N_1,N_2,N_3,N_4, N_1$. Because $v_1\neq v_2,v_4$, we must have $v_1 = v_3, \> u_1 = w_3$, and $w_1 = u_3$. Similarly, $v_2 = v_4, \> u_2 = w_4$ and $w_2 = u_4$. It is straightforward to verify that $v_i \neq u_j,w_j$ for all $i,j$. Thus, 
			\[
				(N_1^{v_1\to (u_1,w_1)})^{v_2\to (u_2,w_2)}
			\] 
			is well-defined, and its four vertices coincide exactly with the realignments $N_1,N_2,N_3$, and $N_4$. In particular, $v_1,v_2,u_1,u_2,w_1$, and $w_2$ are uniquely determined by the realignments $N_1,N_2,N_3$. Hence, these three realignments already fully determine the 2-dimensional realignment, and hence the 4-cycle.

			Since the intersection of any two cells in a cubical complex is empty or a common face, a 2-cell is uniquely determined by any three of its vertices. This carries over to 4-cycles, as each 4-cycle is the boundary of a 2-cell.
		\end{proof}

		\begin{lemma} \label{lem:FlagComplex}
			$X_n$ is a cubical flag complex.
		\end{lemma}

		\begin{proof}
			Let $\iota:(I^k)_1\hookrightarrow X_n$ be an injective cellular map. We identify the vertices of $(I^k)_1$ with subsets of $[k]$, and write $N_S=\iota(S)$ for $S\subseteq [k]$.

			We first show that the leaf being slid depends only on the coordinate direction. Consider the geodesics from $\emptyset$ to $[k]$ in $(I^k)_1$. These geodesics are in bijection with permutations of $[k]$, by recording the order in which the coordinate directions are traversed. If $\sigma\in S_k$, let
			\[
			\emptyset=S_0^\sigma,S_1^\sigma,\ldots,S_k^\sigma=[k]
			\]
			be the corresponding geodesic, where $S_j^\sigma=\{\sigma(1),\ldots,\sigma(j)\}$. Let $v_j^\sigma$ be the leaf which is slid along the edge from $N_{S_{j-1}^\sigma}$ to $N_{S_j^\sigma}$.

			We claim that for any two permutations $\sigma_1,\sigma_2\in S_k$ and any $i\in[k]$ one has
			\[
			v_i^{\sigma_1}=v_{\sigma_2^{-1}\sigma_1(i)}^{\sigma_2}.
			\]
			We first prove this when $\sigma_2$ is obtained from $\sigma_1$ by interchanging two adjacent entries. Suppose that $\sigma_2$ is obtained from $\sigma_1$ by interchanging the entries in positions $j$ and $j+1$. Then the two geodesics agree except on the square face spanned by the coordinate directions $\sigma_1(j)$ and $\sigma_1(j+1)$. Since $\iota$ is injective, the image of this square is an injective $4$-cycle in $\mathcal{G}_n$. By \autoref{lem:FourCycles}, opposite edges in this square correspond to sliding the same leaf. Hence,
			\[
			v_j^{\sigma_2}=v_{j+1}^{\sigma_1},\qquad v_{j+1}^{\sigma_2}=v_j^{\sigma_1},\qquad v_m^{\sigma_2}=v_m^{\sigma_1}\quad\text{for }m\neq j,j+1.
			\]
			This is precisely the claimed formula in the case of one adjacent transposition. Since any two permutations are connected by a sequence of adjacent transpositions, the formula follows inductively for arbitrary $\sigma_1$ and $\sigma_2$.

			It follows that, for a fixed coordinate direction $\ell\in[k]$, the leaf slid when a geodesic $\sigma$ traverses the direction $\ell$ is
			\[
			v_{\sigma^{-1}(\ell)}^\sigma=v_\ell^{\operatorname{Id}_{[k]}},
			\]
			and hence is independent of the chosen geodesic $\sigma$. We denote this leaf by $v_\ell$. Consequently, for every $S\subseteq[k]$ and every $i\notin S$, the edge from $N_S$ to $N_{S\cup\{i\}}$ is obtained by sliding the leaf $v_i$.

			We now show that the edge along which $v_i$ is slid also only depends on the coordinate direction $i$. Let $T_0=N_\emptyset$. For each $i\in[k]$, the edge from $N_\emptyset$ to $N_{\{i\}}$ is obtained by sliding the leaf $v_i$. Define
			\[
			u_i=p_{T_0}(v_i),\qquad w_i=p_{N_{\{i\}}}(v_i).
			\]
			Thus the edge from $N_\emptyset$ to $N_{\{i\}}$ is obtained by sliding the leaf $v_i$ along the edge $(u_i,w_i)$.

			We claim that for every $S\subseteq[k]$ and every $i\notin S$, the edge from $N_S$ to $N_{S\cup\{i\}}$ is obtained by sliding the leaf $v_i$ along the same edge $(u_i,w_i)$. We prove this by induction on $|S|$. For $S=\emptyset$ this is the definition of $u_i$ and $w_i$. Now, let $S\neq\emptyset$, choose $j\in S$, and set $S'=S\setminus\{j\}$. Consider the square face with vertices
			\[
			N_{S'},\qquad N_{S'\cup\{i\}},\qquad N_{S'\cup\{j\}},\qquad N_{S'\cup\{i,j\}}.
			\]
			By the induction hypothesis, the edge from $N_{S'}$ to $N_{S'\cup\{i\}}$ is obtained by sliding $v_i$ along $(u_i,w_i)$. By \autoref{lem:FourCycles}, the image of this square is the boundary of a unique $2$-dimensional network realignment. In such a $2$-cell, opposite edges correspond to the same leaf slide. Therefore, the opposite edge from $N_{S'\cup\{j\}}=N_S$ to $N_{S'\cup\{i,j\}}=N_{S\cup\{i\}}$ is also obtained by sliding $v_i$ along $(u_i,w_i)$. This proves the claim.

			We next construct the desired $k$-cell. We claim that the data
			\[
			B=\{v_1,\ldots,v_k\},\qquad T=T_0\setminus B,\qquad r(v_i)=(u_i,w_i)
			\]
			defines a $k$-dimensional network realignment. Applying \autoref{lem:FourCycles} to the square with vertices $N_\emptyset, N_{\{i\}},N_{\{i,j\}},N_{\{j\}}$ shows that $v_i\neq v_j$ for $i\neq j$. The same $2$-dimensional realignment also shows that neither $v_i$ nor $v_j$ is an endpoint of the edge along which the other leaf is slid. Hence, $v_i\notin \{u_j,w_j\}$ for $i\neq j$. Since the $v_i$ are pairwise distinct leaves of $T_0$, deleting them from $T_0$ leaves a tree $T=T_0\setminus B$. Moreover, each edge $(u_i,w_i)$ lies in this tree. Since the base graph is $K_n$, the above data defines a cell $N$ of $X_n$.

			By construction, the realignment $N_S$ is obtained from $N$ by specifying $v_i$ to be $w_i$ if $i\in S$ and $u_i$ if $u\notin S$. Hence, $N$ extends the embedded copy of $(I^k)_1$. Since $X_n$ is a cubical complex, this extension is automatically unique.
		\end{proof}

		\begin{proposition} \label{prop:AutomorphismReduction}
			$\Aut(X_n) \cong \Aut(\mathcal{G}_n)$.
		\end{proposition}

		\begin{proof}
			Clearly, any automorphism $\tilde{f}:X_n\to X_n$ restricts to an automorphism $f:\mathcal{G}_n\to \mathcal{G}_n$. On the other hand, any automorphism $f: \mathcal{G}_n\to \mathcal{G}_n$ can be extended uniquely to an automorphism $\tilde{f}:X_n\to X_n$ by applying \autoref{lem:FlagComplex} to the diagram 
			\begin{center}
				\begin{tikzcd}
					(I^k)_1 \arrow[rr, hook] \arrow[d, hook] && I^k \arrow[d, dashed, "\exists!"] \\
					\mathcal{G}_n \arrow[r, "f"] & \mathcal{G}_n \arrow[r, hook] & X_n
				\end{tikzcd}
			\end{center}
			for each cube $I^k \subseteq X_n$ with $k\geq 2$. We denote this extension by $\tilde{f}|_{I^k}$. If $I^{k-1}$ is a facet of $I^k\subseteq X_n$, then $f|_{(I^{k-1})_1} = f|_{(I^k)_1}$. Since the extension of the above diagram is unique, it follows that
			\[
				\tilde{f}|_{I^{k-1}} = (\tilde{f}|_{I^k})|_{I^{k-1}}.
			\]
			This property extends to all faces, and hence to any intersection of cells.
		\end{proof}

	\subsection{A Rigid Subcomplex} \label{sec:RigidSubcomplex}

		Next, we identify an $\Aut(\mathcal{G}_n)$-invariant subgraph $H\subseteq \mathcal{G}_n$ whose automorphism group is the symmetric group $S_n$. The key idea is that any automorphism of $X_n$ must preserve the dimension of the cubes, and thus restricts to an automorphism on $\Sigma_n = \Sigma_{K_n}$, which contains exactly the $(n-2)$-cubes $C_{i,j}$. The only vertices of a cube $C_{i,j}$ that lie in another $(n-2)$-cube are the star trees $s_i$ and $s_j$, which occur at opposite corners. The automorphisms of a $C_{i,j}$ fixing these two vertices form a group isomorphic to $S_{n-2}$, and such automorphisms may be chosen independently for each cube $C_{i,j}$. Together with the natural action of $S_n$, it follows that
		\[
			\Aut(\Sigma_n) = S_{n-2}^{\binom{n}{2}} \rtimes S_n.
		\]
		Hence, to obtain a rigid subgraph $H$, we must include additional cells linking distinct $(n-2)$-cubes. 

		\begin{definition}
			Let $i\in [n]$. Let
			\[
				H_i = \{ T\in V(\mathcal{G}_n) \mid d_{\mathcal{G}_n} (T,s_i) \leq 2 \}
			\]
			denote the induced subgraph of all network realignments of distance at most 2 from the star tree $s_i$, and let
			\[
				H = (\Sigma_n)_1 \cup \{ T\in V(\mathcal{G}_n) \mid \exists i\in [n]: d_{\mathcal{G}_n} (T,s_i) \leq 2 \}.
			\]
		\end{definition}

		Since automorphisms of $\mathcal{G}_n$ permute the star trees and preserve graph distances, the subgraph $H$ is $\Aut(\mathcal{G}_n)$-invariant. Hence, the restriction
		\[
			\Aut(\mathcal{G}_n) \to \Aut(H)
		\]
		is well-defined.

		\begin{definition}
			Let $T$ be a vertex of $\mathcal{G}_n$, and $v$ a leaf of $T$. Let $u\neq v$ be adjacent to $p_T(v)$. Define
			\[
				T^{v\rightsquigarrow u} = (T^{v\to (p_T(v),u)})^{v\to u}.
			\]
		\end{definition}

		\begin{lemma} \label{lem:FixStarTrees}
			Let $n\geq 5$, and $f: H\to H$ be an automorphism. If $f$ fixes each star tree, then $f = id_H$.
		\end{lemma}

		\begin{proof}
			If $f$ fixes each star tree, then $f(C_{i,j}) = C_{i,j}$  and $f(H_i) = H_i$. Suppose $f\neq id_H$. 

			In this case, there always exist a vertex $T$ with $f(T) \neq T$. We claim that such a vertex can be chosen to lie in $\Sigma_n$. Let $\tilde{T}$ be a tree such that $f(\tilde{T}) \neq \tilde{T}$. If $\tilde{T}\in \Sigma_n$, set $T = \tilde{T}$. Otherwise $\tilde{T}\in H\setminus \Sigma_n$. Then $\tilde{T}$ contains a single path $v_1, j_1, i, j_2, v_2$ of length 4 and every other vertex is a leaf of $\tilde{T}$ with parent $i$. Further, $\tilde{T}^{v_1\rightsquigarrow i}\in C_{i,j_2}$ and $\tilde{T}^{v_2\rightsquigarrow i}\in C_{i, j_1}$. Note that there exist exactly two geodesics between $\tilde{T}^{v_1\rightsquigarrow i}$ and $\tilde{T}^{v_2\rightsquigarrow i}$, namely
			\[
				\tilde{T}^{v_1\rightsquigarrow i}, s_i, \tilde{T}^{v_2\rightsquigarrow i} \quad \text{and} \quad \tilde{T}^{v_1\rightsquigarrow i}, \tilde{T}, \tilde{T}^{v_2\rightsquigarrow i}.
			\]
			Since $f(s_i) = s_i$ and $f(\tilde{T}) \neq \tilde{T}$, we can not have both
			\[
				f(\tilde{T}^{v_1\rightsquigarrow i}) = \tilde{T}^{v_1\rightsquigarrow i} \quad \text{and} \quad f(\tilde{T}^{v_2\rightsquigarrow i}) = \tilde{T}^{v_2\rightsquigarrow i}.
			\]
			Hence, there exists $T\in \Sigma_n$ with $f(T) \neq T$.

			Without loss of generality, assume $T\in C_{1,2}$. Since $\Aut(\Sigma_n) = S_{n-2}^{\binom{n}{2}} \rtimes S_n$, it follows that $f|_{C_{1,2}}$ is induced by a permutation $\sigma$ of the vertices $3, \ldots, n$. Thus, $T$ can be chosen to satisfy $d_{\mathcal{G}_n} (T, s_1) = 1$. Then there exists a single leaf $k$ with $p_T(k) = 2$, and all other vertices are adjacent to $1$. Since $f(T) \neq T$ and $f(C_{1,2}) = C_{1,2}$, we have $\sigma(k) \neq k$. 
			
			Let $j\in [n]\setminus \{1,2,k,\sigma(k)\}$ and consider the path
			\[
				T, T_1 = T^{j\rightsquigarrow \sigma(k)}, T_2 = T_1^{k\rightsquigarrow 1}.
			\]
			\autoref{fig:IdentityContradiction} shows the network realignments $T$ and $T_2$ and their images under $f$. We have $d_{\mathcal{G}_n} (T, T_2) = 2$ and $T_2 \in C_{1,\sigma(k)} \cap H_1$. It follows that $f(T_2) \in C_{1,\sigma(k)}\cap H_1$. Clearly, the only path of length 2 from $f(T)$ to $f(T_2)$ passes through the star tree $s_1$. Therefore, $f(T_1) = s_1 = f(s_1)$, contradicting that $f$ is an automorphism. Hence, $f = id_H$.
		\end{proof}

		\begin{figure}[tbp]
			\centering
			\begin{tikzpicture}[dot/.style={circle, draw, fill, inner sep=0pt, minimum size=4pt}]

				\newcommand{\diagramScale}{0.7}

				\begin{scope}[scale=\diagramScale]
					% vertices
					\node[dot, label={below:{\scriptsize $1$}}] (u1) at (0,0) {};
					\node[dot, label={below:{\scriptsize $2$}}] (u2) at (1,0) {};
					\node[dot, label={right:{\scriptsize $k$}}] (u3) at (1.4,1) {};
					\node[dot, label={left:{\scriptsize $\sigma(k)$}}] (u4) at (-1,-.4) {};
					\node[dot, label={left:{\scriptsize $j$}}] (u5) at (-1,.4) {};
					\node[dot] (u6) at (-.4,1) {};
					\node[dot] (u7) at (.4,1) {};
					\node at (0,1) {$\cdot\!\cdot\!\cdot$};
					\node at (.5,-1.2) {$T$};
					% arrows
					\draw (u1) -- (u2) -- (u3);
					\draw (u4) -- (u1) -- (u5);
					\draw (u6) -- (u1) -- (u7);
				\end{scope}

				\begin{scope}[scale=\diagramScale, xshift = 5cm]
					% vertices
					\node[dot, label={below:{\scriptsize $1$}}] (v1) at (0,0) {};
					\node[dot, label={below:{\scriptsize $2$}}] (v2) at (1,0) {};
					\node[dot, label={right:{\scriptsize $k$}}] (v3) at (.8,1) {};
					\node[dot, label={below:{\scriptsize $\sigma(k)$}}] (v4) at (-1,0) {};
					\node[dot, label={left:{\scriptsize $j$}}] (v5) at (-1.6,1) {};
					\node[dot] (v6) at (-.8,1) {};
					\node[dot] (v7) at (0,1) {};
					\node at (-.4,1) {$\cdot\!\cdot\!\cdot$};
					\node at (0,-1.2) {$T_2$};
					% arrows
					\draw (v5) -- (v4) -- (v1) -- (v2);
					\draw (v1) -- (v3);
					\draw (v6) -- (v1) -- (v7);
				\end{scope}

				\begin{scope}[scale=\diagramScale, xshift = 8.5cm]
					% vertices
					\node[dot, label={below:{\scriptsize $1$}}] (w1) at (0,0) {};
					\node[dot, label={below:{\scriptsize $2$}}] (w2) at (1,0) {};
					%\node[dot, label={left:{\scriptsize $k$}}] (w3) at (-1,-.4) {};
					\node[dot, label={right:{\scriptsize $\sigma(k)$}}] (w4) at (1.4,1) {};
					%\node[dot, label={left:{\scriptsize $j$}}] (w5) at (-1,.4) {};
					\node[dot] (w6) at (-.4,1) {};
					\node[dot] (w7) at (.4,1) {};
					\node at (0,1) {$\cdot\!\cdot\!\cdot$};
					\node at (0.5,-1.2) {$f(T)$};
					% arrows
					\draw (w1) -- (w2) -- (w4);
					%\draw (w3) -- (w1) -- (w5);
					\draw (w6) -- (w1) -- (w7);
				\end{scope}

				\begin{scope}[scale=\diagramScale, xshift = 13.5cm]
					% vertices
					\node[dot, label={below:{\scriptsize $1$}}] (x1) at (0,0) {};
					\node[dot, label={below:{\scriptsize $2$}}] (x2) at (1,0) {};
					\node[dot] (x3) at (-1.4,1) {};
					\node[dot, label={below:{\scriptsize $\sigma(k)$}}] (x4) at (-1,0) {};
					\node[dot] (x5) at (-.4,1) {};
					\node[dot] (x6) at (.4,1) {};
					\node at (0,1) {$\cdot\!\cdot\!\cdot$};
					\node at (0,-1.2) {$f(T_2)$};
					% arrows
					\draw (x3) -- (x4) -- (x1) -- (x2);
					\draw (x5) -- (x1) -- (x6);
				\end{scope}
			\end{tikzpicture}
			\vspace{-3mm}
			\caption{The network realignments used in the proof of \autoref{lem:FixStarTrees}. While there are two paths of length 2 between $T$ and $T_2$, there is a single path of length 2 between $f(T)$ and $f(T_2)$. This path passes through the star tree $s_1$.}
			\label{fig:IdentityContradiction}
		\end{figure}
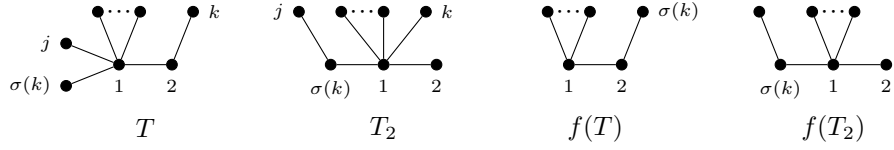

		\begin{proposition} \label{prop:RigidSubcomplex}
			$\Aut(H) \cong S_n$ for $n\geq 5$.
		\end{proposition}

		\begin{proof}
			Clearly, the map
			\[
				\Aut(H) \to \Aut \mathcal{S}_{K_n} \cong S_n, \quad f\mapsto f|_{\mathcal{S}_{K_n}}.
			\]
			is surjective. Injectivity follows from \autoref{lem:FixStarTrees}.
		\end{proof}

	\subsection{Determining Automorphisms from Restrictions}

		It remains to show that the map
		\[
			\Aut(\mathcal{G}_n) \to \Aut(H)
		\]
		given by restriction is injective. To this end, we again make use of 4-cycles. As observed earlier, every 4-cycle in $\mathcal{G}_n$ is uniquely determined by any three of its vertices. Hence, if an automorphism fixes three vertices of a 4-cycle, then it must also fix the fourth.

		\begin{lemma} \label{lem:FourCycleExtension}
			Let $f: \mathcal{G}_n \to \mathcal{G}_n$ be an automorphism. Let $C = T_1,T_2,T_3,T_4$ be a 4-cycle in $\mathcal{G}_n$ such that $f(T_i) = T_i$ for $1\leq i\leq 3$, then $f(T_4) = T_4$. 
		\end{lemma}

		\begin{proof}
			Any automorphism must map $C$ to a 4-cycle. By \autoref{lem:FourCycles}, the only 4-cycle that contains $T_1,T_2,T_3$ is $C$. Hence, $f(C) = C$ and $f(T_4) = T_4$. 
		\end{proof}

		More generally, automorphisms of $G_n$ are determined by their values on any subgraph from which all remaining vertices can be recovered by repeatedly adjoining the fourth vertex of a 4-cycle whenever the other three vertices are already present. This motivates the following closure property for subgraphs.

		\begin{definition} \label{def:FourCycleCompletion}
			Let $H \subseteq \mathcal{G}_n$ be a subgraph. We say that $H\subseteq \mathcal{G}_n$ is \textit{closed under 4-cycle completion} if for every 4-cycle $C$ in $\mathcal{G}_n$ 
			\begin{equation}
				|V(C) \cap V(H)| \geq 3 \text{ implies } C\subseteq H.
			\end{equation}
		\end{definition}

		From the following proposition we can deduce that $H$ is such an aforementioned subgraph whose values fully determine an automorphism of $G_n$. This will allow us to directly determine the automorphism group $\Aut(G_n)$.

		\begin{proposition} \label{prop:FourCycleCompletion}
			Let $H$ be a subgraph of $\mathcal{G}_n$ that is closed under 4-cycle completion. If $\Sigma_n \subseteq H$, then $H = \mathcal{G}_n$.
		\end{proposition}

		We begin by introducing additional tools needed for the proof.

		\begin{lemma} \label{lem:MaxComponentCentroid}
			Let $T$ be a tree, and let $c_1, c_2$ be two distinct centroids of $T$. Then $F_{c_1}(T) = F_{c_2}(T)$. 
		\end{lemma}

		\begin{proof}
			By \autoref{existenceofcentral}, $(c_1, c_2) \in E(T)$. Write $T\setminus \{(c_1,c_2)\} = A_{(c_1,c_2)} \cup B_{(c_1,c_2)}$ with $c_1 \in V(A_{(c_1,c_2)})$ and $c_2 \in V(B_{(c_1,c_2)})$. Then $|V(B_{(c_1,c_2)})| \leq \frac{n}{2}$, since it is a connected component of $T\setminus c_1$ and $c_1$ is a centroid. Also $|V(A_{(c_1,c_2)})| \leq \frac{n}{2}$, since it is a connected component of $T\setminus c_2$ and $c_2$ is a centroid. Further,
			\[
				|V(A_{(c_1,c_2)})| + |V(B_{(c_1,c_2)})| = |V(T)| = n.
			\] 
			Hence, $|V(A_{(c_1,c_2)})| = \frac{n}{2} = |V(B_{(c_1,c_2)})|$, and the equality follows.
		\end{proof}

		For a tree $T$, define $F_{\mathrm{cent}}(T) = F_c(T)$ for any centroid $c$ of $T$. By \autoref{lem:MaxComponentCentroid}, this is well-defined.

		\begin{lemma} \label{lem:MaxComponentLeafSlide}
			Let $T$ be a tree and $c$ a centroid of $T$. Let $v$ be a leaf of $T$ with $p_T(v) \neq c$. Denote by $u_v$ the vertex adjacent to $p_T(v)$ along the unique path from $p_T(v)$ to $c$. Then
			\[
				F_\mathrm{cent}(T) - 1 \leq F_\mathrm{cent}(T^{v\rightsquigarrow u_v}) \leq F_\mathrm{cent}(T).
			\]
			Moreover, $F_\mathrm{cent}(T^{v\rightsquigarrow u_v}) = F_\mathrm{cent}(T)$, whenever $u_v \neq c$.
		\end{lemma}

		\begin{proof}
			If $u_v \neq c$, then $p_T(v)$ and $u_v$ lie in the same connected component of $T\setminus \{c\}$. Hence, $v$ remains in its connected component and the sizes of the connected components do not change, since $T\setminus \{v\} = T^{v\rightsquigarrow u_v}\setminus \{v\}$. Thus, $F_c(T) = F_c(T^{v\rightsquigarrow u_v})$ and  $c$ is a centroid of $T^{v\rightsquigarrow u_v}$.

			If $u_v = c$, let $C$ be the connected component of $T\setminus \{c\}$ containing $v$. Then $C\setminus \{v \}$ and $\{v\}$ are connected components of $T^{v\rightsquigarrow c}\setminus \{c\}$. All other connected components remain unchanged. Hence, $F_c(T)-1 \leq F_c(T^{v\rightsquigarrow u_v}) \leq F_c(T)$, and thus $c$ is a centroid of $T^{v\rightsquigarrow u_v}$. 
		\end{proof}

		Now, we are ready to prove \autoref{prop:FourCycleCompletion}. The strategy is as follows. Given a proper subgraph $H\subset \mathcal{G}_n$ with $\Sigma_n \subseteq H$, we seek a vertex $T\notin H$ minimising a suitable quantity and a 4-cycle containing $T$ whose other three vertices take strictly lower values in this quantity. Such a 4-cycle would show that $H$ is not closed under 4-cycle completion. 

		A natural candidate is the quantity $\Psi$. However, for the tree $T$ shown in \autoref{fig:FourCycleCompletion}, no such 4-cycle exists. In this case, the quantity $\Psi - F_\mathrm{cent}$ allows us to find a different tree $\tilde{T}$ with $\Psi(\tilde{T}) - F_\mathrm{cent}(\tilde{T}) = \Psi(T) - F_\mathrm{cent}(T)$ together with a 4-cycle containing $\tilde{T}$ whose other three vertices take strictly lower values in $\Psi - F_\mathrm{cent}$. We combine these two quantities in the following.

		\begin{proof} [Proof of \autoref{prop:FourCycleCompletion}]
			It suffices to show that for any proper subgraph $H\subset \mathcal{G}_n$ with $\Sigma_n \subseteq H$ there exists a 4-cycle $C$ with $|V(C) \cap V(H)| = 3$. This proves that $H$ is not closed under 4-cycle completion.

			Let $T\in \mathcal{G}_n\setminus H$ that minimises $\Psi(T) - F_\mathrm{cent}(T)$, and among all such trees, minimises $\Psi(T)$. 

			\textbf{Case 1:} $T$ has a centroid $c$ and two leaves $v,w$ with $p_T(v) \neq c \neq p_T(w)$.\\
			Let $u_v$ and $u_w$ denote the vertices adjacent to $p_T(v)$ and $p_T(w)$ on the unique paths from $p_T(v)$ to $c$ and from $p_T(w)$ to $c$, respectively. Set
			\[
				T_1 = T^{v\rightsquigarrow u_v}, \quad T_3 = T^{w\rightsquigarrow u_w}, \quad \text{and} \quad T_2 = T_1^{w\rightsquigarrow u_w} = T_3^{v\rightsquigarrow u_v}.
			\]
			By Lemmas \ref{iter} and \ref{lem:MaxComponentLeafSlide}, $\Psi(T_i) - F_\mathrm{cent}(T_i) \leq \Psi(T) - F_\mathrm{cent}(T)$ and $\Psi(T_i) < \Psi(T)$ for $1\leq i \leq 3$. Hence, $T_1, T_2, T_3\in H$. Therefore, the 4-cycle $C = T, T_1, T_2, T_3, T$ satisfies $|V(C) \cap V(H)| = 3$.

			\textbf{Case 2:} For every centroid $c$ of $T$, at most one leaf $v$ satisfies $p_T(v) \neq c$.\\
			Suppose $T$ has two centroids. Then $T$ has at most two leaves each connected to a centroid. Thus, $|V(T)| \leq 4$, contradicting $T\notin \Sigma_n\subseteq H$. 
			
			Let $c_1$ denote the single centroid of $T$. Since $T\notin \Sigma_n$, there exists a leaf $v$ such that $d_T(v,c_1) \geq 3$. Let $c_2$ denote the vertex adjacent to $c_1$ along the unique path from $c_1$ to $v$. Let $c_2, v_1,\ldots, v_k$ denote the neighbours of $c_1$. Set
			\[
				T_0 = T \quad \text{and} \quad T_i = T_{i-1}^{v_i\rightsquigarrow c_2} \quad \text{for } 1\leq i\leq k.
			\]
			\autoref{fig:FourCycleCompletion} shows the trees $T, T_i$ and $T_k$ of the above sequence for a path from $c_1$ to $v$ of length 3. 
			% Figure depicts the involved trees in Case 2.2
			%-----------------------------------------------------------------
			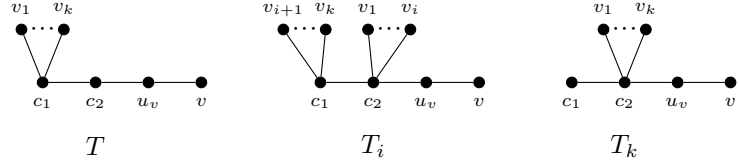
\begin{figure}[tbp]
				\centering
				\begin{tikzpicture}[dot/.style={circle, draw, fill, inner sep=0pt, minimum size=4pt}]

					\newcommand{\diagramScale}{0.7}

					\begin{scope}[scale=\diagramScale]
						% vertices
						\node[dot, label={below:{\scriptsize $c_1$}}] (u1) at (0,0) {};
						\node[dot, label={below:{\scriptsize $c_2$}}] (u2) at (1,0) {};
						\node[dot, label={below:{\scriptsize $u_v$}}] (u3) at (2,0) {};
						\node[dot, label={below:{\scriptsize $v$}}] (u4) at (3,0) {};
						\node[dot, label={above:{\scriptsize $v_1$}}] (u5) at (-.4,1) {};
						\node[dot, label={above:{\scriptsize $v_k$}}] (u6) at (.4,1) {};
						\node at (0,1) {$\cdot\!\cdot\!\cdot$};
						\node at (1,-1.2) {$T$};
						% arrows
						\draw (u1) -- (u2) -- (u3) -- (u4);
						\draw (u5) -- (u1) -- (u6);
					\end{scope}

					\begin{scope}[scale=\diagramScale, xshift = 5.25cm]
						% vertices
						\node[dot, label={below:{\scriptsize $c_1$}}] (v1) at (0,0) {};
						\node[dot, label={below:{\scriptsize $c_2$}}] (v2) at (1,0) {};
						\node[dot, label={below:{\scriptsize $u_v$}}] (v3) at (2,0) {};
						\node[dot, label={below:{\scriptsize $v$}}] (v4) at (3,0) {};
						\node[dot, label={[yshift=-.8pt] above:{\scriptsize $v_{i+1}$}}] (v5) at (-.7,1) {};
						\node[dot, label={above:{\scriptsize $v_k$}}] (v6) at (.1,1) {};
						\node[dot, label={above:{\scriptsize $v_1$}}] (v7) at (.9,1) {};
						\node[dot, label={above:{\scriptsize $v_i$}}] (v8) at (1.7,1) {};
						\node at (-.3,1) {$\cdot\!\cdot\!\cdot$};
						\node at (1.3,1) {$\cdot\!\cdot\!\cdot$};
						\node at (1,-1.2) {$T_i$};
						% arrows
						\draw (v1) -- (v2) -- (v3) -- (v4);
						\draw (v5) -- (v1) -- (v6);
						\draw (v7) -- (v2) -- (v8);
					\end{scope}

					\begin{scope}[scale=\diagramScale, xshift = 10cm]
						% vertices
						\node[dot, label={below:{\scriptsize $c_1$}}] (w1) at (0,0) {};
						\node[dot, label={below:{\scriptsize $c_2$}}] (w2) at (1,0) {};
						\node[dot, label={below:{\scriptsize $u_v$}}] (w3) at (2,0) {};
						\node[dot, label={below:{\scriptsize $v$}}] (w4) at (3,0) {};
						\node[dot, label={above:{\scriptsize $v_1$}}] (w5) at (.6,1) {};
						\node[dot, label={above:{\scriptsize $v_k$}}] (w6) at (1.4,1) {};
						\node at (1,1) {$\cdot\!\cdot\!\cdot$};
						\node at (1,-1.2) {$T_k$};
						% arrows
						\draw (w1) -- (w2) -- (w3) -- (w4);
						\draw (w5) -- (w2) -- (w6);
					\end{scope}
				\end{tikzpicture}
				\vspace{-3mm}
				\caption{If the tree $T$ minimises the quantity $\Psi$, then there is no 4-cycle containing $T$ whose other three vertices take strictly lower value in $\Psi$. In this situation, we can construct a sequence $T,T_1,\ldots, T_k$ to find a different tree $T_i$ and a desired 4-cycle. Along this sequence, $\Psi - F_\mathrm{cent}$ never increases.}
				\label{fig:FourCycleCompletion}
			\end{figure}
			%-----------------------------------------------------------------
			
			We claim that $\Psi(T_{i+1}) - F_\mathrm{cent} (T_{i+1}) \leq \Psi(T_i) - F_\mathrm{cent} (T_i)$ for $0\leq i \leq k-1$. By construction, the maximal connected component of $T_i\setminus \{c_1\}$ always contains all vertices except $c_1$ and $v_{i+1},\ldots, v_k$. Hence,
			\[
				F_{c_1} (T_i) = n-k+i-1.
			\]
			Meanwhile, the maximal connected component of $T_i\setminus c_2$ contains either $c_1$ together with $v_{i+1},\ldots, v_k$ or the unique path from $c_2$ to $v$ with $c_2$ removed. Therefore, 
			\[
				F_{c_2} (T_i) = \max \{ k-i+1, \, n-k-2 \}
			\]
			for all $i$. Further, for every $w \neq c_1,c_2$, the sizes of the connected components of $T_i\setminus \{w\}$ and $T\setminus \{w\}$ agree. Since $w$ is not a centroid of $T$, it is not a centroid of $T_i$ either. Hence, we distinguish the following three cases.

			If $c_1$ is a centroid of both $T_i$ and $T_{i+1}$, then 
			\[
				F_{c_1} (T_{i+1}) = n - k + i = F_{c_1} (T_i) + 1,
			\]
			and by \autoref{iter} we have
			\[
				\Psi(T_{i+1}) = \Phi_{c_1}(T_{i+1}) = \Phi_{c_1} (T_i) + 1 = \Psi(T_i) + 1.
			\]
			Therefore, 
			\[
				\Psi(T_{i+1}) - F_\mathrm{cent} (T_{i+1}) = \Psi(T_i) - F_\mathrm{cent} (T_i).
			\]
			If $c_2$ is a centroid of both $T_i$ and $T_{i+1}$, then 
			\[
				\Psi(T_{i+1}) = \Phi_{c_2}(T_{i+1}) = \Phi_{c_2} (T_i) - 1 = \Psi(T_i) - 1
			\]
			by \autoref{iter}. Further, the value of $F_\mathrm{cent}$ changes along a leaf slide by at most one. Hence, 
			\[
				\Psi(T_{i+1}) - F_\mathrm{cent} (T_{i+1}) \leq \Psi(T_i) - F_\mathrm{cent} (T_i).
			\]
			It remains to assume that $c_1$ is only a centroid of $T_i$ and $c_2$ is only a centroid of $T_{i+1}$. Since $c_1$ is not a centroid of $T_{i+1}$, it follows that $\frac{n}{2} < F_{c_1}(T_{i+1}) = n-k+i$. Further, $\frac{n}{2} < F_{c_2}(T_i) = k-i+1$, since $c_2$ is not a centroid of $T_i$ and $n-k-2 \leq \frac{n}{2}$, as $c_2$ is not a centroid of $T$. Combining these two inequalities yields
			\[
				\frac{n}{2} - 1 < k-i < \frac{n}{2}.
			\]
			It follows that $n$ is odd and $k-i = \frac{n-1}{2}$. Hence, $F_{c_1}(T_i) = \frac{n-1}{2} = F_{c_2}(T_{i+1})$ and $\Psi(T_i)-\Psi(T_{i+1}) = \sigma_{T_i}(c_1,c_2) - \sigma_{T_{i+1}}(c_1,c_2) = 0$. Therefore,
			\[
				\Psi(T_{i+1}) - F_\mathrm{cent} (T_{i+1}) = \Psi(T_i) - F_\mathrm{cent} (T_i).
			\]

			Now, it suffices to show that $\Psi(T) > \Psi(T_k)$ to prove that $T_k\in H$. Since each leaf $v_1,\ldots,v_k$ was realigned along the edge $e = (c_1, c_2)$, we have $\sigma_T (\tilde{e}) = \sigma_{T_k} (\tilde{e})$ for all edges $\tilde{e}\neq e$. Hence,
			\[
				\Psi(T) - \Psi(T_k) = \sigma_T (e) - \sigma_{T_k} (e) = d_T(v,c_1) - 1 > 0,
			\]
			Therefore, since $T\notin H$ and $T_k\in H$, there must exist $0\leq i \leq k-1$ such that $T_i\notin H$ and $T_{i+1}\in H$. Let $u_v$ denote the vertex adjacent to $p_T(v)$ on the unique path from $p_T(v)$ to $c_1$. Set
			\[
				S_i = T_i^{v\rightsquigarrow u_v} \quad \text{and} \quad S_{i+1} = T_{i+1}^{v\rightsquigarrow u_v}.
			\]
			If $u_v = c_2$, then $S_i$ and $S_{i+1}$ are weak double star trees, and hence $S_i, S_{i+1} \in \Sigma_n \subseteq H$. Otherwise, $u_v$ is not a centroid, and thus we get 
			\[
				\Psi(T) - F_\mathrm{cent} (T) \geq \Psi(T_j) - F_\mathrm{cent} (T_j) > \Psi(S_j) - F_\mathrm{cent} (S_j)
			\] 
			for $j = i, i+1$ using Lemmas \ref{iter} and \ref{lem:MaxComponentLeafSlide}. Again, the trees $S_i, S_{i+1}$ are contained in $H$. Therefore, $C = T_i, S_i, S_{i+1}, T_{i+1},, T_i$ is a 4-cycle with $|V(C) \cap V(H)| = 3$.
		\end{proof}

		\begin{theorem}
			$\Aut(X_n) \cong S_n$ for $n\geq 5$.
		\end{theorem}

		\begin{proof}
			By \autoref{prop:AutomorphismReduction}, it is enough to show that $\Aut(\mathcal{G}_n) \cong S_n$. Consider the following sequence of maps:
			\begin{center}
				\begin{tikzcd}
					S_n \arrow[r]
						& \Aut(\mathcal{G}_n) \arrow[r]
						& \Aut(H) \arrow[r]
						& \Aut(\Sigma_n)
				\end{tikzcd}
			\end{center}
			Along the first horizontal map, an element of $S_n$ is mapped to the induced automorphism of $\mathcal{G}_n$, which relabels the vertices of each network realignment. The second and third horizontal maps are given by restriction. 
			
			By \autoref{prop:RigidSubcomplex}, the composition of the first two maps is an isomorphism. Injectivity of $\Aut(\mathcal{G}_n) \to \Aut(\Sigma_n)$ follows by combining \autoref{lem:FourCycleExtension} and \autoref{prop:FourCycleCompletion}. Therefore, the map $\Aut(\mathcal{G}_n) \to \Aut(H)$ is bijective. By \autoref{prop:RigidSubcomplex}, $\Aut (H) \cong S_n$, and the result follows.
		\end{proof}

    \newpage

	\newpage 

	% richtiges Literaturverzeichnis
	\printbibliography

	\newpage
\appendix
\section{Equivariant Generalised Morse Theory}\label{AppA}

Now, we prove the main theorem of generalised Morse theory. The key step is to show that an interval $[\sigma, \tau]$ collapses onto its complement in the cube $\tau$. By ordering the nontrivial intervals $\Lambda$-equivariantly, we can inductively perform these collapses to obtain a $\Lambda$-equivariant strong deformation retraction of the space onto the subcomplex formed by the critical cells.

\begin{definition}
	Let $\tau$ be an $n$-cube, $\sigma$ a proper face of $\tau$. Define 
	\begin{equation*}
		\mathcal{W}_{\tau}(\sigma) = \vert\tau\vert \setminus \bigcup\limits_{\sigma\leq\mu\leq\tau} 		\textup{Int}(\mu).
	\end{equation*}
\end{definition}

Note that $\mathcal{W}_{\tau}(\sigma)$ contains all faces of the cube $\tau$ that do not contain the face $\sigma$. Hence, $\mathcal{W}_{\tau}(\sigma)\subseteq \partial\tau$.

\begin{proposition} \label{prop:EquivariantDeformationRetract}
	Let $\tau$ be an $n$-cube, $\sigma$ a proper $k$-dimensional face of $\tau$, and let $\Lambda$ denote the subgroup of $\Aut(\tau)$ that fixes $\sigma$, i.e. $g\sigma = \sigma$ for $g\in \Lambda$. Then, $\mathcal{W}_{\tau}(\sigma)$ is a $\Lambda$-equivariant strong deformation retract of the cube $\tau$.
\end{proposition}

\begin{proof}
	We identify $\vert\tau\vert$ with $[-1,1]^n$ such that 
	\begin{equation*}
		\sigma = \{(v_1,\dots,v_k,1,\dots,1) \mid v_1,\dots,v_k\in [-1,1]\}.
	\end{equation*}
	Then, the barycenter of $\sigma$ is $b_{\sigma}=(\underbrace{0,\dots,0}_{k-times},1,\dots,1)$. Choose some $\varepsilon>0$. We define $y=(1+\varepsilon)b_{\sigma}$. For each point $x\in\vert\tau\vert$, we denote the line going through $y$ and $x$ by
	\begin{equation*}
		L_x = \{tx+(1-t)y\mid t\in \mathbb{R}\},
	\end{equation*}
	and by $r_x$ the unique point in the intersection of $L_x$ and $\mathcal{W}_{\tau}(\sigma)$. Now, we can define the following homotopy
	\[
		H\colon\vert\tau\vert\times [0,1] \rightarrow \vert\tau\vert, \quad (x,t) \mapsto (1-t)x + tr_x,
	\]
	where $H(x,0)=x$ and $H(x,1)=r_x\in\mathcal{W}_{\tau}(\sigma)$. If $x\in \mathcal{W}_{\tau}(\sigma)$, then $r_x = x$, and thus $H(x,t) = x$. This proves that $\mathcal{W}_{\tau}(\sigma)$ is a strong deformation retract of $\vert\tau\vert$.

	It remains to show that $H$ is $\Lambda$-equivariant. For $g\in \Lambda$, we have $gy=y$, since $g\sigma=\sigma$ implies $gb_{\sigma}=b_{\sigma}$, and hence
	\begin{equation}\label{G-equivLine}
		L_{gx} = \{tgx+(1-t)y\mid t\in\mathbb{R}\} =gL_x.
	\end{equation}
	Now, $g$ preserves both $\partial\tau$ and Int($\sigma$), and hence also $\mathcal{W}_{\tau}(\sigma)$, which implies
	\begin{equation*}
		gr_x = gL_x\cap \mathcal{W}_{\tau}(\sigma) \overset{(\ref{G-equivLine})}{=} L_{gx}\cap \mathcal{W}_{\tau}(\sigma) = r_{gx}.
	\end{equation*}
	This proves that
	\[
		H(gx,t) = (1-t)gx+tr_{gx} = (1-t)gx+tgr_{x} = gH(x,t),
	\]
	concluding the proof.
\end{proof}

Given a generalised Morse matching on a cubical $\Lambda$-complex $X$, let $X/\sim$ denote the set of equivalence classes of $\sim$, equipped with the relation  given by the transitive closure of
\[
	[\sigma_1, \tau_1] \leq [\sigma_2, \tau_2] \quad \text{if} \quad \sigma_1 \leq \tau_2.
\]
Since $\sim$ is acyclic, $X/\sim$ is a poset. If $\sim$ is $\Lambda$-equivariant, then $\Lambda$ acts naturally on $X/\sim$. We analyse this action in the next lemma.

\begin{lemma} \label{lem:horizontal}
	Let $X$ be a cubical $\Lambda$-complex, and let $\sim$ be a generalised Morse matching that is $\Lambda$-equivariant. Then the induced $\Lambda$-action on $X/\sim$ is horizontal, i.e. $[\sigma, \tau]\leq g[\sigma,\tau]$ implies $[\sigma, \tau] = g[\sigma,\tau]$ for all $[\sigma,\tau]$ in $\sim$ and $g\in G$.
\end{lemma}

\begin{proof}
	Any order-preserving group action of a group $\Lambda$ on a finite poset $P$ is horizontal. Indeed, suppose there exist $p\in P$ and $g\in \Lambda$ such that $p < gp$. Since the action is order-preserving, it follows inductively that
	\[
		p < gp < \ldots < g^ip < g^{i+1}p < \ldots,
	\]
	which contradicts that $P$ is finite. In particular, the induced $\Lambda$-action on $X/\sim$ is horizontal.
\end{proof}

\GeneralizedMorse*

\begin{proof}
	Let $\bigl(X\setminus U\bigr)/\sim$ denote the poset of non-critical equivalence classes of $\sim$. Since $\sim$ is $\Lambda$-equivariant, the set of critical cells is $\Lambda$-invariant, and hence $\Lambda$ acts on $\bigl(X\setminus U\bigr)/\sim$. By \autoref{lem:horizontal}, the induced $\Lambda$-action on $\bigl(X\setminus U\bigr)/\sim$ is  horizontal, and thus $\Lambda\backslash\bigl(X\setminus U\bigr)/\sim$ is a poset \cite{Babson05}.
	
	Let $\Lambda[\sigma_1, \tau_1] < \ldots < \Lambda[\sigma_k, \tau_k]$ be a linear extension of $\Lambda \backslash \bigl(X\setminus U\bigr)/\sim$. We proceed by induction on $k$. The case $k = 0$ is immediate, since then $X=U$. For $k>0$, set 
	\[
		\widetilde{X} = X\setminus \Bigl(\bigcup_{\sigma\in \Lambda[\sigma_k,\tau_k]} \Int(\sigma) \Bigr).
	\]
	Note that $\mathcal{W}_{g\tau_k}(g\sigma_k) \subset \tilde{X}$ for each $g\in \Lambda$. Since $\sim$ is $\Lambda$-equivariant, any group element that fixes $\tau_k$ also fixes $\sigma_k$. Hence, \autoref{prop:EquivariantDeformationRetract} can be applied simultaneously to all intervals in the orbit $\Lambda[\sigma_k, \tau_k]$. Choosing the identifications of the cells $g\tau_k$ with $[-1,1]^n$ $\Lambda$-equivariantly ensures that the resulting strong deformation retractions are compatible with the $\Lambda$-action. It follows that there exists a $\Lambda$-equivariant strong deformation retraction of $X$ onto $\widetilde{X}$. Further, by induction, $U$ is a $\Lambda$-equivariant strong deformation retract of $\widetilde{X}$. Combining these two strong deformation retractions yields a $\Lambda$-equivariant strong deformation retraction of $X$ onto $U$.
\end{proof}

\end{document}